\documentclass[a4paper,11pt]{article}

\usepackage{amsfonts}
\usepackage{amsmath,amssymb}
\usepackage[dvipsnames]{xcolor}
\usepackage{graphicx} 
\usepackage{caption} 
\usepackage{subcaption} 

\usepackage{authblk}
\usepackage{physics}
\usepackage{amsfonts}
\usepackage{mathtools}
\usepackage{latexsym}
\usepackage{amsthm}
\usepackage{amssymb}
\usepackage{amsmath}
\usepackage{eufrak}
\usepackage{mathrsfs}
\usepackage{graphicx}   
\usepackage{geometry}
\usepackage{graphics}
\usepackage{comment}
\usepackage[scr=esstix,cal=boondox]{mathalfa} 
\usepackage[all]{xy}
\usepackage{rotating}
\usepackage{algorithmic}
\usepackage{algorithm}
\usepackage{mathtools}
\usepackage{tikz}
\usetikzlibrary{positioning}
\usepackage{caption}
\usepackage{amsthm} 
\usepackage{amsfonts}
\usepackage{latexsym}
\usepackage{amsthm}
\usepackage{amsmath}
\usepackage{amssymb}
\usepackage{mathcomp}
\usepackage{amsmath}
\usepackage{xcolor}
\usepackage{eufrak}
\usepackage{calrsfs}
\usepackage{caption}
\usepackage{mathrsfs}
\usepackage{enumerate} 
\usepackage{pifont}
\usepackage{geometry}
\usepackage{tikz}
\usetikzlibrary{decorations.pathreplacing,decorations.markings}
\usetikzlibrary{arrows.meta}
 \usepackage{pgfplots}
 \usepackage{enumitem}
 \pgfplotsset{compat=newest}  
\usepackage{color, colortbl}
\usepackage{lipsum}

\newcommand\blfootnote[1]{%
  \begingroup
  \renewcommand\thefootnote{}\footnote{#1}%
  \addtocounter{footnote}{-1}%
  \endgroup
}
\newtheorem{theo}{Theorem}
\newtheorem*{THM}{Theorem}
\newtheorem{definition}{Definition}
\newtheorem{Prop}{Proposition}[section]
\newtheorem{Cor}{Corollary}[section]
\newtheorem{Lemma}{Lemma}[section]

\newtheorem{Rm}{Remark}

\newtheorem*{Ex}{Example}

\usepackage{color, colortbl}

\title{ A group from a map and orbit equivalence.}
\author [1] {J\'er\^ome Los}
\affil [1]{\small{Aix-Marseille Universit\'e, CNRS, ECM, I2M, UMR 7373.}}
\author [2] {Natalia A. Viana Bedoya}
\affil [2] {\small{Universidade Federal de S\~ao Carlos, DM-UFSCar.}}
\date{}
\begin{document}
\maketitle
\begin{abstract}
\noindent  \em{In two papers published in 1979, R. Bowen and C. Series defined a dynamical system from a Fuchsian group, acting on the hyperbolic plane $\mathbb{H}^2$. The dynamic is given by a map on $S^1$ which is,
in particular, a discontinuous expanding piecewise homeomorphism of the circle.
In this paper we consider a reverse question: which dynamical conditions for a discontinuous expanding piecewise homeomorphism of $S^1$
are sufficient for the map to be a ``Bowen-Series-type" map (see below) for some group $G$ and which groups can occur? We give a partial answer to these questions.}
 \end{abstract}

 \section{Introduction}
 \blfootnote{{\it AMS Subject Classification (2020):} 20F65, 20F67, 37E10.}
\noindent In this paper we introduce a class of {\em discontinuous expanding piecewise homeomorphisms of the circle}. Such a map 
$ \Phi : S^1 \rightarrow S^1$ is given by a finite partition of the circle so that the restriction of $\Phi$ to each partition interval is an expanding homeomorphism onto its image. The class of maps we consider is motivated by two related questions:\\
- Can we construct a group $G_{\Phi}$ from such a map $\Phi$?\\
- Which groups can be obtained?

The groups that can possibly be constructed are naturally subgroups of $\textrm{Homeo}(S^1)$ which is well known to have many different classes of subgroups.
Since the possible groups $G_{\Phi}$ and the map $\Phi$ act on the same space, $S^1$, it is natural to compare the two actions and the best possible situation is when the two actions are ``orbit equivalent". This means that the orbits of $\Phi$ and of $G_{\Phi}$ are the 
same, modulo possibly finitely many exceptions. In such cases we say that the map $\Phi$ is a {\em Bowen-Series-type} map for the group $G_{\Phi}$.

This program is a reverse problem of a beautiful construction initiated by R. Bowen and C. Series in the late 70's in \cite{B} and \cite{BS}, where they discovered a striking relationship between some groups and some dynamics.
The Bowen-Series construction starts with a Fuchsian group $G$ given by an action on $\mathbb{H}^2$ with specific properties and they obtained a particular map $ \Phi_{BS} : S^1 \rightarrow S^1$, where $S^1$ is the boundary $\partial \mathbb{H}^2$.
Some variations of this construction have been studied by Adler and Flato in \cite{AF} and more recently in \cite{AKU}.
The maps $ \Phi_{BS}$ satisfy very strong properties:\\
- The maps are piecewise M\"obius, in particular piecewise analytic,\\
- each $ \Phi_{BS}$ is orbit equivalent to the $G$-action on $S^1$,\\
- each $ \Phi_{BS}$ is an expanding Markov map.

The idea of the Bowen-Series construction has been revisited in \cite{Los} for hyperbolic surface groups, in the geometric group theory context. The group is given abstractly by a presentation $P = \langle X; R\rangle$ i.e., a set of generators and relations. The presentations belong to a particular class called ``geometric", meaning that the associated Cayley 2-complex is planar. The classical presentations of surface groups are geometric in this sense as well as the presentations considered in \cite{AF}.
This construction starts with a geometric presentation $P$ and defines a map $ \Phi_{P} : S^1 \rightarrow S^1$ that is an expanding  piecewise homeomorphism and the circle is the Gromov boundary of the group $G$ (see \cite{Gr}).

\vspace*{2pt}
\noindent The maps $ \Phi_{P}$ and  $\Phi_{BS}$ are different i.e. non-conjugate, even in the cases they can be compared, for the classical presentations. But they satisfy the same two main features:

{\it The orbit equivalence and the Markov properties.}

\vspace*{2pt}
\noindent The map $ \Phi_{P}$ satisfies an additional property relating the group and the dynamics:

{\it The volume entropy of  $P$ (see \cite{GH}) equals the topological entropy of $ \Phi_{P}$ (see \cite{AKM}).}

\vspace*{3pt}
\noindent The construction of more maps from any geometric presentation $P$ of a surface group has been given in \cite{AJLM2} by defining a multiparameter family of maps and an entropy stability result has been obtained for all maps in this family.

\vspace*{3pt}
The problem we consider in this paper is a converse question:\\
{\em How particular are the maps obtained from a surface group presentation among discontinous piecewise homeomorphisms of the circle?}\\
We obtain a partial answer to this general question. Here a map $\Phi$ is given, as a discontinuous piecewise homeomorphism of the circle, one goal is to find dynamical conditions on $\Phi$ that allow us first to construct a group from the map and then to analyse which groups could be obtained. Each map $\Phi$ is given by a finite partition of $S^1$, it will soon become clear that the number of partition intervals has to be even. A point at the boundary of two partition intervals is called a {\em cutting point}, at such points the map is not continuous. The map is expanding means that each partition interval is mapped onto an interval that contains it compactly thus the map is surjective and not globally injective.

 
  The conditions we found on the map $\Phi$ are explained in \S 2, they can be expressed roughly as:

 
\noindent $\bullet$  An Eventual Coincidence condition (EC): the left and right orbits of each cutting point coincide after some well defined iterate.
 
\noindent $\bullet$  The conditions (E+) and (E-) that control the left and right orbits of the cutting points before the coincidence. These two conditions imply a Strong Expansivity condition (SE): each partition interval is mapped to an interval that contains it and intersects all but one partition interval.

Finally we do not restrict to maps $\Phi$ satisfying a Markov property, as in \cite{BS}, \cite{AF}, \cite{Los}, \cite{AJLM}, which would be too restrictive. We replace it by a weaker condition which quantifies the expansivity property:

\noindent $\bullet$ The Constant Slope condition (CS-$\lambda$): the map is conjugate to a piecewise affine map with constant slope $\lambda > 1$.\\
 Under this set of conditions our main result is:
  \textcolor{black}{
  \begin{THM}\label{first}
  Let $\Phi: S^1 \rightarrow S^1$ be an orientation preserving discontinuous piecewise homeomorphism satisfying the conditions: $\rm  (EC), (E+), (E-), (CS{\textrm{-}}\lambda)$, for some $\lambda>1$.
 Then there exists a discrete subgroup $G_{\Phi}$ of ${\rm Homeo}^+ (S^1)$ such that: 
 \vspace*{-3pt}
   \begin{enumerate} [noitemsep, leftmargin=15pt]
 \item $G_{\Phi}$ and $\Phi$ are orbit equivalent.
 \item $G_{\Phi}$ is conjugate in ${\rm Homeo}({S}^1)$ to the restriction of a torsion-free Fuchsian group \textcolor{black}{action on $S^1$}.
 \item Each $g \in G_{\Phi}$ is a piecewise affine homeomorphism with slopes in $\{\lambda^k, k \in \mathbb{Z}\}$ and $\lambda$ is an algebraic integer.
 \end{enumerate}
  \end{THM}
  }

 The set of maps satisfying the above conditions is not empty. Indeed if a surface group, for an orientable surface of genus larger than 2, has a geometric presentation $P$, where all the relations have even length (for instance the classical presentations) then the map $\Phi_{P}$ of \cite{Los} satisfies the conditions of the Theorem. For the same set of presentations, the multiparameter family 
 $\Phi_{P, \Theta}$ defined in [AJLM2], satisfies the conditions of the Theorem for an open set of parameters (see Lemma \ref{condSatisfied}) and the numbers $\lambda$ are very specific.
 On the other hand the set of conditions of the Theorem is not optimal (see Remark \ref{nonempty}) in order to obtain a classification of the Bowen-Series-type maps.
 
  The strategy of proof has several steps. The first one is to analyse the dynamical properties of the map 
 $\Phi$ (see \S 2 and \S 3). Then we construct a group $G_{X_{\Phi}}$, as a subgroup of ${\rm Homeo}^+ (S^1)$, by producing a generating set 
 $X_{\Phi}$ from the map $\Phi$ (see \S 3). 
 This step exhibits a surprising generating set $X_{\Phi}$.
 
  The next step is to prove that the group  $G_{X_{\Phi}}$, as an abstract group, is hyperbolic in the sense of M. Gromov (see \cite{Gr} or \cite{GdlH}). 
This property is obtained by showing that $G_{X_{\Phi}}$ acts geometrically on a hyperbolic metric space.
 This is a technical part (see \S 4 and \S 5). It requires the construction of a hyperbolic space and a geometric action on it, from the only data we have: the map.
 
 The hyperbolic space is obtained by a general dynamical construction inspired by one due to P. Haissinsky and K. Pilgrim \cite{HP} (see \S 4). The hyperbolicity is a consequence of the expansivity, as in \cite{HP}, and the Gromov boundary of the space is $S^1$.
  We adapt the construction and define a new space, suited to the maps $\Phi$, specially to the conditions (EC) and 
  ($\textrm{E}\pm$), in order to define a group action on the space.
  
  This step is new, it defines a class of ``dynamical spaces" in the context of groups.
    The construction of an action of the group on this metric space is also new. In both cases, the space and the action are defined only from the dynamic of the map (see \S 5).
    
  At this point the group $G_{\Phi}$ is hyperbolic with boundary $S^1$.
   A result of E. Freden  \cite{F} implies that the group is a discrete convergence group, as defined by F. Gehring and G. Martin \cite{GM} and thus it satisfies the conditions of the geometrisation theorem of P. Tukia \cite{Tukia}, D. Gabai  \cite{G} and A. Casson-D. Jungreis \cite{CJ}. \textcolor{black}{ The conclusion is that  
  $G_{\Phi}$ is conjugate to the restriction of a Fuchsian group action on $S^1$.}
 One more step shows that, with our assumptions, $G_{\Phi}$ is torsion-free and, by  H. Zieschang \cite{Zi}, the Fuchsian group is a surface group.
 
 Proving that the group  $G_{\Phi}$ and the map $\Phi$ are orbit equivalent follows a strategy similar to \cite{BS}
 (see \S 6) and is, in fact, simpler thanks to some particular dynamical properties of the map.
 
 In the appendix (see \S 7), we give a direct proof that $G_{\Phi}$ is abstractly a surface group, without using the geometrisation theorems of Tukia, Gabai and Casson-Jungreis. 
 All the hard work has been done before: the geometric action constructed in \S 5 is extended to a free, co-compact action on $\mathbb{R}^2$. This also gives some interesting consequences, in relation with \cite{AJLM2}.

\textcolor{black}{
We obtain a partial answer to our general question, the main result is a sort of reciprocal to the Bowen-Series-like construction. The answer is only partial, as explained in Remark \ref{nonempty}, and finding better conditions is a challenge for future works.
}
 The construction of the group from the map leads to some surprises. For instance, as a by-product of our construction, we obtain:
 \textcolor{black}{
 \begin{THM}\label{surprise}
 Let $S$ be a closed, compact, orientable surface of genus larger than 2. There exists a discrete faithful representation $\rho: \pi_1(S)\rightarrow {\rm Homeo}^+ (S^1)$ and a metric $\mu$ on $S^1$ such that $G:=\rho(\pi_1(S))$ satisfies:
 \begin{enumerate} [noitemsep, leftmargin=15pt]
 \item   $G$ admits a presentation where the generators are piecewise affine homeomorphisms, for the metric $\mu$ of $S^1$, with slopes in 
 $\{ \lambda, \lambda^{-1} \}$ for an algebraic integer $\lambda >1$.
 \item  $log (\lambda)$ is the volume entropy of the presentation. 
 \item Each element $g\in G$ is piecewise affine with slopes in $\{\lambda^k : k \in \mathbb{Z}\}$, if $g$ has length $n$, with respect to the generating set, it admits an interval $I_g \subset S^1$ so that $g_{| I_g}$ is affine of slope $\lambda^n$.
   \end{enumerate}
  \end{THM}
 }
\noindent
\textcolor{black}{
The group $G:=\rho(\pi_1(S))$ of the above theorem is the group $G_{\Phi}$ obtained in the proof of the main theorem for some map $\Phi$. By the main theorem, $G_{\Phi}$
is topologically conjugate to $G_F$, the restriction of a Fuchsian group on $S^1$.
There is thus a topological conjugacy between $G$ in ${\rm Homeo}^+ (S^1)$ and $G_F$ in ${\rm Diff}^{ w }(S^1)$.
The existence and differentiability of a conjugacy for two representations of surface groups in some 
${\rm Diff}^{ k}(S^1)$ is a question that has been considered in many works, for instance by Matsumoto \cite{Ma} in class $C^0$ and by Ghys \cite{Gh} in class $C^k$, for $k\geq 3$. Here the conjugacy cannot be better than $C^0$ since $G$ is in ${\rm Homeo}^+ (S^1)$ and the elements are not $C^1$.
The group $G_{\Phi}$ we obtain from the map $\Phi$ is rigid, uniqueness comes from a limit argument. We could expect that 
 some variations of our construction, for instance without a limit, is more flexible and leads to different classes of groups in 
 ${\rm Diff} (S^1)$ from a given map. Which groups could be obtained is an interesting question.
  } 
The condition (EC) is central in our approach, it seems to be a new dynamical condition and is interesting in its own right. The class of discontinuous maps satisfying a condition (EC) is much larger than the one studied here.\\


{\bf Acknowledgements}: {\small This work was partially supported by FAPESP (2016/24707-4). We thanks the hospitality of our research institutes: I2M in Marseille and DM-UFSCar in S\~ao Carlos, Brazil. We thanks the French-Uruguayan laboratory (IFUMI)
in Montevideo where this work has been completed. 
 The authors would like to thank P. Haissinsky and M. Boileau for their interest in this work and some comments on 
 an earlier versions of the paper.
 We would like to thank the anonymous referee of a previous version who pointed out a gap in one argument. Closing this gap has proved to be very interesting. We would like to thank the reviewer of this version for her/his careful reading and observations, which have helped us to greatly improve this work.}

 \section{A class of piecewise homeomorphisms on $S^1$}\label{sec 2}
 
We define in this section the class of maps that will be considered throughout the paper.\\ 
     A map  $\Phi: S^1 \rightarrow S^1$ is a {\it piecewise orientation preserving homeomorphism of the circle} if  
 there is a finite partition of $S^1$:
 \begin{equation}\label{theMap}
  S^1 = \bigcup_{j=1}^{M}  I_j \textrm{, where each } I_j \textrm{ is half open,} 
 \end{equation} 
 \textcolor{black}{so that }  $\Phi_j :=   \Phi_{|_{I_{j}}}$ is an orientation preserving homeomorphism 
 onto its image and each $I_j$ is maximal. We require further that the number of partition intervals is even: $M = 2N$.

 \subsection{ The class of maps}\label{TheClassPhi}
 To state the next properties of the maps in our class, we introduce some notations.\\
 Let $\zeta,\iota, \delta, \gamma$ be permutations of $\{1,\dots,2N\}$, such that:
 \begin{itemize} [noitemsep, leftmargin=10pt]
\item $\zeta$ is a cyclic permutation of order $2N$,
\item $\iota$  is a  fixed point free involution i.e., for all $j \in \{1,\dots,2N\}$, $\iota(j)\neq j$ and 
$\iota^2=\textrm{id}$,\\ such that:
   $\iota(j) \neq \zeta^{ \pm 1}(j), \forall j \in \{1,\dots,2N\}$.\\
      This implies that $N>1$ and, to avoid special cases, we assume for the rest of the paper, that $N \geq 4$.
   \item From the permutations $\zeta$ and $\iota$ we define: $\gamma:=\zeta^{-1} \iota$  and  $\delta:=\zeta\;  \iota $.
\end{itemize}
\noindent Geometrically $\zeta$ is the permutation that realizes the adjacency permutation of the intervals 
$\{ I_{1}, \dots, I_{{2N}} \}$ along a given  positive (clockwise) orientation  of $S^1$. By convention $I_{{\zeta(j)}}$ is the interval that is adjacent to $I_{j}$ in the positive direction.\\
The interval $I_{{\iota(j)}}$ is an interval that is not $I_{j}$ and is not adjacent to $I_{j}$. The two intervals 
$I_{{\gamma (j)}}$ and $I_{{\delta (j)}}$ are the intervals adjacent to $I{_{\iota(j)}}$ (see Figure \ref{SE}).

\noindent From now on we assume that all the cycles of $\gamma$ (and $\delta$ see Lemma \ref{conj} \textcolor{black}{below)}, in its cycle decomposition, have   even length  and greater or equal to $4$, i.e., \textcolor{black}{if $\ell[j]$ denotes the length of the $\delta$ cycle containing $j$ then:}
\begin{eqnarray}\label{even-cycles}
\ell[j] \textrm{ is even and }  k(j)=\ell[j]/2  \geq 2 \textrm{, for all } j \in \{1, \dots, 2N\}.
\end{eqnarray}

 Using the permutations above, the map $\Phi : S^1 \rightarrow S^1$ satisfies the following set of conditions:
 the {\it Strong Expansivity} condition if:\\
\noindent {\bf(SE)}\centerline{$ \forall j \in \{1, \dots, 2N\}$, $ \Phi(I_{j}) \bigcap  I_{k} = \emptyset \Leftrightarrow  k = \iota (j)$, (see Figure \ref{SE}).}

\vspace{5pt}
\noindent This condition has some immediate consequences:\\
(I) $ \Phi(I_j) \bigcap  I_k = I_k $, $\forall k \neq \iota (j), \gamma(j), \delta (j)$,\\
(II) The map $\Phi$ has a fixed point in the interior of each $I_j$.\\
This is immediate from the definition of $\iota$ and (I), since 
$I_j \subset \Phi(I_j)$.\\
(III) The map is surjective, non-injective and each point $z \in S^1$ has $2N-1$ or $2N-2$ pre-images.

\vspace{2pt}
  To fix the notations we write each interval  $I_j := [z_j, z_{\zeta(j)} )$, the points 
$z_j \in S^1$ are called the {\it cutting points} of $\Phi$. The map $\Phi$ is not continuous at each $z_j$.
\begin{figure}[h]\label{SE}
\begin{center}
\resizebox{0.4\textwidth}{!}{
\begin{tikzpicture}[scale=0.57]
\draw [ black] circle [radius=2.8];
\draw [ black] circle [radius=4.8];
 \draw[line width=1.2mm, gray]
  (-.9,2.6)
  arc[start angle=110,end angle=66,radius=2.8]
  node[midway]{{$ $}};
\draw[line width=1.2mm, gray]
  (-2.5,-4.2)
  arc[start angle=240,end angle=-80,radius=4.8]
  node[midway]{{$$}};

 \draw[line width=1.5mm, white]
  (-2.5,-4.2)
  arc[start angle=240,end angle=280,radius=4.8]
  node[midway]{{$$}};

 \draw[dashed,gray,very thin,thick,-] (-2.45,-4) -- (-1.55,-2.3);
 
 \draw [decorate,color=black] (-1.55,-4.8)
   node[left] {\begin{footnotesize}$\Phi_j({z_{_j}})$\end{footnotesize}};

\draw[dashed,gray,very thin,thick,-] (.75,-4.75) -- (0.37,-2.7);

\draw [decorate,color=black] (2.5,-5.35)
   node[left] {\begin{footnotesize}$\Phi_j({z_{_{\zeta(j)}}})$\end{footnotesize}};
   
\draw [decorate,color=black] (-.5,2.65)
   node[left] {$\boldsymbol{\cdot}$};
   
   \draw[dashed,gray,very thin,thick,-] (-.9,2.65) -- (-1.6,4.5);
   
   \draw [decorate,color=black] (-.3,2.2)
   node[left] {\begin{footnotesize}${z_{_j}}$\end{footnotesize}};
   
    \draw [decorate,color=black] (.6,3.3)
   node[left] {\begin{footnotesize}$I_{_j}$\end{footnotesize}};
   
   \draw [decorate,color=black] (1.5,2.55)
   node[left] {$\boldsymbol{\cdot}$};
   
    \draw[dashed,gray,very thin,thick,-] (1.17,2.55) -- (2,4.2);

    \draw [decorate,color=black] (1.9,2.2)
   node[left] {\begin{footnotesize}$z_{_{\zeta(j)}}$\end{footnotesize}};

    \draw [decorate,color=black] (-2,-1.5)
   node[left] {$\boldsymbol{\cdot}$};
   
    \draw[dashed,gray,very thin,thick,-] (-2.5,-1.6) -- (-4.1,-2.5);

    \draw [decorate,color=black] (-1.5,-2.2)
   node[left] {\begin{footnotesize}$I_{_{\delta(j)}}$\end{footnotesize}};

    \draw [decorate,color=black] (-1,-2.45)
   node[left] {$\boldsymbol{\cdot}$};
   
   \draw [decorate,color=black] (0.2,-3.2)
   node[left] {\begin{footnotesize}$I_{_{\iota(j)}}$\end{footnotesize}};
   
     \draw [decorate,color=black] (0.5,-2.8)
   node[left] {$\boldsymbol{\cdot}$};
   
     \draw [decorate,color=black] (2.2,-3.1)
   node[left] {\begin{footnotesize}$I_{_{\gamma(j)}}$\end{footnotesize}};

     \draw [decorate,color=black] (2,-2.3)
   node[left] {$\boldsymbol{\cdot}$};
   
    \draw[dashed,gray,very thin,thick,-] (1.75,-2.4) -- (3,-3.65);
    
     \draw [decorate,color=black] (7,0)
   node[left] {\begin{footnotesize}$\Phi_j(I_{_{j}})$\end{footnotesize}};
   \end{tikzpicture}
   }
\caption{Condition (SE)}\label{SE}
\end{center}
\end{figure}
The next condition makes the map $\Phi$ really particular, it is called the {\it Eventual Coincidence} condition:\\
\noindent {\bf(EC)}\hspace{.5cm}{  $\forall j \in \{1, \dots, 2N\}$ and $\forall n \geq k(j) - 1$, where  $k(j) \geq 2$ is given by (\ref{even-cycles}):\\
 \centerline{ $ \Phi^n ( \Phi _{\zeta^{-1}(j) }(z_j) ) = \Phi^n (\Phi_j (z_j)).$ \textcolor{black}{ Let  $\mathscr{z}_j^{k(j)}:=\widetilde{\Phi}^{k(j)-1}(\widetilde{\Phi}_{\zeta^{-1}(j)}(\widetilde{z}_j)) =
 \widetilde{\Phi}^{k(j)-1}(\widetilde{\Phi}_{j}(\widetilde{z}_j)). $}}

\vspace{5pt}
\noindent In other words, each cutting point has a priori two different orbits, one from the positive side and one from the negative side of the point. The condition (EC) says that after $k(j)$ iterates these two orbits coincide.
By (\ref{theMap}) each $I_j := [z_j, z_{\zeta(j)} )$ is half open, the notation 
$\Phi _{\zeta^{-1}(j) }(z_j)$ is well defined by continuity at the left of $z_j$ of 
$\Phi _{\zeta^{-1}(j) }$.

The next set of conditions on the map gives some control on the first $k(j)-1$ iterates of the cutting points $z_j$, namely: 
 For all $ j \in \{1,\dots, 2N \}$ and all $ 0\leq m \leq  k(j) - 2$:\\
{\bf(E+)}  \centerline{$ \Phi^m ( \Phi_j (z_j))  \in I_{_{\delta^{m+1} (j)}} $, }\\
{\bf(E-)}  \centerline{$ \Phi^m ( \Phi_{\zeta^{-1}(j) }(z_j))  \in I_{_{\gamma^{m+1} (\zeta^{-1}(j)) }}$. }

\vspace{5pt}
\noindent   These two conditions are interpreted as follows:\\
   Consider (E+), for $m=0$  this is condition (SE) since 
  $\Phi_j (z_j)  \in I_{\delta (j)} $, see Figure \ref{SE}. Then $\Phi_j (z_j)$ is near the cutting point $z_{\delta (j)} $ in $I_{\delta (j)} $, since  $ \Phi ( \Phi_j (z_j))  \in I_{\delta^{2} (j)}$ (by $m=1$) and 
  $I_{\delta^{2} (j)}$ is the interval containing $\Phi_{\delta(j )}(z_{\delta (j)})$ (by $m=0$ for $z_{\delta (j)}$ ) and so on up to $m = k(j) - 2$.  As observed above the conditions (E+) and (E-) imply condition (SE).

 The last condition quantifies the expansivity property of the map, it is called the {\it Constant Slope} condition:
 
\noindent {\bf(CS-$\lambda$)} $\Phi$ is topologically conjugate, \textcolor{black}{by $g \in \textrm{Homeo}^+(S^1)$},  to a piecewise affine map $\widetilde{\Phi}$ with constant slope $\lambda > 1$. \textcolor{black}{ Most of the time the constant $\lambda $ will be implicit and removed from the notations.}\\
(II') In complement to (II): $\Phi$ has a unique expanding fixed point on each $I_j$.\\
The following result is a combination of several statements in \cite{AJLM2} (see Theorem A and Lemma 5.1). It implies that the set of piecewise homeomorphisms of the circle satisfying the conditions (SE), (EC), (E-), (E+), (CS-$\lambda$) is non-empty. 

\begin{Lemma}\label{condSatisfied}
Let $S$ be a closed compact orientable surface of negative Euler characteristic, and let $P = \langle X; R \rangle$ be a geometric presentation
of the fundamental group $G = \pi_1 (S)$ so that all the relations in $R$ have even length, for instance the classical presentation. Then, in the Bowen-Series-Like family of maps $\Phi_{P,\Theta}$ defined in \cite{AJLM2}, there is an open set of parameters $\Theta$  so that the corresponding maps satisfy the conditions $\rm (SE), (EC), (E\textrm{-}), (E+),(CS\textrm{-}\lambda)$. The parameters $\Theta$ belong to a product of $2N$ intervals, where 2N is the number of generators
\textcolor{black}{ and  
$\lambda$ is an algebraic integer.}
\end{Lemma}

\begin{Rm}\label{nonempty}
By the previous result, the set of piecewise orientation preserving homeomorphisms satisfying the conditions $\rm (SE), (EC), (E\textrm{-}), (E+),(CS\textrm{-}\lambda)$ is non-empty and there is a family of such maps for each orientable surface and each geometric presentation with even length relations. For the maps $\Phi_{P}$ constructed in \cite{Los}, the proof of these properties is a direct check. In particular the constant slope condition 
$\rm (CS\textrm{-}\lambda)$ is obtained using the Markov property satisfied by $\Phi_{P}$ via a standard Perron-Frobenius argument. 
In the more general cases of the family 
$\Phi_{P,\Theta}$ defined in \cite{AJLM2}, the constant slope condition is one statement of the main theorem of that paper. \textcolor{black}{ In these cases the number $\lambda$ is an algebraic integer and $\log (\lambda)$ is the volume entropy of the presentation.
}

 If a presentation $P$ of a surface group $G$ is geometric and has some relations with odd length then the constructions in \cite{Los} and  \cite{AJLM2} apply but not those in \cite{B} and \cite{BS}. For these presentations, some conditions similar but different  from $\rm(E+)$ and $\rm(E-)$ are satisfied. When the presentation $P$ has some relations of length 3, a condition weaker than $\rm(SE)$ is satisfied (see Lemma 5.2 in \cite{Los}). In all these cases a condition $\rm (CS\textrm{-}\lambda)$ is satisfied, for particular $\lambda$, and a condition similar to $\rm (EC)$ is satisfied for some integers $k$. The condition $\rm (EC)$ is crucial in this paper and is not satisfied by all possible maps constructed via the general 
 Bowen-Series-Like strategy as in \cite{AJLM2}. In particular it is not satisfied by the original map in \cite{BS}.
 The set of conditions considered in this paper is thus non-optimal to obtain a complete answer to our general question. 
\end{Rm}

 \subsection{ Elementary properties of the permutations $\delta$ and $\gamma$}\label{combinatorics}
 
\noindent The combinatorics of our class of maps is mainly encoded via the permutations $\delta$ and $\gamma$. For the rest of the work we need to understand, in particular, the cycle structure of these permutations. These cycles will appear everywhere. 
In this paragraph we point out some elementary  properties of these  permutations.
\begin{Lemma}\label{conj}
 The permutations $\gamma$ and $\delta$ are conjugate, more precisely 
$\gamma=\iota^{-1} \delta^{-1} \iota.$
\end{Lemma}
\begin{proof}
Since $\delta$ and $\delta^{-1}$ are conjugate and 
$ \iota^{-1} \delta^{-1} \iota=\iota(\iota  \zeta^{-1})\iota= \zeta^{-1} \iota=\gamma,$
 then $\delta$ and $\gamma$ are conjugate.
\end{proof}

 To simplify the notations we will sometimes use:
 $\overline{j}:=\iota(j)$.

\begin{Rm}\label{lj}  The two permutations $\gamma$ and $\delta$ have the same cycle structure. We obtain 
  $\gamma$ from $\delta^{-1}$ by changing $j$ to $\overline{j}$ on its cycles. 
  The cycle of $\gamma$ that contains $\overline{j}$ and the cycle of $\delta$ that contains $j$ have the same length. We denote this number by $  \ell[j]$.
\end{Rm}
 \begin{Lemma}\label{same-cycle}
The integers $\zeta^{-1}(j)$,  $\overline{j}$ and $\overline{\delta^{m-1}(j)}$ belong to the same cycle of $\gamma$ of length 
$\ell[j]$, for all $j \in \{ 1,\dots, 2N \}$ and $0<m \leq \ell[j]$. 
\end{Lemma}
\begin{proof} From the definitions of $\iota,\;\gamma,\; \delta$ and Lemma \ref{conj},
we have:
$ \gamma(\overline{j})=\zeta^{-1}\iota (\iota (j))=\zeta^{-1}(j)$ and
$ \overline{\delta^{m-1}(j)}=\iota (\delta^{m}\delta^{-1}(j))=(\iota \delta^{m}\iota^{-1})\zeta^{-1}(j)=\gamma^{-m}(\zeta^{-1}(j)).$
\end{proof}

\begin{Lemma} \label{adj}
 If $1\leq m \leq \ell[\overline{j}],$ then
$\zeta (\gamma^m(j))=\overline{\gamma^{m-1}(j)}$. In particular 
 if $\ell[\overline{j}]$ is even and  $k(\overline{j})= \ell[\overline{j}]/2$ then
$\zeta (\overline{\delta^{k(\overline{j})-1}\zeta(j)})=\overline{\gamma^{k(\overline{j})-1}(j)}.$
\end{Lemma}
\begin{proof}
 Notice that
$\zeta (\gamma^m(j))=\zeta \gamma (\gamma^{m-1}(j))=\zeta (\zeta^{-1}\iota)(\gamma^{m-1}(j))=\overline{\gamma^{m-1}(j)}$,
and suppose that  $\ell[\overline{j}]$ is even and let $k(\overline{j})= \ell[\overline{j}]/2$. 
From the first part of this Lemma, to obtain   $\zeta (\overline{\delta^{k(\overline{j})-1}\zeta(j)})=\overline{\gamma^{k(\overline{j})-1}(j)}$,  it is enough  to show that 
$\overline{\delta^{k(\overline{j})-1}\zeta(j)}=\gamma^{k(\overline{j})}(j)$. In fact, by Lemma \ref{conj} and the definition of $\delta$ we have:
$\gamma^{k(\overline{j})}(j)=\iota^{-1} \delta^{-k(\overline{j})}\iota(j)=\overline{\delta^{k(\overline{j})}\iota(j)}=\overline{\delta^{k(\overline{j})-1}\delta \iota (j)}=\overline{\delta^{k(\overline{j})-1}\zeta(j)}$.
\end{proof}
\begin{Lemma}\label{gamma-delta}
$\gamma(\overline{\delta^m(j)})=\overline{\delta^{m-1}(j)}$ and  $\delta(\overline{\gamma^m(\zeta^{-1}(j))})=\overline{\gamma^{m-1}(\zeta^{-1}(j))}$,  for $m=1,\dots,\ell[\overline{j}]$.
\end{Lemma}
\begin{proof}
In fact, by Lemma \ref{adj} and  $\iota,\gamma, \delta$, we have:
$\gamma(\overline{\delta^m(j)})=\zeta^{-1}\iota(\iota \delta^m(j))=\zeta^{-1}\delta(\delta^{m-1}(j))
=\zeta^{-1}\zeta(\overline{\delta^{m-1}(j)})=\overline{\delta^{m-1}(j)},$
$\textrm{and }\;\delta(\overline{\gamma^m(\zeta^{-1}(j))})=\zeta \iota (\iota \gamma^{m}(\zeta^{-1}(j)))=
\zeta \gamma (\gamma^{m-1}(\zeta^{-1}(j)))=\overline{\gamma^{m-1}(\zeta^{-1}(j)))}.$
\end{proof}

\begin{Lemma}\label{alpha-beta}  If $\ell[{j}]$ is even and  $k({j})= \ell[{j}]/2$ then the relations $\alpha(j):= \gamma^{k(j)-1}(\zeta^{-1}(j))$ and  $\beta(j):= \delta^{k(\zeta(j))-1}(\zeta(j))$ are permutations on $\{1,\dots,2N\}$, satisfying \textcolor{black}{$\beta=\iota^{-1}\alpha \iota$ and $\beta=\alpha^{-1}$. }
\end{Lemma}
  \begin{proof}
  It is enough to show that both are injective. In fact,\\ 
$
\small
{ \alpha(i)=\alpha(j) \iff \gamma^{k(i)-1}(\zeta^{-1}(i))=\gamma^{k(j)-1}(\zeta^{-1}(j)) \iff \zeta^{-1}(i)=\gamma^{k(j)-k(i)}(\zeta^{-1}(j))},
$\\
 \normalsize
 then both, $\zeta^{-1}(i)$ and $\zeta^{-1}(j)$, are in the same cycle of $\gamma$. Equivalently, by Lemma \ref{conj},   $\delta^{-1}(i)$ and $\delta^{-1}(j)$ are in the same cycle of $\delta$. But this equivalence means that $i$ and $j$ are in the same cycle of $\delta$ and therefore $k(i)=k(j)$. Using this in the last equality above, we obtain $\zeta^{-1}(i)=\zeta^{-1}(j)$ and $i=j$.
 Analogously we show that $\beta$ is injective.\\
\noindent \textcolor{black}{Notice that: \\
$\iota^{-1}\alpha \iota(j)=\iota^{-1}\gamma^{k(\overline{j})-1}(\zeta^{-1}(\overline{j}))=\iota^{-1}\gamma^{k(\overline{j})}({j})=\iota^{-1}(\iota \delta^{k(\overline{j})} \iota)(j))=\delta^{k(\overline{j})-1}\zeta(j)=\beta(j),$ where the last three equalities come from Lemma \ref{conj},  $\iota=\iota^{-1}$ and $k(\overline{j})=k(\zeta(j))$.\\
Finally, to show that $\beta=\alpha^{-1}$, since  $\beta=\iota^{-1}\alpha \iota$, it is equivalent to show that   $(\alpha \iota)^2(j)=j$, for $j\in \{1, \dots, 2N\}$. In fact,  notice that $\alpha \iota (j)=\alpha(\overline{j})=\gamma^{k(\overline{j})-1}(\zeta^{-1}\iota(j))=\gamma^{k(\overline{j})}(j)$, Then $(\alpha \iota)^2(j)=\gamma^{2k(\overline{j})}(j)=j$, where the last equality comes from Remark \ref{lj}.}
  \end{proof}
\section{Construction of a group from the map $\Phi$}\label{g-map}
In this section we construct a group from any map 
$\Phi $ in the class defined in \S \ref{TheClassPhi}, as announced in the title. The group is a subgroup of 
${\rm Homeo}^+ (S^1)$ obtained in several steps. The first goal is to construct a finite set of homeomorphisms, as potential generators of the group. The final step is to verify that this collection satisfies some relations in 
${\rm Homeo}^+ (S^1)$. It is notoriously difficult to check a relation, i.e., an equality in a group. Here the guidelines are the specific properties of the map $\Phi$, in particular the conditions (EC), (E+), (E-). 

The construction of the potential generating homeomorphisms has several steps. It starts by a ``toy model" construction and then a family of diffeomorphisms with integer parameters. The toy model construction from $\Phi$ is a simple connect-the-dots operation. 
The family of diffeomorphisms is an improvement of this construction, it has the property that some particular compositions of the diffeomorphisms, given by the map, are equal ``locally" i.e., on some intervals. These local equalities are obtained from the dynamical conditions (EC), (E+), (E-).
The final step is a limit process: when the parameters grow, the local equalities become global i.e., the intervals on which an equality occurs grow to cover all $S^1$ at the limit.

\subsection{A toy model construction of diffeomorphisms from $ \Phi$}
By condition (CS-$\lambda$) we replace our initial piecewise homeomorphism $\Phi$ by 
the piecewise affine map $\widetilde{\Phi}$ with constant slope $\lambda > 1$, where 
$ \widetilde{\Phi} = g^{-1} \circ \Phi \circ g$, for $g \in \textrm{Homeo}^+ (S^1) $.
The piecewise affine map $ \widetilde{\Phi}$ is defined by a partition: 
$S^1 = \bigcup_{j = 1}^{2N} \widetilde{I}_j$,
where: \\
 \centerline{$ \widetilde{I}_j =[\tilde{z}_j, \tilde{z}_{\zeta(j)}):= g^{-1} ( I_j) \textrm{ and } \widetilde{\Phi}_{j}:=\widetilde{\Phi}_{|_{\widetilde{I}_j}}, \textrm{ for  } j \in \{ 1, \dots , 2N\}.$}

\begin{Lemma}\label{prop-gen}
Assume  $\Phi : S^1 \rightarrow S^1$ is a piecewise homeomorphism of $S^1$ satisfying the conditions $\rm (SE)$ and {\rm (CS-$\lambda$)} with slope $\lambda > 1$. For each $j \in \{1, \dots, 2N\}$, using the notations above, there is a class of diffeomorphisms 
$ [ f_{j}] \subset {\rm Diff}^+ (S^1) $ such that:
 \vspace*{-3pt}
\begin{itemize}[noitemsep, leftmargin=17pt]
\item[$(1)$] For each \textcolor{black}{ $f \in [ f_{j}]$,  $f_{|\widetilde{I_j}} =  \widetilde{\Phi}_{j}$ and 
$f_{| \widetilde{\Phi}_{{\iota(j)}} (\widetilde{I}_{\iota(j)})} =
( \widetilde{\Phi}_{\iota(j)})^{- 1}_{| \widetilde{\Phi}_{{\iota(j)}} (\widetilde{I}_{\iota(j)})} $,}
 \item[$(2)$] $f$ is a hyperbolic M\"obius like diffeomorphism i.e., with one attractive and one repelling fixed point and one pair of neutral points i.e., with derivative one.
\item[$(3)$]  \textcolor{black}{ If $f \in [ f_{j}]$ then $f^{-1} \in [ f_{\iota(j)}] $.}
\end{itemize}
\end{Lemma}
\begin{proof} Since the intervals $I_j$ and $\Phi ( I_{\iota(j)} )$ are disjoint by condition (SE), then    $ \widetilde{I}_j$ and 
 $\widetilde{\Phi}_{{\iota(j)}} (\widetilde{I}_{\iota(j)})$ are disjoint and the condition $(1)$ has no constraints.\\
By condition \textcolor{black}{(CS-$\lambda$)} the slope of $\widetilde{\Phi}$  in  $ \widetilde{I}_j$ and 
$ \widetilde{I}_{\iota(j)}$ is $\lambda$, then  \textcolor{black}{by condition (1), $f_{|\widetilde{I}_{j}}$ is affine of slope 
$\lambda$ and 
$f_{|\widetilde{\Phi}({\widetilde{I}_{\iota(j)}})}$} is affine of slope 
$\lambda^{-1}$. The map \textcolor{black}{$f$} is defined on
$\widetilde{I}_j \bigcup \widetilde{\Phi}_{\iota(j)}( \widetilde{I}_{\iota(j)})$, it remains to define it on the complementary intervals:
 \vspace{-5pt}
\begin{equation}\label{complement}
 S^1 -  ( \widetilde{I}_j \bigcup \widetilde{\Phi}_{\iota(j)}( \widetilde{I}_{\iota(j)})) = 
 L_j \bigcup R_j,
\end{equation}

 \vspace{-5pt}
\noindent where $L_j:=[\widetilde{\Phi}_{\iota (j) }( \tilde{z}_{_{\delta(j)}}) ,\tilde{z}_{_j}]$ and $R_j:=[\tilde{z}_{_{\zeta(j)}}, \widetilde{\Phi}_{\iota (j) }( \tilde{z}_{_{\iota (j) }} ) ]$
(see Figure \ref{connect-dots}).\\
The existence of the diffeomorphism \textcolor{black}{$f$} is a ``differentiable connect-the-dots" construction. The constraints are the images of the extreme points:\\
\textcolor{black}{ \centerline{ $f ( \partial \widetilde{I}_j ) = \partial \widetilde{\Phi}_j  ( \widetilde{I}_j ) \textrm{ and } 
f ( \partial \widetilde{\Phi}_{\iota(j)} (  \widetilde{I}_{\iota(j )} ) ) =  \partial \widetilde{I}_{\iota(j)},$}}
 together with the derivatives at these points which are, respectively 
 $ \lambda$ and $ \lambda^{-1}$.
 \begin{figure}[h]
\begin{center}
 \resizebox{0.8\textwidth}{!}{%
\begin{tikzpicture}[scale=0.50]

\draw [black,line width=0.2mm] circle [radius=3.2];
  \draw [black,line width=0.2mm](11,0) circle [radius=3.2];

   \draw[line width=1.2mm, gray]
  (-1,3)
  arc[start angle=114,end angle=70,radius=2.8]
  node[midway]{{$$}};
   
    \draw[line width=1.2mm, gray](1.8,2.6)
  arc[start angle=55,end angle=-235,radius=3.15]
  node[midway]{{$$}};
  
   \draw[line width=1.2mm, gray](11.25,-3.15)
  arc[start angle=-87.655,end angle=-118,radius=3.14]
  node[midway]{{$$}};

\draw[line width=1.2mm, gray]
  (9,-2.37)
  arc[start angle=230,end angle=-70,radius=3.15]
  node[midway]{{$$}};
   
   
   \draw[black,very thin,thick,-] (1.15,3.25) -- (1,2.8);
    
    \draw [decorate,color=black] (-1.1,-2.85)
   node[left] {$\boldsymbol{\cdot}$};
 \draw [decorate,color=black] (-0.8,-3.2)
   node[left] {\begin{tiny}${\tilde{z}_{_{\delta(j)}}}$\end{tiny}}; 
   
    \draw [decorate,color=black] (0.55,-3.2)
   node[left] {$\boldsymbol{\cdot}$};
    \draw [decorate,color=black] (1.1,-3.6)
   node[left] {\begin{tiny}${\tilde{z}_{_{\iota(j)}}}$\end{tiny}}; 
   
    \draw[black,very thin,thick,-] (-1.1,3.2) -- (-.95,2.8);
\draw [decorate,color=black] (-.7,3.5)
   node[left] {\begin{tiny}${\tilde{z}_{_j}}$\end{tiny}};
   
   \draw [decorate,color=black] (-.6,2.4)
   node[left] {\begin{footnotesize}${L_j}$\end{footnotesize}};
   
   \draw [decorate,color=black] (1.8,2.4)
   node[left] {\begin{footnotesize}${R_j}$\end{footnotesize}};
   
    \draw [decorate,color=black] (1.9,3.6)
   node[left] {\begin{tiny}$\tilde{z}_{_{\zeta(j)}}$\end{tiny}};
   
   \draw[black,very thin,thick,-] (-1.95,2.8) -- (-1.7,2.4);
   \draw [decorate,color=black] (-1.66,2.9)
   node[left] {\begin{tiny}${\widetilde{\Phi}_{\iota(j)}(\tilde{z}_{_{\delta(j)}}})$\end{tiny}};
   
    \draw[black,very thin,thick,-] (1.9,2.8) -- (1.7,2.4);
\draw [decorate,color=black] (4.6,2.9)
   node[left] {\begin{tiny}${\widetilde{\Phi}_{\iota(j)}(\tilde{z}_{_{\iota(j)}}})$\end{tiny}};


  \draw [decorate,color=black] (10,2.9)

   node[left] {$\boldsymbol{\cdot}$};
   \draw [decorate,color=black] (10,3.4)
   node[left] {\begin{tiny}${\tilde{z}_{_j}}$\end{tiny}};
   \draw [decorate,color=black] (12.3,3.05)
   node[left] {$\boldsymbol{\cdot}$};
    \draw [decorate,color=black] (12.7,3.6)
   node[left] {\begin{tiny}$\tilde{z}_{_{\zeta(j)}}$\end{tiny}};

    \draw[black,very thin,thick,-] (8.8,-2.65) -- (9.1,-2.2);

 \draw [decorate,color=black] (9,-2.5)
   node[left] {\begin{tiny}${\widetilde{\Phi}_{j}(\tilde{z}_{_j}})$\end{tiny}}; 
   
     \draw[black,very thin,thick,-] (12,-2.7) -- (12.15,-3.2);
      \draw [decorate,color=black] (14.5,-3.5)
   node[left] {\begin{tiny}${\widetilde{\Phi}_{j}(\tilde{z}_{_{\zeta(j)}}})$\end{tiny}};

  \draw[black,very thin,thick,-] (9.5,-3.1) -- (9.75,-2.7);
 \draw [decorate,color=black] (10.55,-3.4)
   node[left] {\begin{tiny}${\tilde{z}_{_{\delta(j)}}}$\end{tiny}}; 
   
    \draw [decorate,color=black] (10.6,-2.2)
   node[left] {\begin{footnotesize}${R_{\iota(j)}}$\end{footnotesize}};
   
     \draw[black,very thin,thick,-] (11.2,-2.95) -- (11.25,-3.4);
      \draw [decorate,color=black] (12.1,-3.6)
   node[left] {\begin{tiny}${\tilde{z}_{_{\iota(j)}}}$\end{tiny}}; 
   
   \draw [decorate,color=black] (12.2,-2.6)
   node[left] {\begin{footnotesize}${L_{\iota(j)}}$\end{footnotesize}};
   
    \draw [decorate,color=black] (6,1.5)
   node[left] {\begin{footnotesize}${\widetilde{\Phi}_j}$\end{footnotesize}};
    \draw[black,line width=0.2mm,->] (4.6,1) -- (6.6,1);
    
       \draw [decorate,color=darkgray] (6.6,-1.5)
   node[left] {\begin{footnotesize}${\widetilde{\Phi}_{\iota(j)}}$\end{footnotesize}};
      \draw[black,line width=0.2mm,->] (6.6,-1)--(4.6,-1);
\end{tikzpicture}
}%
\caption{The connect-the-dots construction of $f_{j} \in {\rm Diff}^+ (S^1)$}\label{connect-dots}
\end{center}
\end{figure}
The connect-the-dots  construction is simple enough and we could stop here. We give more precision that will be needed later.
Let $X$ and $Y$ be two disjoint intervals of $S^1$, we denote $\partial^{\pm} X$ the two boundary points of $X$, where the indices $\pm $ refer to the orientation of the interval. Let
$\textrm{Diff}^+ (X,Y)$ be the space of orientation preserving diffeomorphisms from $X$ to $Y$. Let 
$a,b \in \mathbb{R}^+$ and $\dd g$ be the derivative of $g$, we define:
\begin{equation}\label{Diffmon}
\begin{array}{l}
\textrm{Diff}_{a,b}^+ (X,Y) = \{ g \in \textrm{Diff}^+ (X,Y): 
\dd g(\partial^{-} X) = a > 0, \dd g(\partial^{+} X) = b > 0 \},\\
\textrm{ if } a \neq b : \textrm{Diff}_{a,b}^{(mon)} (X, Y) :=  \{g \in \textrm{Diff}_{a,b}^{+} (X, Y) :
 \dd g \textrm{ is monotone} \}.
\end{array}
 \end{equation}
We define \textcolor{black}{$f$} on the two intervals $L_j$ and $R_j$. The image of these intervals are, by condition $(1)$, respectively:
  $R_{\iota(j)}$ and $L_{\iota(j)}$.\\
Since 
\textcolor{black}{$f$} is required to be a diffeomorphism, the derivative \textcolor{black}{ $\dd f$} varies continuously from $\lambda > 1$ to $\lambda^{-1} < 1$ along $R_j$ and from $\lambda^{-1} < 1$ to $\lambda > 1$ on $L_j$. 
In other words:\\
 \textcolor{black}{ \centerline{ $ f_{| R_j} \in \textrm{Diff}_{\lambda, \lambda^{-1}}^+ (R_j,L_{\iota(j)})$ and 
$f_{| L_j} \in \textrm{Diff}_{\lambda^{-1} , \lambda, }^+ (L_j;R_{\iota(j)})$.}}
 Thus \textcolor{black}{ $f$} is highly non-unique. By the intermediate value theorem and $\lambda > 1$, there is at least one point with derivative one i.e., a neutral point, in each interval $L_j$ and $R_j$.

Condition $(2)$ requires the existence of exactly one neutral point $N_j^+$ in $R_j$ and one neutral point 
$N_j^-$ in $L_j$. This is the simplest situation, it is realized if the derivative varies monotonically in $R_j$ and $L_j$, in other words:\\
 \textcolor{black}{ \centerline{ $f_{| R_j} \in \textrm{Diff}_{\lambda, \lambda^{-1}}^{( mon )} (R_j,L_{\iota(j)})$ and 
$ f_{| L_j} \in \textrm{Diff}_{\lambda^{-1} , \lambda, }^{( mon )} (L_j;R_{\iota(j)})$.}}\\
By condition (SE) and (CS-\textcolor{black}{$\lambda$}), see the property (II'), the map \textcolor{black}{ $f$} has exactly two fixed points, one expanding in $\widetilde{I}_j$ and one contracting in $\widetilde{I}_{\iota(j)}$. Therefore, with the above choices, condition $(2)$ of the Lemma is satisfied for 
 \textcolor{black}{ $f $}.\\
\noindent Let us denote by $\{ f_{j} \}$ the subset of ${\rm Diff}^+(S^1)$ satisfying conditions $(1)$ and $(2)$. 
Fixing \textcolor{black}{ $f \in \{ f_{j} \}$, by construction we have
 $f^{-1} \in \{ f_{\iota(j)} \}$.} Therefore the pair \textcolor{black}{ $f$, $f^{-1}$} satisfies the condition $(3)$ of Lemma \ref{prop-gen}. 
Let us denote by $ [ f_{j} ]$ the subset of ${\rm Diff}^+ (S^1)$ satisfying $(1), (2), (3).$
\end{proof}

  \subsection{Dynamical properties of  $\Phi$}
From now on the map $\Phi$ satisfies all the ruling conditions of \S \ref{TheClassPhi} i.e., the conditions (SE), (EC), (E$\pm$),  (CS-$\lambda$), they are crucial for the next result.

 \begin{Lemma}\label{affin-diffeo}
 Let  $\Phi: S^1 \rightarrow S^1$ be a piecewise homeomorphism satisfying conditions
$\rm(EC), (E+), (E\textrm{-})\, \textrm{and} \, (CS\textrm{-}\lambda)$ for some number $\lambda > 1$.
Then there exists a maximal neighborhood $V_j$ of the cutting point $z_j$, for all $j \in \{1,\dots,2N\}$, such that 
$\Phi^{k(j)}|_{V_j}$ is continuous and conjugated to an affine diffeomorphism 
$\widetilde{\Phi}^{k(j)}|_{\widetilde{V}_j}$ with slope 
$\lambda^{k(j)}$. The integer $k(j)$ is given by condition $\rm(EC)$ for the cutting point $z_j$. The neighborhood 
$\widetilde{V}_j$ of $\widetilde{z}_j$ is the image of $V_j$ under $g^{-1} \in {\rm Homeo}^+ (S^1)$  that conjugates 
$\Phi$ to  $\widetilde{\Phi}$.
 \end{Lemma}
 \begin{proof}
 As in the previous proof, we replace the piecewise homeomorphism $\Phi$ by 
the piecewise affine map $\widetilde{\Phi}$ with constant slope $\lambda > 1$, using condition (CS-$\lambda$) and the conjugacy given by 
 $g \in \textrm{Homeo}^+ (S^1) $.\\
 By condition (EC) for the cutting point $\widetilde{z}_j$ we have:\\
 \centerline{ $\mathscr{z}_j^{k(j)}=\widetilde{\Phi}^{k(j)-1}(\widetilde{\Phi}_{\zeta^{-1}(j)}(\widetilde{z}_j)) =
 \widetilde{\Phi}^{k(j)-1}(\widetilde{\Phi}_{j}(\widetilde{z}_j)). $}
 Suppose that $\mathscr{z}_j^{k(j)}\in \widetilde{I}_{a_j}$, for some
 $a_j \in \{1,\dots, 2N \}$.\\
 Consider the pre-images of the point $\mathscr{z}_j^{k(j)} $ from the left and the right, along the orbits of the cutting point $\widetilde{z}_j$. Namely we consider the points:
 
\centerline{ $ \mathscr{z}^{}_{\tiny{\delta^{k(j) -1 } (j)} }= \widetilde{\Phi}^{k(j) -2}  ( \widetilde{\Phi}_{j} (\widetilde{z}_j) ) \in
 \widetilde{I}_{ \delta^{k(j) -1 } (j)} \textrm{ by (E+) }$,} 

\centerline{ $\mathscr{z}^{}_{ \tiny{\gamma^{k(j) -1}(\zeta^{-1} ( j ) ) }}  =  \widetilde{\Phi}^{k(j) -2}  ( \widetilde{\Phi}_{\zeta^{-1} (j)} 
 (\widetilde{z}_j) )  
 \in \widetilde{I}_{ \gamma^{k(j) -1}(\zeta^{-1} ( j ) )} \textrm{ by (E-) }.$}
 
  \vspace*{3pt}
\noindent In order to simplify the notations let us define (see Figure \ref{Vj}):
    \begin{equation}\label{cj-dj}
  \begin{array}{c}
   c_j:=\gamma^{k(j)-1}(\zeta^{-1}(j)),\;
   d_j:=\delta^{k(j)-1}(j),\\
   J_{c_j}:=[ \tilde{z}_{c_j} ,  \mathscr{z}_{c_j } ] \subset I_{c_j},\;
   J_{d_j}:=[ \mathscr{z}_{d_j} ,  \tilde{z}_{\zeta (d_j)} ] \subset I_{d_j},\;
   \end{array}
   \end{equation}
   From condition (EC) and the definitions above we obtain:\\
\centerline{ $ \widetilde{\Phi}_{c_j} ( J_{c_j} ) \cap \widetilde{\Phi}_{d_j} ( J_{d_j} ) = \mathscr{z}_j^{k(j)}$ and $\widetilde{\Phi}_{c_j} ( J_{c_j} ) \cup \widetilde{\Phi}_{d_j} ( J_{d_j} )$ is connected.
}
  Define the 
   ($\widetilde{\Phi}^{k(j)-1}$)-pre-images of $J_{c_j}$ and $J_{d_j}$ along the two orbits of $\widetilde{z}_j$ i.e., the left and the right orbits. 
   These pre-images belong respectively to the intervals $\widetilde{I}_j$ and $\widetilde{I}_{\zeta^{-1}(j)}$ and we obtain:
   \begin{equation}\label{Vcdj}
  V_j^{c_j} = [ \widetilde{\Phi}^{ - k(j) +1} (\widetilde{z}_{c_j} ) , \widetilde{z}_j   ] \subset 
  \widetilde{I}_{\zeta^{-1}(j)} 
  \textrm{ and }
  V_j^{d_j} = [  \widetilde{z}_j , \widetilde{\Phi}^{ - k(j) +1} (\widetilde{z}_{\zeta (d_j )} )   ] \subset \widetilde{I}_{j}.
\end{equation}  
     We define a neighborhood of the cutting point $\widetilde{z}_j$ by 
$ \widetilde{V}_j:=V_j^{c_j}\cup V_j^{d_j}.$
 
 
 \begin{figure}[h]
\begin{center}
\resizebox{0.7\textwidth}{!}{%
\begin{tikzpicture}[scale=0.7]
\draw [ black] circle [radius=6];
\draw [ black] circle [radius=1.8];
\draw [ black] circle [radius=4.8];

 \draw[line width=1.5mm, gray]
  (-0.3,1.85)
  arc[start angle=100,end angle=90,radius=1.8]
  node[midway]{{$$}};
  
   \draw[line width=1.5mm, black]
  (-0.,1.85)
  arc[start angle=92,end angle=80,radius=1.8]
  node[midway]{{$$}};

 \draw[line width=1.5mm, gray]
  (-4.15,-2.6)
  arc[start angle=215,end angle=232,radius=4.8]
  node[midway]{{$$}};

 \draw[line width=1.5mm, black]
  (2.3,-4.3)
  arc[start angle=296,end angle=309,radius=4.8]
  node[midway]{{$$}};

 \draw[line width=1.5mm, black]
  (5.4,2.5)
  arc[start angle=25,end angle=-175,radius=6]
  node[midway]{{$$}};
  
   \draw[line width=1.5mm, gray]
  (5.45,2.4)
  arc[start angle=23,end angle=155,radius=6]
  node[midway]{{$$}};

 \draw[dashed,gray,very thin,thick,-] (-1.1,-1.6) -- (-4,-4.5);
 
 \draw [decorate,color=black] (-2,-3.7)
   node[left] {\begin{tiny}$\tilde{z}_{_{c_j}}$\end{tiny}};
   
      \draw [decorate,color=black] (-2.8,-3.6)
   node[left] {$\boldsymbol{\cdot}$};

      \draw [decorate,color=black] (-4,-2.1)
   node[left] {$\boldsymbol{\cdot}$};
   
    \draw [decorate,color=black] (-3.5,-1.8)
   node[left] {\begin{tiny}$\tilde{z}_{_{\zeta(c_j)}}$\end{tiny}};

\draw[dashed,gray,very thin,thick,-] (.75,-2) -- (3,-5.2);

\draw [decorate,color=black] (-0.25,1.7)
   node[left] {$\boldsymbol{\cdot}$};

    \draw[dashed,gray,very thin,thick,-] (0,1.8) -- (0,6);
   
    \draw [decorate,color=black] (0.55,1.2)
   node[left] {\begin{tiny}${\tilde{z}_{_{j}}}$\end{tiny}};
   
    \draw [decorate,color=black] (0.55,2.5)
   node[left] {\begin{footnotesize}$\textcolor{black}{\widetilde{V}}_j$\end{footnotesize}};

   \draw[dashed,gray,very thin,thick,-] (-.57,1.65) -- (-1.6,5.8);
   
   \draw [decorate,color=black] (-0.1,1.6)
   node[left] {\begin{tiny}${\tilde{z}_{_{\zeta^{-1}(j)}}}$\end{tiny}};

   \draw [decorate,color=black] (1,1.7)
   node[left] {$\boldsymbol{\cdot}$};
   
    \draw[dashed,gray,very thin,thick,-] (0.6,1.65) -- (1.5,5.8);

    \draw [decorate,color=black] (1.7,1.25)
   node[left] {\begin{tiny}$\tilde{z}_{_{\zeta(j)}}$\end{tiny}};
   
    \draw[dashed,gray,very thin,thick,-] (-1.7,0) -- (-6,-0.2);
   
  \draw [decorate,color=black] (-1.44,0)
   node[left] {$\boldsymbol{\cdot}$};
   
   \draw[dashed,gray,very thin,thick,-] (-1.7,0.5) -- (-6,1);
   
   \draw [decorate,color=black] (-1.35,0.5)
   node[left] {$\boldsymbol{\cdot}$};
   
    \draw [decorate,color=black] (-1.15,1)
   node[left] {$\boldsymbol{\cdot}$};
   
    \draw[dashed,gray,very thin,thick,-] (-1.7,1) -- (-5.7,2.2);
      
    \draw [decorate,color=black] (-1.15,-1)
   node[left] {$\boldsymbol{\cdot}$};
   
    \draw[dashed,gray,very thin,thick,-] (-1.5,-1) -- (-5.5,-2.5);

    \draw [decorate,color=gray] (-1.7,-2.)
   node[left] {\begin{footnotesize}$\widetilde{I}_{_{c_j}}$\end{footnotesize}};
   
   \draw [decorate,color=black] (-2.8,-2.8)
   node[left] {\begin{scriptsize}$J_{_{c_j}}$\end{scriptsize}};

    \draw [decorate,color=black] (-0.6,-1.56)
   node[left] {$\boldsymbol{\cdot}$};

     \draw [decorate,color=black] (1,-1.7)
   node[left] {$\boldsymbol{\cdot}$};
   
     \draw [decorate,color=gray] (2.5,-2.5)
   node[left] {\begin{footnotesize}$\widetilde{I}_{_{d_j}}$\end{footnotesize}};
   
    \draw [decorate,color=black] (3.1,-3.6)
   node[left] {\begin{scriptsize}$J_{_{d_j}}$\end{scriptsize}};
   
    \draw [decorate,color=black] (3.9,-3.2)
   node[left] {$\boldsymbol{\cdot}$};
   
    \draw [decorate,color=black] (4.2,-2.9)
   node[left] {\begin{tiny}$\tilde{z}_{_{d_j}}$\end{tiny}};
   
     \draw [decorate,color=black] (2.7,-4.2)
   node[left] {$\boldsymbol{\cdot}$};
   
    \draw [decorate,color=black] (2.5,-4.2)
   node[left] {\begin{tiny}$\tilde{z}_{_{\zeta(d_j)}}$\end{tiny}};

     \draw [decorate,color=black] (1.65,-1.3)
   node[left] {$\boldsymbol{\cdot}$};
   
    \draw[dashed,gray,very thin,thick,-] (1.5,-1.45) -- (4.55,-4);
    
     \draw [decorate,color=black] (4.8,-.5)
   node[left] {\begin{footnotesize}$\widetilde{\Phi}^{k(j)-1}|_{\textcolor{black}{\widetilde{V}}_j}$\end{footnotesize}};
   
     \draw[black,line width=0.1mm,->] (2.1,-1.25) -- (4,-1.25);
     
       \draw [decorate,color=black] (5.8,-.5)
   node[left] {\begin{footnotesize}$\widetilde{\Phi}$\end{footnotesize}};
   
    \draw[black,line width=0.1mm,->] (5.2,-1.2) -- (5.6,-1.2);
    
     \draw [decorate,color=black] (1.9,1)
   node[left] {$\boldsymbol{\cdot}$};
   
    \draw[dashed,gray,very thin,thick,-] (1.5,1) -- (4.95,3.5);
    
     \draw [decorate,color=black] (2.1,.5)
   node[left] {$\boldsymbol{\cdot}$};
   
    \draw[dashed,gray,very thin,thick,-] (1.7,.5) -- (5.75,2);

 \draw [decorate,color=black] (0.48,1.63)
   node[left] {$\boldsymbol{\star}$};
   
    \draw [decorate,color=black] (5.65,2.35)
   node[left] {$\boldsymbol{\star}$};
   
    \draw [decorate,color=black] (7,2.8)
   node[left] {\begin{tiny}$\mathscr{z}_j^{k(j)}$\end{tiny}};
   
    \draw [decorate,color=black] (6.05,1.95)
   node[left] {$\boldsymbol{\cdot}$};
   
    \draw [decorate,color=black] (7.5,1.75)
     node[left] {\begin{tiny}$\tilde{z}_{_{\zeta(a_j)}}$\end{tiny}};

  \draw [decorate,color=black] (5.27,3.5)
   node[left] {$\boldsymbol{\cdot}$};
   
    \draw [decorate,color=black] (6,3.5)
     node[left] {\begin{tiny}$\tilde{z}_{_{a_j}}$\end{tiny}};

    \draw [decorate,color=black] (-3.5,-2.4)
   node[left] {$\boldsymbol{\star}$};
   
     \draw [decorate,color=black] (3.5,-3.5)
   node[left] {$\boldsymbol{\star}$};

    \draw [decorate,color=gray] (-4.6,1.5)
   node[left] {\begin{footnotesize}$\widetilde{I}_{_{\overline{c_j}}}$\end{footnotesize}};
   
    \draw [decorate,color=black] (-5.4,2.7)
   node[left] {\begin{tiny}$ \widetilde{\Phi}_{c_j}(\tilde{z}_{c_j})$\end{tiny}};
   
    \draw [decorate,color=black] (-5.75,1)
   node[left] {\begin{tiny}$\tilde{z}_{\overline{c_j}}$\end{tiny}};
   
    \draw [decorate,color=black] (-5.55,1)
   node[left] {$\boldsymbol{\cdot}$};
   
   \draw [decorate,color=black] (-5.9,-0.7)
   node[left] {\begin{tiny}$\widetilde{\Phi}_{d_j}(\tilde{z}_{\zeta(d_j)})$\end{tiny}};
   
   \draw [decorate,color=gray] (-4.7,0.3)
   node[left] {\begin{footnotesize}$\widetilde{I}_{_{\overline{d_j}}}$\end{footnotesize}};
 \end{tikzpicture}
 }
\caption{ \textcolor{black}{The neighborhood $\widetilde{V}_j$ and its image by $\widetilde{\Phi}^{k(j)}$} }\label{Vj} 
\end{center}
\end{figure}

\noindent By condition (EC),  $\widetilde{\Phi}^{k(j)}$ is continuous on $\widetilde{V}_j$ and:\\
 $\widetilde{\Phi}^{k(j)}(\widetilde{V}_j) = \widetilde{\Phi}_{c_j}( J_{c_j}) \cup 
  \widetilde{\Phi}_{d_j}( J_{d_j}) = 
 [\widetilde{\Phi}_{c_j}(\widetilde{z}_{c_j}), 
\widetilde{\Phi}_{d_j}(\widetilde{z}_{\zeta(d_j)})].$
It satisfies, by condition (SE), the following property (see Figure \ref{Vj}): 
   \begin{eqnarray}\label{k(j)-image} 
 \widetilde{\Phi}^{k(j)}(\widetilde{V}_j)\cap \widetilde{I}_k \neq \emptyset  ,  \forall k \neq \overline{c_j}, \overline{d_j}.
 \end{eqnarray}
 \noindent  By Lemma \ref{adj}, the indices $\overline{c_j}$ and $\overline{d_j}$ in condition (\ref{k(j)-image}) are adjacent with:
   $\zeta(\overline{d_j}) = \overline{c_j}$.
    From Lemma \ref{same-cycle}, the cycles containing $j$ and $\overline{c_j}$ are the same and thus $ k(\overline{c_j}) = k (j)$.

The map $\widetilde{\Phi}^{k(j)}|_{\widetilde{V}_j}$ is affine of slope 
$\lambda^{k(j)}$. Indeed, by definition of $ V_j^{c_j}$, $V_j^{d_j}$ and conditions ($\textrm{E}\pm $), the following properties are satisfied: \\
\centerline{$ \forall z \in  V_j^{d_j}:  \widetilde{\Phi}^{m}(z) \in \tilde{I}_{\delta^{m}(j)} 
\textrm{ and  } \forall z \in  V_j^{c_j} :
  \widetilde{\Phi}^{m}(z) \in \tilde{I}_{\gamma^{m}(\zeta^{-1} (j) )}, \textrm{ for } m=1,\dots,k(j)-1.$} Then we obtain:
\begin{equation}\label{compositionV}
  \begin{array}{c}
   \widetilde{\Phi}^{k(j)} (z) = 
  \widetilde{\Phi}_{\delta^{k(j)-1}(j)} \circ \dots \circ \widetilde{\Phi}_{\delta (j)} \circ \widetilde{\Phi}_{j} (z), \forall z \in  V_j^{d_j} \textrm{ and }\\
   \widetilde{\Phi}^{k(j)} (z) = \widetilde{\Phi}_{\gamma^{k(j)-1}(\zeta^{-1} (j) )} \circ \dots \circ 
\widetilde{\Phi}_{\gamma (\zeta^{-1} (j) )} \circ \widetilde{\Phi}_{\zeta^{-1} (j)} (z) , \forall z \in  V_j^{c_j}. 
     \end{array}
   \end{equation}

 \noindent Thus, $\widetilde{\Phi}^{k(j)} (z)$ is affine of slope $\lambda^{k(j)}$ for
$z \in V_j^{c_j} \cup V_j^{d_j} = \widetilde{V}_j$, as a composition of $k(j)$ affine maps, each of slope 
$\lambda$, on each side.
The definition of the intervals $J_{d_j}$ and $J_{c_j}$ in (\ref{cj-dj}) implies that in the above composition, 
$\widetilde{\Phi}_{\delta^{k(j)-1}(j)}$ and $\widetilde{\Phi}_{\gamma^{k(j)-1}(\zeta^{-1} (j) )}$ are affine of slope $\lambda$ and these intervals are maximal with that property for the compositions in (\ref{compositionV}). This completes the proof of the maximality property.
The neighborhood $V_j$ of the Lemma is then simply: $V_j = g (\widetilde{V}_j )$, where $g$ conjugates $\Phi$ with 
$\widetilde{\Phi}$.
\end{proof}

\subsection{ Affine extensions}

  In this subsection we extend the construction of the  diffeomorphisms in the \textcolor{black}{class} $[f_j]$ given by Lemma \ref{prop-gen}. The idea for these extensions comes from the properties (EC), ($\textrm{E}\pm$) and the expressions in (\ref{compositionV}) that are two compositions, equal at one point, and are expected to become an  equality on an interval.\\
    The first step is to enlarge the intervals on which the diffeomorphisms constructed in 
    Lemma \ref{prop-gen} are affine. 
    To that end we consider a collection $\nu := \{ \nu_j ; j = 1, \dots, 2N \} $ 
    of  neighborhoods $\nu_j = \nu_j (\widetilde{z}_j)$ of the cutting points $\widetilde{z}_j$.
  These neighborhoods are chosen small enough to satisfy:
   $\quad \nu_j \subset \widetilde{I}_j \cup \widetilde{I}_{\zeta^{-1}(j)}$ with $\nu_j \cap \nu_{\zeta (j)} = \emptyset$
  and $\nu_j \cap \nu_{\zeta^{-1} (j)} = \emptyset$.\\
    We define the {\em  $\lambda$-affine extension} 
   $\widetilde{\Phi}^{\nu}_{j}$ of $\widetilde{\Phi}_{j}$
    as a $\lambda$-affine map on the interval:
    \begin{eqnarray}\label{affine-extension}
  I^{\nu}_j := \widetilde{I}_j \cup \nu_j \cup \nu_{\zeta (j)}, \textrm{ so that }
  (\widetilde{\Phi}^{\nu}_{j})_{|I^{\nu}_j } \textrm{ is } \lambda\textrm{-affine and }
  (\widetilde{\Phi}^{\nu}_{j})_{|\widetilde{I}_j  } =
  (\widetilde{\Phi}_{j})_{|\widetilde{I}_j }.
  \end{eqnarray}

  \begin{Prop}\label{small-nu}
  If $\Phi$ satisfies the ruling conditions  $\rm(EC), (E\pm) \, \textrm{and} \, (CS\textrm{-}\lambda)$ for some number $\lambda > 1$, 
 then for small enough  neighborhoods $\nu_j$, satisfying (\ref{affine-extension}) for all 
  $j \in \{1, \dots, 2N \}$, the $\lambda$-affine extensions $\widetilde{\Phi}^{\nu}_{j}$ satisfy:\\
 \centerline{ $\widetilde{\Phi}^{\nu}_{j} (\nu_j) \subset \widetilde{I}_{\delta(j)} \setminus \nu_{\delta(j)}$ and 
  $\widetilde{\Phi}^{\nu}_{j} (\nu_{\zeta (j)} ) \subset \widetilde{I}_{\gamma(j)} \setminus \nu_{\iota(j)}$ for all 
   $j \in \{1, \dots, 2N \}$.}
  \end{Prop}
\begin{proof}  
  From condition (SE): $\widetilde{\Phi}_{j} (\widetilde{z}_j) \in \widetilde{I}_{\delta (j)}$ and 
  $\widetilde{\Phi}_{j} (\widetilde{z}_{\zeta (j)}) \in \widetilde{I}_{\gamma (j)}$.
  The $\lambda$-affine extension $\widetilde{\Phi}^{\nu}_{j}$ is continuous at $\widetilde{z}_{j}$ and
  $\widetilde{z}_{\zeta (j)}$. Thus, if the neighborhoods $\nu_j , \nu_{\zeta(j)} , \nu_{\delta(j)}, \nu_{\iota(j)}$ are sufficiently small then the conditions of the Proposition are satisfied by continuity.
\end{proof}

  If all the neighborhoods $\nu_j$ are small enough for Proposition \ref{small-nu} to apply then the sets $S^1 \setminus ( I^{\nu}_j \cup \widetilde{\Phi}^{\nu}_{\iota(j)}( I^{\nu}_{\iota(j)}))$
  and  $S^1 \setminus (I^{\nu}_{\iota(j)} \cup \widetilde{\Phi}^{\nu}_{j}( I^{\nu}_{j}))$ are non-empty and each one has two connected components:
   \begin{eqnarray}\label{complement-nu}
  S^1 \setminus ( I^{\nu}_j \cup \widetilde{\Phi}^{\nu}_{\iota(j)}( I^{\nu}_{\iota(j)})) = 
 L^{\nu}_j \cup R^{\nu}_j \textrm{ and }
 S^1 \setminus ( I^{\nu}_{\iota(j)} \cup \widetilde{\Phi}^{\nu}_{j}( I^{\nu}_{j})) = 
 L^{\nu}_{\iota(j)} \cup R^{\nu}_{\iota(j)} .
 \end{eqnarray}
 If all the neighborhoods $\nu_j$ are small enough so that each interval in (\ref{complement-nu}) contains an open interval,
  then we define the following family of diffeomorphisms for  $j\in \{1, \dots, 2N\}$:
  
  \vspace{3pt}
 \centerline{$ [ f^{\nu}_j ] \subset \rm{Diff}^+ (S^1)$, ``parametrised" by 
 $ \nu = \left\lbrace  \nu_j : j \in \{1, \dots , 2N\} \right\rbrace$ \textrm {such that:}} 
    \begin{equation}\label{Diff-nu}
  \begin{array}{l}
 \textrm{[i]  } ( f^{\nu}_{j})_{| I^{\nu}_{j}} := (\widetilde{\Phi}^{\nu}_{j})_{| I^{\nu}_{j}} \textrm{  and  }
  ( f^{\nu}_{j})_{| \widetilde{\Phi}^{\nu}_{\iota(j)}( I^{\nu}_{\iota(j)})} := 
  (\widetilde{\Phi}^{\nu}_{\iota(j)})^{-1}_{| \widetilde{\Phi}^{\nu}_{\iota(j)}( I^{\nu}_{\iota(j)})}, 
  \\ [0.85em]
   \textrm{[ii]  } ( f^{\nu}_{j})_{| L^{\nu}_{j}} \in {\rm{Diff}}^{(mon)}_{\lambda^{-1} , \lambda } ( L^{\nu}_{j} ; R^{\nu}_{\iota(j)}) 
   \textrm{  and  }   
   ( f^{\nu}_{j})_{| R^{\nu}_{j}} \in {\rm{Diff}}^{(mon)}_{\lambda, \lambda^{-1}} (R^{\nu}_{j} ; L^{\nu}_{\iota(j)}),\\
   [0.75em]

    \textrm{[iii]  }  (f^{\nu}_{j})^{-1} := f^{\nu}_{\iota (j)}.
   \end{array}
   \end{equation}  
   The diffeomorphisms in the class $[ f^{\nu}_j ]$ are similar but different from the class $[f_j]$ of Lemma \ref{prop-gen}. They are affine on larger intervals and the diffeomorphisms 
  $ f^{\nu}_j$ and $ f^{\nu}_{\zeta^{\pm 1}(j)}$ are affine on a common interval: $\nu_{\zeta(j)}$ or $\nu_j$ 
   respectively.
  
\begin{Lemma}\label{diffeo-V}
Let $\Phi$ satisfy the ruling conditions $\rm(EC), (E\pm),  \textrm{and} \, (CS\textrm{-}\lambda)$ for some number $\lambda > 1$, and
$\tilde{V}:=\{\widetilde{V}_j: j \in \{1,\dots,2N\}\}$ be the set of neighborhoods of Lemma \ref{affin-diffeo}.
Then, for all $j$:
\vspace{-3pt}
\begin{enumerate}[noitemsep, leftmargin=17pt]
 \item[$(a)$] $\widetilde{V}_j$ satisfies  Proposition \ref{small-nu}:
$\widetilde{\Phi}^{\tilde{V}}_{j}(\widetilde{V}_j) \subset \widetilde{I}_{\delta(j)} \setminus 
\widetilde{V}_{\delta(j)}$
and 
$\widetilde{\Phi}^{\tilde{V}}_{j}(\widetilde{V}_{\zeta(j)}) \subset \widetilde{I}_{\gamma(j)} \setminus 
\widetilde{V}_{\iota(j)}$, and
\item[$(b)$]
   $\left(  \widetilde{\Phi}^{\tilde{V}}_{\delta^{k(j)-1}(j)} \circ \dots \circ \widetilde{\Phi}^{\tilde{V}}_{\delta (j)} 
  \circ \widetilde{\Phi}^{\tilde{V}}_{j}\right)_{| \widetilde{V}_j} 
  =  
  \left(  \widetilde{\Phi}^{\tilde{V}}_{\gamma^{k(j)-1}(\zeta^{-1} (j) )} \circ \dots \circ 
\widetilde{\Phi}^{\tilde{V}}_{\gamma (\zeta^{-1} (j) )} \circ \widetilde{\Phi}^{\tilde{V}}_{\zeta^{-1} (j)} \right)_{| \widetilde{V}_j} $,\\
   where $k(j)$ is the integer of condition $\rm(EC)$ at the cutting point $\widetilde{z}_j$.
   \end{enumerate}
   \end{Lemma}
  \begin{proof} From the proof of Lemma \ref{affin-diffeo}, the condition (\ref{k(j)-image}) is satisfied for the neighborhoods $\widetilde{V}_j$.
 To simplify the formulation we consider the situation where $k(j) = 3$. 
 Conditions (\ref{k(j)-image}) and (SE) imply, in particular that: \\
 \centerline{ $ \widetilde{\Phi}^{3}(\widetilde{V}_j) \subset \widetilde{\Phi}_{\delta^2 (j)}(\widetilde{I}_{\delta^2 (j)})$ and symmetrically
 $ \widetilde{\Phi}^{3}(\widetilde{V}_j) \subset \widetilde{\Phi}_{\gamma^2 (\zeta^{-1} (j))}
 (\widetilde{I}_{\gamma^2 (\zeta^{-1} (j))})$. Thus: }\\ 
 \centerline{$  \widetilde{\Phi}_{\delta^2 (j)}^{-1}\left( \widetilde{\Phi}^{3}(\widetilde{V}_j)\right)  \subset \widetilde{I}_{\delta^2 (j)}$ and 
 $  \widetilde{\Phi}_{\gamma^2 (\zeta^{-1} (j))}^{-1}\left( \widetilde{\Phi}^{3}(\widetilde{V}_j)\right)  \subset 
 \widetilde{I}_{\gamma^2 (\zeta^{-1} (j))}$. }
 We focus on one side, for instance the $\delta(j)$-side.
 The inclusion is in fact more restrictive:
$\widetilde{\Phi}_{\delta^2 (j)}^{-1}\left( \widetilde{\Phi}^{3}(\widetilde{V}_j)\right) = \left[ \alpha ; \widetilde{z}_{\zeta (\delta^2(j))} \right] \subset  \widetilde{I}_{\delta^2 (j)}$ and $\alpha$ 
 satisfies: $ \alpha > \widetilde{\Phi}_{\delta (j)} (\widetilde{z}_{\delta (j)})$ with $\widetilde{\Phi}_{\delta (j)} (\widetilde{z}_{\delta (j)}) \in \widetilde{I}_{\delta^2 (j)}$.
 Indeed, by condition (E+) for the cutting point $\widetilde{z}_{\delta (j)}$, we have:\\ 
  $ \widetilde{\Phi}_{\delta^2 (j)} ( \widetilde{\Phi}_{\delta (j)} (\widetilde{z}_{\delta (j)} ) ) \in 
 \widetilde{I}_{\overline{\gamma^2} (\zeta^{-1}(j))}$ and 
 $ \widetilde{\Phi}_{\delta^2 (j)} ( \alpha) \notin 
 \widetilde{I}_{\overline{\gamma^2} (\zeta^{-1}(j))}$, by (\ref{k(j)-image}).
 This implies:\\
   $   \widetilde{\Phi}_{\delta^2 (j)}^{-1}\left( \widetilde{\Phi}^{3}(\widetilde{V}_j)\right)  
   \subset  \widetilde{\Phi}_{\delta (j)} ( \widetilde{I}_{\delta (j)} ) \cap \widetilde{I}_{\delta^2 (j)}$ and thus we obtain
 $\widetilde{\Phi}_{\delta (j)}^{-1} \circ \widetilde{\Phi}_{\delta^2 (j)}^{-1}\left( \widetilde{\Phi}^{3}(\widetilde{V}_j)\right)   \subset   \widetilde{I}_{\delta (j)}.$ 
  The map 
 $\widetilde{\Phi}_{\delta (j)}^{-1} \circ \widetilde{\Phi}_{\delta^2 (j)}^{-1} \circ \widetilde{\Phi}^{3}_{| \widetilde{V}_j}$
    is defined from $\widetilde{V}_j $ to  $\widetilde{I}_{\delta (j)}$.
     It is an affine map of slope $\lambda$ since 
 $\widetilde{\Phi}_{\delta (j)}^{-1}$ and $ \widetilde{\Phi}_{\delta^2 (j)}^{-1}$ are affine of slope $\lambda^{-1}$ and 
 $\widetilde{\Phi}^{3}_{| \widetilde{V}_j}$ is affine of slope $\lambda^3$ by Lemma \ref{affin-diffeo}.
 By definition of the $\lambda$-affine extension 
 $\widetilde{\Phi}^{\nu}_{j}$ with $\nu = \tilde{V}$ in (\ref{affine-extension}) and,  since: \\
 $\widetilde{\Phi}_{\delta (j)}^{-1} \circ 
 \widetilde{\Phi}_{\delta^2 (j)}^{-1} \circ \widetilde{\Phi}^{3}_{| \widetilde{V}_j \cap \tilde{I}_j} = 
 {\widetilde{\Phi_j}}_{{| \widetilde{V}}_j \cap \tilde{I}_j}$,
 we obtain 
$\widetilde{\Phi}^{\tilde{V}}_{j} (\widetilde{V}_j) = \widetilde{\Phi}_{\delta (j)}^{-1} \circ 
 \widetilde{\Phi}_{\delta^2 (j)}^{-1} \circ \widetilde{\Phi}^{3} ( \widetilde{V}_j) \subset 
 \widetilde{I}_{\delta (j)}$, this is a part of the result $(a)$ in the Lemma.\\
 We apply the same arguments to the neighborhood $\widetilde{V}_{\delta (j)}$ and we obtain:\\
 \centerline{ $\widetilde{\Phi}^{\tilde{V}}_{\delta(j)} (\widetilde{V}_{\delta(j)}) \subset \widetilde{I}_{\delta^2 (j)}$ and 
 $\widetilde{\Phi}^{\tilde{V}}_{\delta^2(j)} \circ \widetilde{\Phi}^{\tilde{V}}_{\delta(j)} (\widetilde{V}_{\delta(j)}) \subset 
 \widetilde{I}_{\overline{\gamma^2} (\zeta^{-1}(j))}$.}
 The last inclusion comes from Lemma \ref{gamma-delta} for $k(j) = 3$, i.e.: 
 $\delta^3 (j) = \overline{\gamma^2} (\zeta^{-1}(j))$.\\
 Hence, we obtain that: $\widetilde{\Phi}^{\tilde{V}}_{\delta(j)} (\widetilde{V}_{\delta(j)})$ and 
 $\widetilde{\Phi}^{\tilde{V}}_{\delta(j)} \circ \widetilde{\Phi}^{\tilde{V}}_{j} (\widetilde{V}_{j})$ are two disjoint sub-intervals of 
 $\widetilde{I}_{\delta^2 (j)}$ and then $ \widetilde{V}_{\delta(j)} \cap \widetilde{\Phi}^{V}_{j} (\widetilde{V}_{j}) = \emptyset$. 

  This completes the proof of condition $(a)$ for the $\delta$-side in the case $k(j) =3$.
 For the $\gamma$-side we replace condition (E+) by (E-) and use the same arguments. The general argument, for any $k(j)$, is the same with more compositions.
  
The neighborhoods $\widetilde{V}_{j} = V_j^{c_j} \cup V_j^{d_j}$ in the proof of 
 Lemma \ref{affin-diffeo}  satisfy the equalities in (\ref{compositionV}).
 Moreover, by  definition of the $\lambda$-affine extension $\widetilde{\Phi}^{\tilde{V}}_{j}$ on the interval $I^{\tilde{V}}_j$
 in (\ref{affine-extension}), the two maps $\widetilde{\Phi}^{\tilde{V}}_{j}$ and 
 $\widetilde{\Phi}^{\tilde{V}}_{\zeta^{-1} (j)}$ are  $\lambda$-affine on $\widetilde{V}_{j}$ with:\\
 $\widetilde{\Phi}^{\tilde{V}}_{j}(\widetilde{V}_j) \subset \widetilde{I}_{\delta(j)} \setminus \widetilde{V}_{\delta(j)}$
and 
$\widetilde{\Phi}^{\tilde{V}}_{\zeta^{-1} (j)}(\widetilde{V}_{j}) \subset \widetilde{I}_{\gamma(\zeta^{-1} (j))} 
\setminus \widetilde{V}_{\iota(\zeta^{-1} (j))}$, from the property $(a)$ 	above.\\
 Hence, as in Lemma \ref{affin-diffeo}, both compositions in each side of $(b)$ are affine of slope $\lambda^{k(j)}$ and by (EC) they are equal at the point $\{\tilde{z}_j\}= V_{j}^{c_j} \cap V_{j}^{d_j}$. Thus the equality $(b)$ is satisfied in $ \widetilde{V}_{j}$ .
\end{proof}
\subsection{ A parametrised extension family from $\Phi$}\label{parametrised-family}

 
 The goal of this paragraph is to extend further, in a parametrised way, the set of neighborhoods 
 $\nu= \{ \nu_j : j = 1, \dots, 2N \}$ used in the family $[ f^{\nu}_j]$ in (\ref{Diff-nu}).\\
We first enlarge each neighborhood in the collection from 
 $ \widetilde{V}_j$ to  $W_j$, on which the diffeomorphisms are affine.
   Recall the definition of $  \widetilde{V}_j$ via the left and right pre-images of the intervals:
 $J_{c_j}:=[ \widetilde{z}_{c_j} ,  \mathscr{z}_{_{c_j}} ] \subset \widetilde{I}_{c_j}$ and
 $J_{d_j}:=[ \mathscr{z}_{_{d_j}} ,  \widetilde{z}_{\zeta (d_j)} ] \subset \widetilde{I}_{d_j}$, given in (\ref{cj-dj}).
Recall that $I^{\tilde{V}}_j = \widetilde{V}_j \cup \widetilde{I}_j \cup \widetilde{V}_{\zeta (j)}$,  
as in (\ref{affine-extension}), for $\nu=\tilde{V}=\{\widetilde{V}_j: j \in \{1,\dots,2N\}\}$,  the set of neighborhoods of Lemma \ref{affin-diffeo}. 
Let us consider the intervals:\\ 
 \centerline{  $J'_{c_j} := [ \partial^{-}(I_{c_j}^{\tilde{V}}) ;  \mathscr{z}_{_{c_j}} ]$ and
    $J'_{d_j} := [ \mathscr{z}_{_{d_j}} ; \partial^{+}(I_{d_j}^{\tilde{V}})]$, they satisfy:}
     \vspace{-5pt}
  \begin{equation}\label{J'}
  \begin{array}{l}
  \textrm{[i]  }  {I}_{c_j}^{\tilde{V}} \supset J'_{c_j} \supset J_{c_j} \textrm{ and } 
   I_{d_j}^{\tilde{V}} \supset J'_{d_j} \supset J_{d_j}, 
  \\ 
   \textrm{[ii]  } \widetilde{\Phi}^{\tilde{V}}_{c_j} (J'_{c_j}) \cap \widetilde{\Phi}^{\tilde{V}}_{d_j} (J'_{d_j}) = \mathscr{z}_{j}^{k(j)} \textrm{ by condition (EC) }.
     \end{array}
   \end{equation}
    \vspace{-6pt}
\noindent Let  ${\cal W}_j := \widetilde{\Phi}^{\tilde{V}}_{c_j} (J'_{c_j}) \cup 
    \widetilde{\Phi}^{\tilde{V}}_{d_j} (J'_{d_j})$, then from Lemma \ref{diffeo-V}, exactly as in (\ref{k(j)-image}), it satisfies:
     \vspace{-5pt}
   \begin{equation}\label{almost all}  
  {\cal W}_j \cap  {I}^{\tilde{V}}_{k} \neq \emptyset, \textrm{ for all }  k \neq \overline{c_j}, \overline{d_j} \textrm{ and } j\in \{1,\dots,2N\}. 
 \end{equation}
  
 \vspace{-6pt}
 \noindent The neighborhood $\widetilde{V}_j$ was defined as the pre-images of the intervals $J_{c_j}$ and $J_{d_j}$ along the orbits of the cutting point $\tilde{z_j}$. We do the same for the larger  intervals $J'_{c_j}$ and $J'_{d_j}$. The various pre-images of $J'_{c_j}$ and $J'_{d_j}$ under $\widetilde{\Phi}$ are well defined, for instance:\\
    \centerline{$J'_{d_j} \subset \widetilde{\Phi}_{\delta^{k(j) - 2}} (\widetilde{I}_{\delta^{k(j) - 2}})$ and thus 
  $(\widetilde{\Phi}_{\delta^{k(j) - 2}} )^{-1} (J'_{d_j} ) \subset \widetilde{I}_{\delta^{k(j) - 2}}$.}\\
  We consider the $\widetilde{\Phi}^{k(j) - 1}$ pre-images of $J'_{d_j}$ and $J'_{c_j}$ along the two orbits of 
 $\widetilde{z}_j$ exactly as in (\ref{Vcdj}), and we define: 
  \begin{equation} \hspace{-5pt}
\label{Wj}
  \begin{array}{l}  
  W_{j}^{-} := \Big[\widetilde{\Phi}^{- k(j) + 1} \Big(\partial^{-}(I_{c_j}^{\tilde{V}})\Big) ; \widetilde{z}_j \Big ] 
  \subset \widetilde{I}_{\zeta^{-1}(j)},\, 
  W_{j}^{+} := \Big[ \widetilde{z}_j ; \widetilde{\Phi}^{- k(j) + 1} \Big(\partial^{+}(I_{d_j}^{\tilde{V}})\Big) \Big ] 
  \subset \widetilde{I}_{j} \\
\hspace{5cm} \,\textrm{ and }\; W_j :=  W_{j}^{-} \cup  W_{j}^{+} \supset \widetilde{V}_j.
  \end{array}
   \end{equation}
   \noindent \textcolor{black}{ At this point,  we obtain two  neighborhoods for each cutting point $\tilde{z}_j$,  either $ W_j^{0,0} := \tilde{V}_j$ or $ W_j^{1,1} := W_j $, }
 \textcolor{black}{and we can define  several extensions of the intervals $\tilde{I}_j$, replacing $I^{\tilde{V}}_j$,} namely:
    \begin{equation}\hspace{-3pt}\label{Interval-IW}
  \begin{array}{l} 
       \textcolor{black}{ I^{{1,0}}_{j}} := W_j \cup \widetilde{I}_j \cup 
   \widetilde{V}_{\zeta (j)} ;\quad
  \textcolor{black}{ I^{{0,1}}_{j}} := \widetilde{V}_j \cup \widetilde{I}_j \cup 
   W_{\zeta (j)} ; \quad
   \textcolor{black}{ I^{{1,1}}_{j} }:= W_j \cup \widetilde{I}_j \cup W_{\zeta (j)}.
   \end{array}
    \end{equation}
     \textcolor{black}{We also denote the various $\lambda$-affine extensions, as in (\ref{affine-extension}), by
   \textcolor{black}{$\widetilde{\Phi}^{*}_j$}, where the symbol $*$ stands for one possible choice }\textcolor{black}{ in $\{(0,0),(0,1), (1,0), (1,1)\}$.}

   \noindent    The enlargement operation: $\widetilde{V}_j \rightarrow W_j $ defined above can be iterated by replacing the intervals $ I^{\tilde{V}}_{j}$ 
  in definition (\ref{Wj}) by any of the intervals in (\ref{Interval-IW}).
   This iteration can be done ``$p_j$" times on the left $(-)$ and ``$q_j$" times on the right $(+)$, for $p_j,q_j$ arbitrary positive integers. 
   More precisely, consider the \textcolor{black}{recursive definition} for each $j = 1, \dots, 2N:$ 
 \begin{eqnarray}\label{Wpq}
 \begin{array}{l} 
  W_j^{0,0} = \widetilde{V}_j,\textrm { and }
  W_j^{p_j,q_j} := \Big[   \widetilde{\Phi}^{- k(j) + 1} \Big(\partial^{-}(W_{c_j}^{p_j-1, q'})\Big) ; 
  \widetilde{\Phi}^{- k(j) + 1} \Big(\partial^{+}(W_{\zeta(d_j)}^{p', q_j-1})\Big)    \Big], \\
  \textrm {for }  p_j, q_j > 0 \textrm { and }  p'\leq p_j-1, q'\leq q_j-1 .
  \end{array}
  \end{eqnarray} 
   This iterated enlargement defines a family of neighborhoods  
  $$W^* := \{ W^{p_j, q_j}_{j} , j\in \{1, \dots, 2N\},\; p_j \geq 0, \; q_j \geq 0\},$$ (see Figure \ref{VariationInterval}), and we define a
  $\lambda$-affine extensions \textcolor{black}{$\widetilde{\Phi}^{p_j,q_{\zeta(j)}}_j$} of 
  \textcolor{black}{$\widetilde{\Phi}_j$}, on the interval
   \vspace{-5pt}
  \textcolor{black}{
  \begin{equation}\label{Interval-IW-pq}
   I^{p_j,q_{\zeta(j)}}_j  := W^{p_j,q'}_j \cup \widetilde{I}_j \cup 
   W^{p',q_{\zeta(j)}}_{\zeta (j)},
      \end{equation} 
      }
   for $p',\; q' \geq 0$ free indices.        \\
     \noindent  The following result is a version of Lemma \ref{diffeo-V} for the neighorhoods \textcolor{black}{in $W^{*}$}.
     
      \textcolor{black}{In order to make notation  lighter in the following results, we will use the generic symbol $\widetilde{\Phi}^{W^{*}}_j$ for the different affine extensions associated to  neighborhoods in $W^*$. }  
          \begin{Prop}\label{gen-ext}
      For the intervals  $I^{p_j,q_{\zeta(j)}}_j$ and the extensions $\widetilde{\Phi}^{W^*}_j$ defined above,
      and all pairs of integers $(p_j,q_{j}) \in \mathbb{N} \times \mathbb{N}$, 
      $j \in \{1, \dots, 2N\}$
      the following properties are satisfied:
     \begin{enumerate}[noitemsep, leftmargin=17pt]
\item[$(a)$]  $\widetilde{\Phi}^{W^*}_j(W^{p_j,q_j}_j) \subset \widetilde{I}_{\delta(j)} \setminus  W^{p_{\delta(j)},q_{\delta(j)}}_{\delta(j)}$ and 
  $\widetilde{\Phi}_j^{W^*}(W^{p_{\zeta(j)},q_{\zeta(j)}}_{\zeta(j)}) \subset \widetilde{I}_{\gamma(j)} 
  \setminus  W^{p_{\iota(j)},q_{\iota(j)}}_{\iota(j)}$,
\item[$(b)$] $\left(  \widetilde{\Phi}^{{W^{*}}}_{\delta^{k(j)-1}(j)} \circ \dots \circ \widetilde{\Phi}^{{W^*}}_{\delta (j)} 
  \circ \widetilde{\Phi}^{{W^{*}}}_{j}\right)_{|_{W^{p_j,q_j}_j}}
  =  
  \left(  \widetilde{\Phi}^{{W^*}}_{\gamma^{k(j)-1}(\zeta^{-1} (j) )} \circ \dots \circ 
\widetilde{\Phi}^{{W^*}}_{\gamma (\zeta^{-1} (j) )} \circ \widetilde{\Phi}^{{W^*}}_{\zeta^{-1} (j)} \right)_{|_{W^{p_j,q_j}_j}}.$ 
\end{enumerate}
\end{Prop}
\begin{proof}
 The statement is difficult to read, in particular because of all the indices. We simplify the notations by writing $(p,q)$ instead of $(p_j, q_j)$. For $(p,q) = (0,0)$ the result is Lemma \ref{diffeo-V}, whose proof is based on the property (\ref{k(j)-image}). We observe that the condition (\ref{almost all}) is exactly (\ref{k(j)-image}) when 
$\widetilde{V}_j$ is replaced by $W_j = W_j^{1,1}$ as given in (\ref{Wpq}).
The condition (\ref{almost all}) can be expressed as:
$ (\widetilde{\Phi}^{{W^*}})^{k(j)} (W^{1,1}_{j}) \cap  \tilde{I}^{1,1}_{k} \neq \emptyset, 
\textrm{ for all }  k \neq \overline{c_j}, \overline{d_j}$.\\
This is the first step of an induction giving, with an abuse of notations:
\begin{equation}\label{PhiWk}
 (\widetilde{\Phi}^{{W^*}})^{k(j)} (W^{p,q}_{j}) \cap \tilde{I}^{p,q}_{k}  \neq \emptyset, 
\textrm{ for all }  k \neq \overline{c_j}, \overline{d_j}, \textrm{ and all finite } (p,q).
\end{equation}
  The arguments in the proof of Lemma \ref{diffeo-V} are now used inductively, using (\ref{PhiWk})
  in place of (\ref{k(j)-image}) with no new difficulties.
  \end{proof}
   
 \subsection{Generators and relations from $\Phi$}\label{generators-relations}
 The family of diffeomorphism classes  $[f^{\nu}_j]$ defined in (\ref{Diff-nu}) requires the collection of 
neighborhoods  $ \nu = \left\lbrace  \nu_j \right\rbrace$ to satisfy the conditions of Proposition \ref{small-nu}. 
This is part $(a)$ in Lemma \ref{diffeo-V} and Proposition \ref{gen-ext} for the collection of neighborhoods 
\textcolor{black}{ $V= \{ \tilde{V}_j \}$ and $W^* = \{ W^{p_j,q_j}_j \}$}.
Therefore the family of diffeomorphism classes $\{ [ f^{W^*}_j]: j \in \{1,\dots, 2N\} \}$ obtained from the neighborhoods in
$ W^*$ is well defined. In the previous notation, the set of ``parameters" is hidden in the symbol 
$ *$, it represents  $ * = \big\{ \big(p_j, q_j\big)\in \mathbb{N} \times \mathbb{N} : j \in \{1, \dots, 2N\} \big\}$.

      
       Each diffeomorphism \textcolor{black}{in $[f^{W^*}_{j}]$} is affine of slope $\lambda$ on the interval $I^{p_j,q_{\zeta(j)}}_j$, defined in  (\ref{Interval-IW-pq}), 
 and is affine of slope $\lambda^{-1}$ on the interval
      $ \widetilde{\Phi}^{W^*}_{\iota(j)}( I^{p_{\iota(j)} , q_{\delta(j)} }_{\iota(j)})$. 
   The complementary intervals, called the {\em variation intervals}, are defined by (\ref{complement-nu}), for $j=1,\dots,2N$: 
         \begin{equation}\label{complement-W}
 \begin{array}{l}       
        S^1 \setminus \{  I^{p_j,q_{\zeta(j)}}_j \cup 
  \widetilde{\Phi}^{W^*}_{\iota(j)}( I^{p_{\iota(j)} , q_{\delta(j)} }_{\iota(j)}) \} = 
 L^{q_{\delta(j) } , p_{ j}  }_j \cup R^{{q_{\zeta(j) } , p_{ \iota (j)}  }}_j .
  \end{array}
 \end{equation}

 The first goal in this section is to study how the various parameters $(p_j, q_j)$ are related from one variation interval to another. This is obtained in Lemma \ref{equality-variation} below. It is a key step, it defines an induction on the variation intervals, via an equality among some of the $R^{*}_j$ or $L^{*}_j$, for specific choices of the indices $(p, q)$.\\
  

 
  \begin{Lemma}\label{equality-variation}
  
  With the above notations, the following equalities, among variation intervals around the cutting point 
  $\widetilde{z}_j$ are satisfied, for $a,b,m,n \geq 1$:
  
  \vspace*{-3pt}
     \begin{enumerate}[noitemsep, leftmargin=20pt]
\item[$(a)$] $R^{a,b}_{\zeta^{-1} (j)} =  (\widetilde{\Phi}^{W^*}_{j})^{-1} \circ (\widetilde{\Phi}^{W^*}_{\delta (j)})^{-1} \circ \dots \circ
  (\widetilde{\Phi}^{W^*}_{\delta^{k(j) -2} (j)})^{-1}
  \Big[  R^{a-1,b-1}_{\delta^{k(j) - 1} (j)} \Big] $,
 \item[$(b)$]   $L^{m,n}_{j} = (\widetilde{\Phi}^{W^*}_{\zeta^{-1} (j)})^{-1} \circ 
(\widetilde{\Phi}^{W^*}_{\gamma (\zeta^{-1} (j))})^{-1} 
\circ \dots \circ (\widetilde{\Phi}^{W^*}_{\gamma^{k(j) -2} (\zeta^{-1} (j))})^{-1} 
\Big[  L^{m-1, n-1 }_{\gamma^{k(j) - 1} (\zeta^{-1} (j))} \Big].$
 \end{enumerate}
 The maps $ \beta : \zeta^{-1} (j) \rightarrow \delta^{k(j) - 1} (j) $ and 
$ \alpha : j \rightarrow \gamma^{k(j) - 1} (\zeta^{-1} (j)) $ that appear in $(a)$ and $(b)$ are the permutations of Lemma \ref{alpha-beta}.


\end{Lemma}
  \begin{proof}
  As in the proof of Lemma \ref{diffeo-V}, we focus on the case $k(j) =3$ and on one of the two symmetric equalities. For simplicity we use specific parameters $(p,q)$ on the neighborhoods only when it is necessary for the formulation, otherwise the indices are replaced by a \textcolor{black}{ ``$\star$"}, the important indices will be in bold, Figure \ref{VariationInterval} should help.
     \begin{figure}[htbp]
\centerline{\includegraphics[height=70mm]{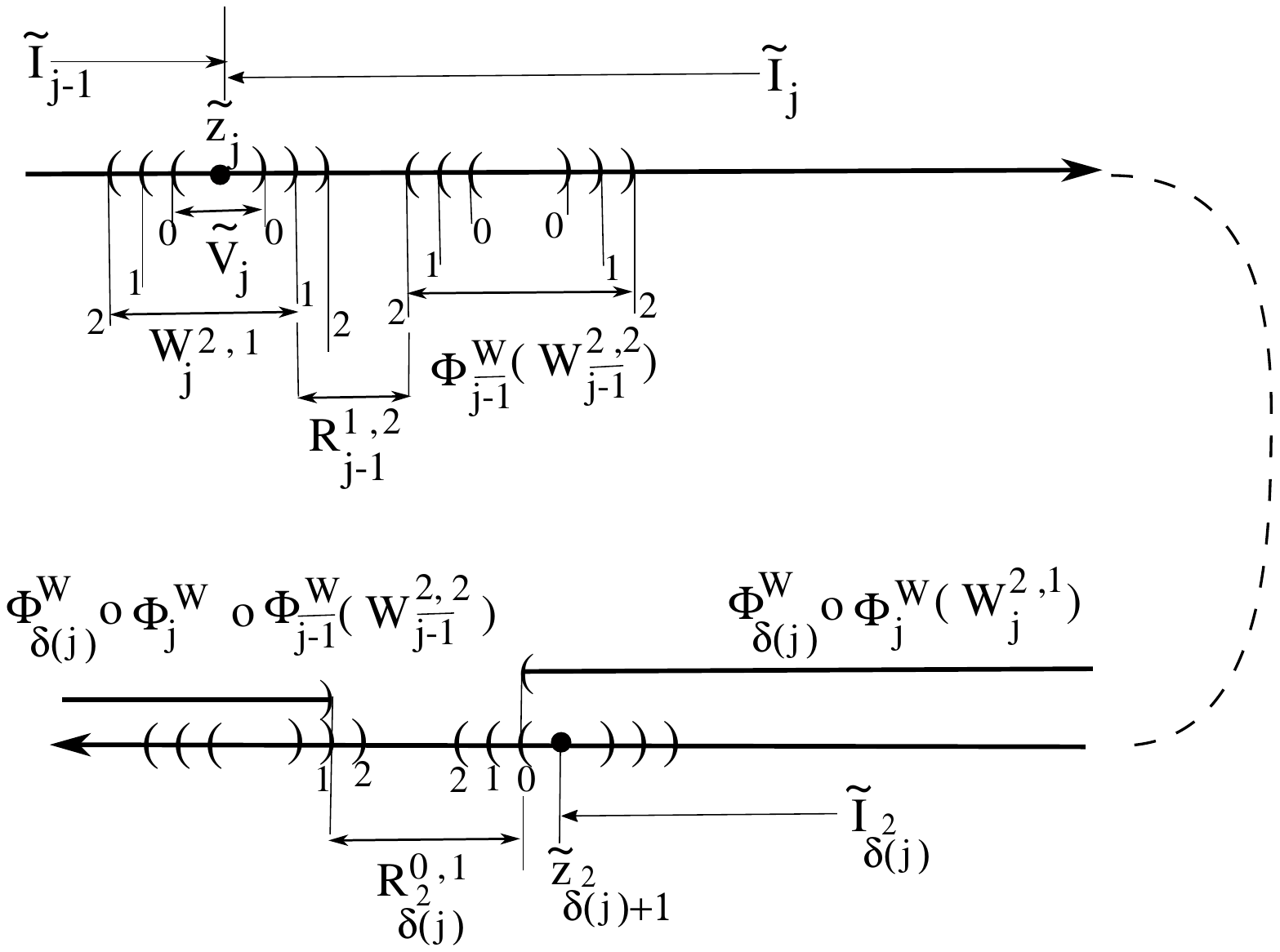}}
\vspace*{-0.25cm}
\caption{ Variation interval equalities for $k(j) =3$, with $\iota(j) = \overline{j}$
   and $\zeta^{\pm 1}(j) = j \pm 1$}
\label{VariationInterval}
\end{figure}

\noindent By definition (see (\ref{complement-W})), the variation interval $R^{\bold{a,b}}_{\zeta^{-1} (j)}$ in equality $(a)$ is the interval between 
$W^{p, \bold{a}}_j \textrm{ and } \widetilde{\Phi}^{W^*}_{\iota(\zeta^{-1} (j))} 
 [ W^{\bold{b}, q}_{\iota(\zeta^{-1} (j))} ]$, see Figure \ref{VariationInterval}, thus: 
\begin{equation}\label{Rab}
 R^{\bold{a,b}}_{\zeta^{-1} (j)}=[\partial^{+}(W^{p, \bold{a}}_j),\partial^{-}(\widetilde{\Phi}^{W^*}_{\iota(\zeta^{-1} )(j)} 
 [ W^{\bold{b}, q}_{\iota(\zeta^{-1} (j))})].
 \end{equation}
 These three intervals, by definition, are contained in 
  \textcolor{black}{$I^{p, \star}_{j}$}, and belong to the domain of definition of $\widetilde{\Phi}^{W^*}_j$.
   The image of $W^{p, \bold{a}}_j$ under $\widetilde{\Phi}^{W^*}_{j}$ is contained in 
   $\widetilde{I}_{\delta(j)}$ by Proposition \ref{gen-ext}-$(a)$ and thus, in the \textcolor{black}{domain} of definition of 
 $\widetilde{\Phi}^{W^*}_{\delta(j)}$ for any parameters  $(p_j,q_j)$.
 From the recursive definition of $W^{p, \bold{a}}_j$ in (\ref{Wpq}) we obtain
  $\widetilde{ \Phi}^{W^*}_{\delta(j)} \circ  \widetilde{\Phi}^{W^*}_{j} \left[  W^{p,\bold{a} }_j  \right] 
 \subset \textcolor{black}{I^{\star, \bold{a -1} }_{\delta^2 (j)} }$ and, in addition:
\begin{equation}\label{bord+}
 \partial^{+} \Big( \widetilde{\Phi}^{W^*}_{\delta(j)} \circ \widetilde{\Phi}^{W^*}_{j} 
 \left[  W^{p,\bold{a} }_j  \right] \Big) = 
 \partial^{+} \Big(\textcolor{black}{ I^{\star, \bold{a -1} }_{\delta^2 (j)}} \Big).
\end{equation} 
\noindent  The image of $\widetilde{\Phi}^{W^*}_{\iota(\zeta^{-1} (j))} 
 [ W^{\bold{b}, q}_{\iota(\zeta^{-1} (j))} ]$ by the same map:
$ \widetilde{\Phi}^{W^*}_{\delta(j)} \circ \widetilde{ \Phi}^{W^*}_{j}  \left[
 \widetilde{\Phi}^{W^*}_{\iota(\zeta^{-1} )(j)} [ W^{\bold{b}, q}_{\iota(\zeta^{-1} (j))} ] \right]$ is one side of 
 the equality $(b)$ in Proposition
 \ref{gen-ext}  for
 $ W^{\bold{b}, q}_{\iota(\zeta^{-1} (j))}$.
 It gives:\\
 $\widetilde{\Phi}_{\delta(j)} ^{W^*}\circ \widetilde{\Phi}_{j} ^{W^*}\circ
 \widetilde{\Phi}^{W^*}_{\iota(\zeta^{-1} (j))} [ W^{ \bold{b}, q}_{\iota(\zeta^{-1} (j))} ] =  
 \widetilde{\Phi}^{W^*}_{\iota (\delta^2(j))} \circ  \widetilde{\Phi}^{W^*}_{\gamma^2( \zeta^{-1}(j))}  \circ
\widetilde{\Phi}^{W^*}_{\gamma(\zeta^{-1} (j))} [ W^{\bold{b}, q}_{\iota(\zeta^{-1} (j))} ]. $
 
\noindent From (\ref{Wpq}) applied to  $W^{\bold{b}, q}_{\iota(\zeta^{-1} (j))}$  we have:
 $\widetilde{\Phi}^{W^*}_{\gamma^2( \zeta^{-1}(j))}  \circ
\widetilde{ \Phi}^{W^*}_{\gamma(\zeta^{-1} (j))} [ W^{ \bold{b}, q}_{\iota(\zeta^{-1} (j))} ] 
 \subset I^{\bold{b-1},q}_{\iota (\delta^2(j) )}$
  and:
   $ \partial^{-} \Big( \widetilde{\Phi}^{W^*}_{\gamma^2( \zeta^{-1}(j))}  \circ
\widetilde{\Phi}^{W^*}_{\gamma(\zeta^{-1} (j))} [ W^{\bold{b}, q}_{\iota(\zeta^{-1} (j))} ] \Big) 
 = \partial^{-}\Big(  I^{\bold{b-1},q}_{\iota (\delta^2(j))} \Big) $.
  Applying $\widetilde{\Phi}^{W^*}_{\iota (\delta^2(j))}$ on both sides of this equality gives:
\begin{equation}\label{bord-}
 \partial^{-} \Big( \widetilde{\Phi}^{W^*}_{\iota (\delta^2(j))} \circ \widetilde{\Phi}^{W^*}_{\gamma^2( \zeta^{-1}(j))}  \circ
 \widetilde{\Phi}^{W^*}_{\gamma(\zeta^{-1} (j))} [ W^{\bold{b} , q}_{\iota(\zeta^{-1} (j))} ] \Big)
 = \partial^{-} \Big( \widetilde{\Phi}^{W^*}_{\iota (\delta^2(j))} [ I^{\bold{b-1} ,q}_{\iota (\delta^2(j))} ] \Big).
\end{equation}
 \noindent Hence, from (\ref{Rab}), (\ref{bord+}) and (\ref{bord-}) we obtain:\\
$ \widetilde{\Phi}^{W^*}_{\delta(j)} \circ  \widetilde{\Phi}_{j}^{W^*}(R^{\bold{a,b}}_{\zeta^{-1} (j)})= [ \partial^{+} \Big(\textcolor{black}{ I^{\star, \bold{a -1} }_{\delta^2 (j)}} \Big), \partial^{-} \Big( \widetilde{\Phi}^{W^*}_{\iota (\delta^2(j))} [ I^{\bold{b-1} ,q}_{\iota (\delta^2(j))} ] \Big)]=R^{\bold{a -1, b -1} }_{\delta^2 (j)} ,$
 which is another formulation of the equality $(a)$ in the case $k(j) =3$.
 
\noindent   The equality $(b)$ of the Lemma is obtained exactly by the same arguments on the other side of the neighborhood $W^*_j$.
   The general case, for any $k(j)$, is obtained with the same arguments using $k(j) - 1$ and $k(j)$ compositions instead of 2 and 3 as above.
   \end{proof}
Lemma \ref{equality-variation} gives a pairing between the variation intervals:
$ R_{\zeta^{-1}(j)}^{p,q}$ and $ R_{\beta (\zeta^{-1}(j))}^{p-1,q-1}$ and between:
$ L_{j}^{p,q}$ and $ L_{\alpha (j)}^{p-1,q-1}$ around the neighborhood $W^*_j$. The combinatorics of the permutations 
$\alpha$ and $\beta$ of Lemma \ref{alpha-beta} are important for the next results and allow an induction process on each variation interval.

\begin{Lemma}\label{induction}
For each $j \in \{1,\dots, 2N\}$, let $r:= r_{\alpha}(j) \geq 1$ be the length of the cycle of $\alpha(j)$
and let $r':= r_{\beta}(\zeta^{-1}(j)) \geq 1$ be the length of the cycle of $\beta (\zeta^{-1}(j))$.\\
 There are two integers $K (r),\, K (r')  \geq k(j) -1$  such that, for all $p,q$:
 
$(a)$ $\widetilde{\Phi}^{-K(r)} (L_j^{p,q}) = L_j^{p+r,q+r}$ and  $ | L_j^{p+r,q+r} | = \lambda^{-K(r)}| L_j^{p,q} |$,

$(b)$ $\widetilde{\Phi}^{-K(r')} (R_{\zeta^{-1}(j)}^{p,q}) = R_{\zeta^{-1}(j)}^{p+r',q+r'}$ and 
$ | R_{\zeta^{-1}(j)}^{p+r',q+r'}| = \lambda^{-K(r')}| R_{\zeta^{-1}(j)}^{p,q} |$,\\
where $|\quad|$ is the length of the  metric intervals  on $S^1$ for which the map $\widetilde{\Phi}$ is affine.
\end{Lemma}
\begin{proof}
The two statements $(a)$ and $(b)$ are symmetric, let us focus on the case $(a)$.\\
 We start with the simple situation when the permutation $\alpha$ has a fixed point: $\alpha (j) = j$. In this case, Lemma \ref{equality-variation}-$(b)$  gives:

$L^{p, q}_{j} = (\widetilde{\Phi}^{W^*}_{\zeta^{-1} (j)})^{-1} \circ 
(\widetilde{\Phi}^{W^*}_{\gamma (\zeta^{-1} (j))})^{-1} 
\circ \dots \circ (\widetilde{\Phi}^{W^*}_{\gamma^{k(j) -2} (\zeta^{-1} (j))})^{-1} 
\Big[  L^{p-1, q-1 }_{j} \Big]$.\\
From the definition of the variation intervals in (\ref{complement-W}) we obtain: 
$L_j^{p,q} \subset \widetilde{I}_{\zeta^{-1} (j)}$ for all parameters $(p,q)$ and thus 
$\widetilde{\Phi}^{W^*}_{\zeta^{-1} (j)} (L_j^{p,q}) = \widetilde{\Phi}(L_j^{p,q})$. Proposition \ref{gen-ext}-$(a)$ gives, in a similar way: 
$\widetilde{\Phi}^{W^*}_{\zeta^{-1} (j)} (L_j^{p,q}) \subset \widetilde{I}_{\gamma (\zeta^{-1} (j))}$ and thus:
$\widetilde{\Phi}^{W^*}_{\gamma (\zeta^{-1} (j) )} \circ \widetilde{\Phi}^{W^*}_{\zeta^{-1} (j)} (L_j^{p,q}) = \widetilde{\Phi}^2(L_j^{p,q})$. By induction, using the same argument, we obtain:\\
\centerline{ $\widetilde{\Phi}^{W^*}_{\gamma^{k(j) -2} (\zeta^{-1} (j) )}\circ \dots \circ \widetilde{\Phi}^{W^*}_{\zeta^{-1} (j)} (L_j^{p,q}) = \widetilde{\Phi}^{k(j) - 1}(L_j^{p,q})$}

\noindent  and thus, from the equality in Lemma \ref{equality-variation}-$(b)$: $L_j^{p,q} = (\widetilde{\Phi}^{k(j) - 1})^{-1}(L_j^{p -1,q -1}) \subset L_j^{p -1,q -1}$.
This is the first part of property $(a)$ when $\alpha (j) = j$ i.e., when $r= 1$. The second statement in $(a)$ in this case is immediate: $ | L_j^{p,q} | = \lambda^{-K(1)}.| L_j^{p-1,q-1} |$, with $K(1) = k(j) -1$, since 
$\widetilde{\Phi}$ is affine of slope $\lambda$ for the metric $ | \quad |$.\\
In the general case, $\alpha (j)$ has a cycle of length $1 \leq r \leq 2N$, i.e., $\alpha^r (j) = j $ and $r$ is minimal for that property. Lemma \ref{equality-variation}-$(b)$ defines a sequence of maps:

\centerline{ $L_j^{p,q} \rightarrow L_{\alpha (j)}^{p-1,q-1} \dots\rightarrow L_{\alpha^r (j) = j}^{p-r,q-r}$,}
\noindent  where each arrow is a composition of 
$k(\alpha^m (j)) -1$ affine maps, for $0 \leq m \leq r-1$.
 From the above arguments, each map in this composition is 
$( \widetilde{\Phi})^{-1}$.
This gives the first statement in $(a)$ for the integer: $K (r) = \sum_{m=0}^{r-1} (k(\alpha^m (j)) -1) \geq k(j) -1$.
\textcolor{black}{The second statement in $(a)$ is immediate, as above}.
The statement $(b)$ is the same by exchanging the permutations $\alpha$ and $\beta$.
\end{proof}

\begin{Cor}\label{PeriodicPoints}
For every $j\in \{ 1, \dots, 2N\}$ and all 
$(p,q)$, there is a unique expanding periodic point $l_j^0 \in L_j^{p,q}$ of period $K( r_{\alpha}(j))$ and a unique periodic point 
$r_{\zeta^{-1}(j)}^0 \in R_{\zeta^{-1} (j)}^{p,q}$ of period $K( r_{\beta}(\zeta^{-1}(j)))$ of the map 
$\widetilde{\Phi}$.
\end{Cor}
\begin{proof} The proof is direct from the Lemma. Indeed, $\widetilde{\Phi}^{K( r_{\alpha}(j))}$ has a unique expanding fixed
point $l_j^0$ in $L_j^{p,q}$ for all $p,q$. This point is periodic of period $K( r_{\alpha}(j))$ under 
$\widetilde{\Phi}$
by the minimality $r$ as the length of the cycle $\alpha (j)$. The integer $K( r_{\alpha}(j))$ is the first iterate of return in 
$L_j^{p,q}$. The arguments are the same for the intervals $R_{\zeta^{-1}}^{p,q}$.
\end{proof}

\begin{Rm}\label{PeriodicOrbits}
Existence and uniqueness of periodic orbits are invariant under conjugacy and thus the original map $\Phi$ has periodic orbits as above.\\
The other observation comes from the combinatorics of the various permutations: $\delta, \gamma, \alpha, \beta$  as discussed in section \ref{combinatorics}.
 \textcolor{black}{ In particular, the cycle of $\alpha$ that permutes $j$ is associated to the orbit of the point $l_j^0 \in L_j^*$ and the orbit of the point 
 $r_{\iota(j)}^0 \in R_{\iota(j)}^* $ is associated to the cycle of  $\beta$ that permutes
 $\iota (j) $, }\textcolor{black}{by Lemma \ref{equality-variation}. From Lemma \ref{alpha-beta}, if the cycle of $\alpha(j)$ has length $r$ then $\beta(\iota (j) )$ is also in a cycle of length $r$ and the two $\widetilde{\Phi}$-orbits are in bijection.}
\end{Rm}

 Before going further, let us make some observations and fix some notations:\\
\centerline{set $W_{j}^{\infty}:=[l_j^0, r_{\zeta^{-1}(j)}^0] $ and $I_{j}^{\infty}:=[l_j^0, r_{j}^0]$, notice that:}
\begin{equation}\label{IntervalLimit}
\forall p, q \geq 1, {\textrm{  }} W_{j}^{p,q} \subset   W_{j}^{\infty}  {\textrm{ and }}
 I_{j}^{p,q} \subset  I_{j}^{\infty}.
\end{equation}

\begin{definition}\label{LimitHomeo}
For all $j \in \{1,\dots, 2N\}$, let $\varphi_{j}^{\infty} \in {\textrm{Homeo}}^+ (S^1)$ be defined as:\\
\centerline{ - $\varphi_{j}^{\infty} : I_{j}^{\infty} \rightarrow [r_{\iota(j)}^0, l_{\iota(j)}^0] $ is affine of slope $\lambda$,}
\centerline{ - $\varphi_{j}^{\infty} : [r_{j}^0, l_{j}^0 ] \rightarrow I_{\iota(j)}^{\infty}$ is affine of slope $\lambda^{-1}$,}
where each interval $[x, y]$ is from $x$ to $y$ along $S^1$ following the positive clockwise orientation.
The points $r_j^0$, $l_j^0$ are called the {\rm {breaking points}} of $\varphi_{j}^{\infty}$.
\end{definition}

The maps $\varphi_{j}^{\infty}$ are indeed homeomorphisms, they are continuous but not differentiable at the breaking points, where the left and right derivatives are in $\{ \lambda, \lambda^{-1} \}$. From the definition, the following properties are immediate:
\begin{equation}\label{varphiInverse}
\varphi_{\iota(j)}^{\infty} = ( \varphi_{j}^{\infty} )^{-1},
\end{equation}
\begin{equation}\label{restrictVarphi}
{\textrm{ for all }} p,q \geq 0,  {\textrm{  }} {\varphi_{j}^{\infty}}_{| \widetilde{I}_{j}^{p,q}} = {{\widetilde{\Phi}}_{j}}^{p,q} 
{\textrm{ and }} 
{\varphi_{j}^{\infty}} ( W_{j}^{p,q} ) = {\widetilde{\Phi}_{j}^{p,q}} ( W_{j}^{p,q} ).
\end{equation}
 
\begin{Lemma}\label{limit}
With the previous notations,  for all
$j \in \{1, \dots,2N\}$ and all
$f_{j}^{p,q} \in \textcolor{black}{[f_j^{W^{*}}]} $:
 $$ \lim_{p,q \rightarrow \infty} f_{j}^{p,q} = \varphi_{j}^{\infty}$$ 
for the topology of uniform convergence in ${\rm{Homeo}}^+ (S^1)$.
\end{Lemma}

\begin{proof}
 The notations in the Lemma are already simplified, we simplity further to avoid confusion on the indices by assuming that all the indices $(p_j,q_j)$ are equal to a single index $p$. For each integer $p \geq 1$, the observation (\ref{restrictVarphi}) implies that the homeomorphisms 
$f_{j}^{p,p}$ and  $\varphi_{j}^{\infty}$ are different only on the compact sets $L_j^{p,p}$ and $R_j^{p,p}$. 
   The image of these compact sets being $R_{\iota(j)}^{p,p}$ and $L_{\iota(j)}^{p,p}$. Each of these compact intervals converge, when $p$ tends to infinity, in the metric topology of $S^1$, to a periodic point, namely to $ l_{j}^0, r_{j}^0, r_{\iota(j)}^{0},  l_{\iota(j)}^{0}$, at an exponential rate by Lemma \ref{induction}. Thus, for the uniform convergence topology on 
${\textrm{Homeo}}^+ (S^1)$, the limit exists and is equal to  $\varphi_{j}^{\infty}$.
\end{proof}

The goal now is to study some properties of the limit homeomorphisms. The next result is central for our global goal:  to define a group from the map $\Phi$.

\begin{Lemma}\label{PartitionVarphi}
  For each cutting point $\widetilde{z}_j$ of $\widetilde{\Phi}$, there is a partition of $S^1$ into $2k(j)$ intervals:
$$S^1 = A^{\infty}_0 \bigcup A^{\infty}_{k(j)} \bigcup_{m = 1}^{k(j) - 1} A_m^{\infty,\pm}, 
    \textrm{ with  }  A^{\infty}_0 = A_0^{\infty,\pm} \textrm{ and } A^{\infty}_{k(j)} = A_{k(j)}^{\infty, \pm} ,$$
\noindent on which the two compositions:

   \vspace{5pt}
\centerline{ $ \Psi^{+}_j := \varphi^{\infty}_{\delta^{k(j)-1}(j)} \circ \dots \circ \varphi^{\infty}_{\delta(j)} \circ \varphi^{\infty}_{j} \;$ and 
   $\; \Psi^{-}_j := \varphi^{\infty}_{\gamma^{k(j)-1}(\zeta^{-1}(j))} \circ \dots \circ \varphi^{\infty}_{\gamma(\zeta^{-1}(j))} 
   \circ \varphi^{\infty}_{\zeta^{-1}(j)}$,}
 
  \vspace{5pt}  

\noindent  satisfy the following properties:
   \vspace*{-3pt}
     \begin{itemize}[noitemsep, leftmargin=20pt]
 \item[$(a)$] $ ( \Psi^{+}_j)_{| A_{m}^{\infty,\pm }} =  ( \Psi^{-}_j)_{| A_{m}^{\infty,\pm }}$, are affine maps of slope 
   $\lambda^{k(j) - 2m}$ for each $m \in \{ 0,\dots, k(j) \}$.
 \item[$(b)$] Two adjacent intervals $A^{\infty,\pm}_{m}$ along $S^1$ intersect either at a breaking point or a point in a
 $\widetilde{\Phi}$-orbit of a breaking point.
 \end{itemize}
    \end{Lemma}
    
   \begin{proof}
   To start the proof let us simplify a little bit the notation. Consider a single integer $p\geq 1$ and the neighborhoods
   $W^p_i := W^{p,p}_i$ for $i$ in the cycle of $\delta(j)$. We denote $\overline{j} = \iota (j)$ and 
   $j\pm 1 $ for $\zeta^{\pm 1} (j)$, we also denote $\Phi^p_j$ instead of $\widetilde{\Phi}^{W^{p,p}}_j$.\\ 
   By Proposition \ref{gen-ext}-$(a)$, the following intervals are disjoint and ordered along $S^1$:
   
          \begin{equation}\label{PartitionAm}
  \begin{array}{l}
 \textrm{$[+]$} \; \; \,  A_{0}^p := W_j^p,\\
 \quad \quad A_{1}^{p,+} := \Phi^p_{\overline{j-1}} (W_{\overline{j-1}}^p), \\
 \quad \quad \quad \vdots \\
  \quad \quad A_{k(j)-1}^{p,+}: = \Phi^p_{\overline{j-1}} \circ \Phi^p_{\overline{\gamma(j-1)}} \circ \dots 
  \circ \Phi^p_{\overline{\gamma^{k(j)-2}(j-1)}}\big(W_{\overline{\gamma^{k(j)-2}(j-1)}}^p\big),\\
 \quad \quad  A_{k(j)}^{p} := \Phi^p_{\overline{j-1}} \circ \dots  
  \circ \Phi^p_{\overline{\gamma^{k(j)-1}(j-1)}}\big(W_{\overline{\gamma^{k(j)-1}(j-1)}}^p\big),\\ [0.85em]
   \textrm{$[-]$} \; \; \,  A_{k(j)-1}^{p,-} := \Phi^p_{\overline{j}} \circ \Phi^p_{\overline{\delta(j)}} \circ \dots 
  \circ \Phi^p_{\overline{\delta^{k(j)-2}(j)}}\big(W_{\textcolor{black}{\delta^{k(j)-1}(j)}}^p\big),\\
 \quad \quad \quad  \vdots \\
  \quad \quad  A_{1}^{p,-} := \Phi^p_{\overline{j}} (W_{\textcolor{black}{\delta (j)}}^p).
   
   \end{array}
   \end{equation} 
  These intervals $A_m^{p,\pm}$ are disjoint on $S^1$ and  separated by either some variation intervals or some images of variation intervals. More precisely, we obtain the following intervals, ordered along $S^1$ after $A_0^p$, with  $R^p_i:=R^{p,p}_i$ and $L^p_i:=L^{p,p}_i $
   (see Figure \ref{Lemm3.5}):

   \begin{equation}\label{PartitionDm}
  \begin{array}{l}
 \textrm{$[+]$} \; \; \, D_1^{p,+} := R_{j-1}^p,\\
 \quad \quad D_2^{p,+} := \Phi^p_{\overline{j-1}} ( R_{\gamma(j-1)}^p),\\
\quad \quad \quad  \vdots \\
 \quad \quad D_{k(j)}^{p,+} := \Phi^p_{\overline{j-1}} \circ \Phi^p_{\overline{\gamma(j-1)}} \circ \dots  
  \circ \Phi^p_{\overline{\gamma^{k(j)-2}(j-1)}} \big( R_{\gamma^{k(j)-1}(j-1)}^p\big),\\ [0.85em]
    \textrm{$[-]$} \; \; \, D_{k(j)}^{p,-} := \Phi^p_{\overline{j}} \circ \Phi^p_{\overline{\delta(j)}} \circ \dots 
  \circ \Phi^p_{\overline{\delta^{k(j)-2}(j)}}(L_{\delta^{k(j)-1}(j)}^p),\\
 \quad \quad \quad   \vdots \\ 
  \quad \quad D_{2}^{p,-} := \Phi^p_{\overline{j}}(L_{\delta(j)}^p),\\
  \quad \quad  D_{1}^{p,-} := L_{j}^p.
 \end{array}
   \end{equation}  
   
 \noindent   The proof of the Lemma has two steps at this point:\\
   (A) To prove that  $ ( \Psi^{+}_j)_{| A_{m}^{p,\pm}} =  ( \Psi^{-}_j)_{| A_{m}^{p,\pm}}$, are affine maps of slope 
   $\lambda^{k(j) - 2m}$ for each\\
    $m \in \{ 0,\dots, k(j) \}$ and all $p \geq 1$.\\
   (B) To take the limit when $p \rightarrow \infty$.

   Step (A) is obtained as a consequence of the equality $(b)$ in Proposition \ref{gen-ext}, applied to each $W_i^p$ for $i$ in the cycle of $\delta (j)$.\\
 $\bullet$  Let us start by the simplest situation: $A_0^p = W_j^p$. In this case the equality $(b)$ of  Proposition \ref{gen-ext}
   gives: 
   $(\Phi^p_{\delta^{k(j)-1}(j)} \circ \dots \circ \Phi^p_{\delta(j)} \circ \Phi^p_{j})_{| W_j^p} =
   (\Phi^p_{\gamma^{k(j)-1}(j-1)} \circ \dots \circ \Phi^p_{\gamma(j-1)} \circ \Phi^p_{j-1})_{| W_j^p}$.\\
 This map is affine of slope $\lambda^{k(j)}$  since it is a composition of $k(j)$ affine maps each of slope $\lambda$.
   By the observation (\ref{restrictVarphi}): 
   $(\varphi_{i}^{\infty})_{ | W_{i}^{p} } = (\Phi_{i}^{p})_{ | W_{i}^{p} }$ for all $i$ and all $p$.
   We replace, in the equality above, each $\Phi_{i}^{p}$ by $\varphi_{i}^{\infty}$ and we obtain the identity:
   $ ( \Psi^{+}_j)_{| A_{0}^{p}} =  ( \Psi^{-}_j)_{| A_{0}^{p}}$ which is the property (A) for $m=0$.\\
    $\bullet$ The next simplest case is for the interval $A_{k(j)}^{p}$. The equality $(b)$ of  Proposition \ref{gen-ext} in this case gives:\\
     $A_{k(j)}^{p} = \Phi^p_{\overline{j-1}} \circ \dots 
  \circ \Phi^p_{\overline{\gamma^{k(j)-1}(j-1)}}\Big(W_{\overline{\gamma^{k(j)-1}(j-1)}}^p\Big) =
  \Phi^p_{\overline{j}} \circ \dots 
  \circ \Phi^p_{\overline{\delta^{k(j)-1}(j)}}\Big(W_{\overline{\gamma^{k(j)-1}(j-1)}}^p\Big) $. We observe that on this interval, each homeomorphism in both compositions in $ \Psi^{+}_j$ and $ \Psi^{-}_j$ are affine of slope $\lambda^{-1}$ and we check that $ \Psi^{+}_j (A_{k(j)}^{p}) = W_{\overline{\gamma^{k(j)-1}(j-1)}}^p$ and 
  $ \Psi^{-}_j (A_{k(j)}^{p}) = W_{\overline{\gamma^{k(j)-1}(j-1)}}^p$ since at each step in both compositions we compose a map with its inverse. We obtain the equality: 
  $ (\Psi^{+}_j)_{ | A_{k(j)}^{p}} = (\Psi^{-}_j)_{ | A_{k(j)}^{p}}$ and the map is affine of slope $\lambda^{-k(j)}$, which is the case $m=k(j)$ in (A).\\
  $\bullet$ For the other intervals $A_m^{p,\pm}$ for $m \in \{ 1, \dots , k(j)-1 \}$ the proofs are essentially the same.
\textcolor{black}{  Each interval $A_m^{p,\pm}$ is a composition of $m$ maps applied to an interval $W_i^p$
  for $i$ in the cycle of $\delta(j)$.} Proposition \ref{gen-ext}-$(b)$  provides  equalities for a composition of $k(j)$ affine maps on the respectives $W_i^p$. We apply  $ \Psi^{+}_j$ and $ \Psi^{-}_j$ on both sides of $A_m^{p,\pm}$ and we observe that, among the $k(j)$ homeomorphisms of each composition, $m$ of them are affine of slope
  $\lambda^{-1}$ and $k(j)-m$ are affine of slope $\lambda$.\\
  Let us give some details for the interval $A_1^{p,+} = \Phi^p_{\overline{j-1}} (W^p_{\overline{j-1}})$.
  The equality $(b)$ of Proposition \ref{gen-ext} gives:\\
  $(\Phi^p_{\delta^{k(j)-2}(j)} \circ \dots \circ \Phi^p_{j} \circ \Phi^p_{\overline{j-1}} )_{| W_{\overline{j-1}}^p} =
   (\Phi^p_{\overline{\delta^{k(j)-1}(j)}} \circ \Phi^p_{\gamma^{k(j)-1}(j-1)} \circ \dots
   \circ \Phi^p_{\gamma(j-1)} )_{| W_{\overline{j-1}}^p}.$\\
   We apply now the two compositions $ \Psi^{+}_j$ and $ \Psi^{-}_j$  on $A_1^{p,+}$  and we use the above equality, with a little  computation we obtain:\\
 \centerline{  $ \Psi^{+}_j (A_1^{p,+}) =  \Psi^{-}_j (A_1^{p,+}) = \Phi^p_{\gamma^{k(j)-1}(j-1)} \circ \dots
   \circ \Phi^p_{\gamma(j-1)} ( W_{\overline{j-1}}^p)$.}\\
    We observe that on $A_1^{p,+} $ the two compositions are affine of  slope $\lambda^{k(j) -2}$. This is the case (A) for $m=1$.\\
    Let us now turn to the step (B) above and study the limit $p \rightarrow \infty$.
   The collection of intervals given in (\ref{PartitionAm}) and (\ref{PartitionDm}) defines the following partition of $S^1$:
   \begin{equation}\label{PartitionAD}
   S^1 = A^p_0 \bigcup A^p_{k(j)} \bigcup_{m = 1}^{k(j) - 1} A_{m}^{p,\pm} 
   \bigcup_{m = 1}^{k(j)} D_{m}^{p,\pm}
   \end{equation}
    \begin{figure}[h!]
\centerline{\includegraphics[height=70mm]{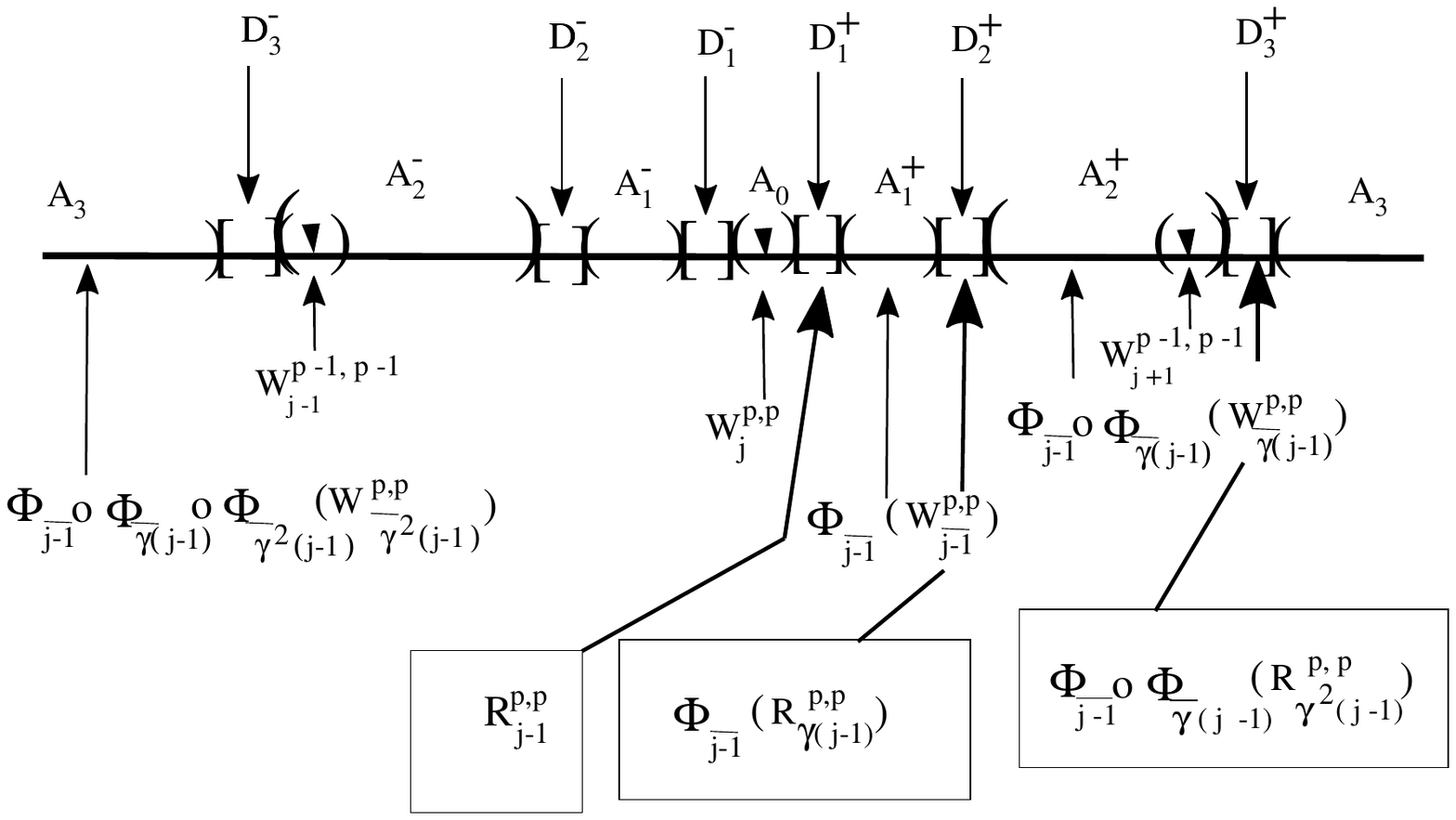} }
\vspace*{-1cm}
\caption{  The partition (\ref{PartitionAD}) in the proof of Lemma \ref{PartitionVarphi} for $k(j) =3$.}
\label{Lemm3.5}
\end{figure}
    From Lemma \ref{induction} and Corollary \ref{PeriodicPoints}, each variation interval converges to a periodic point of $ \widetilde{\Phi}$ when $p \rightarrow \infty$ : $R_i^p \rightarrow r_i^0$ and $L_i^p \rightarrow l_i^0$.
   These periodic points are also, by definition, the breaking points of the homeomorphisms  $\varphi_{i}^{\infty}$.
      From the definition of the intervals $D_{m}^{p,\pm}$ in (\ref{PartitionDm}) we obtain the following limit points:
       \begin{equation}\label{limitD}
      \begin{array}{l}
 \textrm{$[+]$ }  D_1^{p,+} \rightarrow r_{j-1}^0,\; D_2^{p,+} \rightarrow \widetilde{\Phi} \Big(r_{\gamma(j-1)}^0\Big),\dots,\;
D_{k(j)}^{p,+} \rightarrow \widetilde{\Phi}^{k(j)-1} \Big(r_{\gamma^{k(j)-1}(j-1)}^0\Big),\\ 
 
 \textrm{$[-]$ } D_{k(j)}^{p,-} \rightarrow \widetilde{\Phi}^{k(j)-1} \Big(l_{\delta^{k(j)-1}(j)}^0\Big), \dots ,\; D_2^{p,-} \rightarrow \widetilde{\Phi} \Big(l_{\delta(j)}^0\Big),\; D_1^{p,-} \rightarrow l_{j}^0.
           \end{array}
      \end{equation}
   The definition of the intervals $A_m^{p, \pm}$ in (\ref{PartitionAm}) gives, by taking the limit 
  $p \rightarrow \infty$, that each interval $W_i^p \rightarrow W_i^{\infty}$ as defined in (\ref{IntervalLimit})
  for the metric convergence on $S^1$. Lemma \ref{limit} implies that on each interval 
  $W_i^p$ the image interval $\widetilde{\Phi}_i^p (W_i^p)$ converges to
  $\varphi_{i}^{\infty} (W_i^{\infty})$. Putting these convergences together we obtain finally that each interval 
  $A_m^{p, \pm}$ converges to $A_m^{\infty, \pm}$ and is the limit of the corresponding composition of $m$ homeomorphisms $\varphi_{i}^{\infty}$.
  The partition (\ref{PartitionAD}) thus defines a limit partition when $p \rightarrow \infty$:
  \begin{equation}\label{PartitionLimit}
   S^1 = A^{\infty}_0 \bigcup A^{\infty}_{k(j)} \bigcup_{m = 1}^{k(j) - 1} A_{m}^{\infty,\pm}. 
   \end{equation}
$\bullet$ Two adjacent intervals of that partition intersect in a limit point given by (\ref{limitD}), i.e., in a point which is either a breaking point of $\varphi_{i}^{\infty}$ or in a $\widetilde{\Phi}$-orbit of a breaking point.
This is statement $(b)$ of the Lemma.\\
$\bullet$ To complete the proof of the Lemma we have to check the equality:\\
$ ( \Psi^{+}_j)_{| A_{m}^{\infty,\pm}} =  ( \Psi^{-}_j)_{| A_{m}^{\infty, \pm}}$, and that each composition is an affine map of slope 
   $\lambda^{k(j) - 2m}$, for each $m \in \{ 0,\dots, k(j) \}$ and each $j$.
   These properties are obtained by taking the limit $p \rightarrow \infty$ in the equalities in  (A) above.
  \end{proof}

The limit homeomorphisms of Lemma \ref{limit} together with the partition and the properties of Lemma \ref{PartitionVarphi} are the main steps in our goal to construct a group from our initial map $\Phi$.
Observe that the limit homeomorphisms are well defined and unique from the dynamical properties of the map $\Phi$.


\begin{theo}\label{CP-relations}
Let $\Phi$ be a piecewise homeomorphism of $S^1$ satisfying the ruling conditions: $\rm (EC), (E\pm), (CS\textrm{-}\lambda)$, for some $\lambda> 1$ and let $\varphi_{i}^{\infty} \in {\rm{Homeo}}^+ (S^1)$, $ i \in \{1, \dots, 2N\}$ be the set of homeomorphisms of Definition \ref{LimitHomeo},
then each cutting point $z_j$ of $\Phi$ defines an equality in ${\rm Homeo}^+ (S^1)$, called a  cutting point relation:

\vspace{3pt}
\noindent (CPj) \centerline{  $ \varphi^{\infty}_{\delta^{k(j)-1}(j)} \circ \dots\circ \varphi^{\infty}_{\delta(j)} \circ \varphi^{\infty}_{j} 
= \varphi^{\infty}_{\gamma^{k(j)-1}(\zeta^{-1}(j))} \circ \dots \circ \varphi^{\infty}_{\gamma(\zeta^{-1}(j))} 
   \circ \varphi^{\infty}_{\zeta^{-1}(j)}. $ }
\end{theo}
\begin{proof}
From Lemma \ref{PartitionVarphi} there is a partition  of $S^1$ for each index 
$j \in \{ 1, \dots, 2N\}$ so that the equality {\it (CPj)} is satisfied on each partition interval 
$A_m^{\infty, \pm}$ for $m \in \{ 0, \dots, k(j) \}$ since 
$ ( \Psi^{+}_j)_{| A_{m}^{\infty,\pm}} =  ( \Psi^{-}_j)_{| A_{m}^{\infty, \pm}}$. Each composition is affine of slope
$\lambda^{k(j)-2m}$ on each such interval $A_m^{\infty, \pm}$. The intersection of two consecutive intervals is a breaking point of some $\varphi_j^{\infty}$ or an image of a breaking point. At these points the two compositions are not differentiable, they are affine of slope $\lambda^{k(j)-2m}$ on one side and of slope $\lambda^{k(j)-2(m \pm 1)}$ on the other side. By continuity the two compositions are equal at each such extreme point of the partition intervals. Thus the equality {\it(CPj)} is satisfied on $S^1$: this is a relation on $ {\textrm{Homeo}}^+ (S^1)$.
\end{proof}

Observe that for two indices $j$ and $j'$ in the same cycle of the permutation $\delta$, the two relations {\it(CPj)} and {\it(CPj')} differ by a cyclic permutation of the indices along the cycle of $\delta$. These two relations are conjugate in 
$ {\textrm{Homeo}}^+ (S^1)$. This means that the number of non-conjugate cutting point relations is the \textcolor{black}{number of cycles} of the permutation $\delta$.


At this point we have all the tools to define a group from our map $\Phi$.
 
 \begin{definition}\label{groupGPhi} Let $\Phi$ be a piecewise homeomorphism of $S^1$ satisfying the 
 ruling conditions $\rm (EC), (E\pm),  (CS\textrm{-}\lambda),$ for some $\lambda > 1$. Let
 $G_{X_{\Phi}}<{\rm Homeo}^+ (S^1)$ be generated by the set of  homeomorphisms 
 $X_{\Phi} := \{ \varphi_j^{\infty} : j \in \{ 1, \dots, 2N\}  \}$ of Definition \ref{LimitHomeo}. These generators verify, in particular, all the cutting point relations {\it(CPj)} of Theorem \ref{CP-relations}. 
  \end{definition}   
       
\noindent  The set of generators $X_{\Phi}$ is well defined from the map $\Phi$ and the limit process of Lemma 
\ref{limit} remove all the choices that were made in the first steps of the construction, i.e., for the diffeomorphisms 
$f_j^{W^*}$. \textcolor{black}{The goal for the rest of the paper is to study the group $G_{X_{\Phi}}$, and its action on $S^1$ using the dynamics of the map 
$\Phi$.}
}
    
 \section{Some metric spaces associated to  $\Phi$}
 The group $G_{X_{\Phi}}$ of Definition \ref{groupGPhi} is obtained from the map
 $\Phi$. The classical strategy to study the geometry of such groups is via a {\em geometric action} on a well chosen metric space. Unfortunately no ``natural" metric space is given here so we have to construct one from the given data i.e., the dynamic of the map $\Phi$.\\
  The goal of this section is to define a metric space suited to the class of maps $\Phi$ of Section $\S$2.
  The construction of an action will be given in the next section.
  
    \vspace{3pt}
 In the following we will not distinguish between the maps $\Phi$ and $\widetilde{\Phi}$ nor between the partition intervals $I_j$ and 
 $\widetilde{I}_j$.
   \subsection{A first space: $\Gamma^0_{\Phi}$}\label{gamma0}
 The first space we consider is directly inspired by a construction due to P. Haissinsky and K. Pilgrim 
 \cite{HP} (see also \cite{H18}) in the context of coarse expanding conformal maps. In these papers, the authors use the dynamics of a map $F$ on a compact metric space $Y$. They construct a graph out of a sequence of coverings of the space $Y$ by open sets obtained from one covering by the sequence of pre-image coverings. They prove that if the map is ``expanding", in a topological sense, then the resulting space is Gromov hyperbolic with boundary the space $Y$. 

\noindent  We use the same idea where the space is $S^1$ and the dynamic is given by   $\Phi$. \\We replace their coverings by our partition and their sequence of pre-image coverings by the sequence of pre-image partitions. In order to fit with this description we use a partition by closed intervals, so that adjacent intervals do intersect in the simplest possible way i.e., points. \textcolor{black}{With our notations this gives:}\\
  $ S^1 = \bigcup _{j=1}^{2N} I_j, \textrm{ with } I_j = [  z_j, z_{\zeta(j)} ]$, keeping the same notation for simplicity. 
Thus, each interval $I_j$ intersects the two adjacent intervals $I_{\zeta^{\pm 1}(j)}$ exactly at  cutting points.
\begin{definition}
  \noindent We define a graph $\Gamma^0_{\Phi}$, from each map $\Phi$ of section $\S$2, by an iterative process (see Figure \ref{fig:6}):
  
\noindent   $\bullet$ Level 0:
  A base vertex $v_0$ is defined.\\
$\bullet$   Level 1:\\
     $(a)$ Each interval $I_j$ of the partition  defines  a vertex $v_j$.\\
  $(b)$  $v_0$ is connected to $v_j$ by an edge.\\
  $(c)$   $v_j$ is connected to $v_k$ if $I_j \neq I_k $ and $I_j \cap I_k \neq \emptyset$.\\
$\bullet$ Level 2:\\
 $(a)$  A vertex $v_{j_1, j_2}$ is defined for each non-empty connected component (that is not a point) 
  $I_{j_1, j_2} := I_{j_1} \cap \Phi^{-1} (I_{j_2})$.
 This notation is unambiguous since $\Phi^{-1} (I_{j_2})$ has at most one connected component in $I_{j_1}$,
\textcolor{black}{ by condition (SE)}.\\
    $(b)$  $v_{j_1}$ is connected to $v_{j_1, j_2}$ by an edge.\\
 $(c)$  $v_{j_1, j_2}$ is connected to $v_{j'_1, j'_2}$ if  $I_{j_1, j_2} \neq  I_{j'_1, j'_2} $ and 
  $I_{j_1, j_2} \cap I_{j'_1, j'_2} \neq \emptyset$.
   \begin{figure}[htbp]
\centerline{\includegraphics[height=35mm]{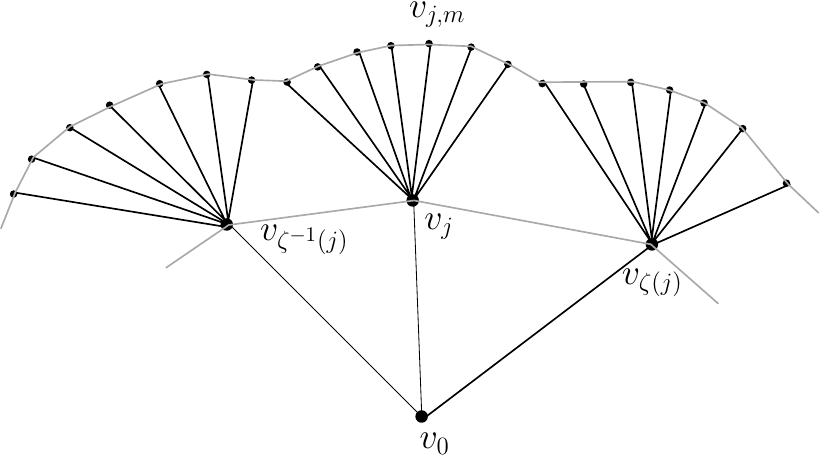} }
\caption{The first levels of the graph $\Gamma^0_{\Phi}$}
\label{fig:6}
\end{figure}

\noindent $\bullet$ Level k:\\
  $(a)$ We repeat level 2 by iteration i.e., we consider a sequence of intervals\\ 
  $ \{ I_{j_1} ; I_{j_1, j_2} ; \dots ; I_{j_1, j_2, \dots, j_k} ; \dots \}$ such  that:\\ 
  \centerline{$I_{j_1, j_2, \dots ,j_k} := I_{j_1, j_2, \dots, j_{k-1}} \cap \Phi^{-k+1} (I_{j_k})\neq \emptyset .$}
   Notice that if the sequence $j_1, j_2, \dots ,j_k$ defines an interval of level $k$, then  
   $j_{i +1} \neq \overline{j_{i}}$, for $1\leq i \leq k-1$, from condition (SE).\\
   $(b)$ A vertex $v_{j_1, j_2, \dots ,j_k}$
  is defined by the interval $I_{j_1, j_2, \dots ,j_k}$,\\
   $(c)$ $v_{j_1, j_2, \dots , j_k}$ is connected to $v_{j_1, j_2, \dots , j_{k-1}}$ by an edge,\\
  $(d)$ $v_{j_1, j_2, \dots, j_k}$ is connected to $v_{j'_1, j'_2, \dots, j'_k}$ if:\\ 
\centerline{  $  I_{j'_1, j'_2, \dots, j'_k}\neq I_{j_1, j_2, \dots, j_k}$
  and 
   $I_{j_1, j_2, \dots, j_k} \cap I_{j'_1, j'_2, \dots, j'_k}\neq \emptyset.$}
\end{definition}
 \begin{Lemma}\label{Graph0Hyper}
  If $\Phi$ is a piecewise homeomorphism of $S^1$ satisfying the condition {\rm(SE)}  and {\rm(CS)} then the graph 
  $\Gamma^0_{\Phi}$, endowed with the combinatorial metric (each edge has length one), is Gromov hyperbolic with boundary $S^1$.
   \end{Lemma} 
   
  \begin{proof}  We adapt word for word the proof in \cite{HP}. Indeed, the essential ingredients for the proof in \cite{HP} are the facts that each vertex is associated to a connected component of the pre-image cover with two properties:
  
 \noindent $ \bullet$ {\em Each component has a uniformly bounded number of pre-images}.\\
  In our case, each interval has at most $2N - 1$ pre-images \textcolor{black}{and at least $2N-2$.}
  
\noindent $ \bullet$ {\em The size of each connected component goes to zero when the level goes to infinity}.\\
 In our case, the size of the intervals $I_{j_1, j_2, \dots,j_k}$ in the sequence of pre-images goes to zero when $k$ goes to infinity by the expansivity properties (SE) and (CS).
 
 \noindent  In fact a much weaker expansivity property than our conditions (SE) and (CS) would be enough to conclude that the graph is hyperbolic.
  Observe that the distance of any vertex to the base vertex is simply the level $k$ and the edge connecting 
  $v_{j_1, j_2, \dots, j_k}$ to $v_{j'_1, j'_2, \dots, j'_k}$, if any, belongs to the sphere of radius $k$ centred at the base vertex. By this observation and our definition of the edges, each sphere of radius $k$ centred at the based vertex is homeomorphic to $S^1$. Therefore the limit space when $k$ goes to infinity is homeomorphic to $S^1$ and the Gromov boundary 
  $\partial \Gamma^0_{\Phi}$ is homeomorphic to $S^1$.
  \end{proof}

\subsection{The dynamical graph $\Gamma_{\Phi}$}\label{DG}

Consider the tree $T_{\Phi}$ obtained from $\Gamma^0_{\Phi}$  by removing the edges on the spheres. 
We define on $T_{\Phi}$   an equivalence relation that identifies some vertices on some of the spheres using the specific properties (EC), (E$\pm$) of the map $\Phi$.

\noindent For $T_{\Phi}$ we use  the same definitions  for the intervals and vertices of Level $0$, Level $1: (a), (b)$, Level $2: (a), (b)$ and Level $k: (a), (b), (c)$ as in $\Gamma^0_{\Phi}$.

The equivalence relation we define below is inductive, due to the dynamical origin of the space we construct. This is not a standard approach in group theory.

\vspace{5pt}
  \noindent  {\bf Labeling the edges:}
   The edge connecting the two vertices  $(v_{j_1, j_2, \dots , j_{k-1}},\,v_{j_1, j_2, \dots , j_{k}})$ is labelled by
   a symbol $\Psi_{j_k}$
   and the reverse edge, i.e., the same edge but read from $v_{j_1, j_2, \dots , j_{k}}$,
    is labelled  $\Psi_{\overline{j_k}}$.
       
    \begin{definition}\label{graph} 
    
     The dynamical graph is defined by $\Gamma_{\Phi} := T_{\Phi} / \sim_{\Phi}$, where 
  $ \sim_{\Phi}$ is the following relation:
  
  \vspace*{5pt}
 \noindent  $(\mathscr{V})$ Two vertices of the same level $k >1$ in $T_{\Phi}$:  
  $v_{j_1, j_2, \dots , j_{k}}$ and $v_{l_1, l_2, \dots , l_{k}}$  are identified if: 
     \begin{itemize}[noitemsep, leftmargin=25pt]
  \item[$(a)$] There is a level $ 0 \leq r < k-1 $ such that:
  \begin{itemize}[noitemsep, leftmargin=10pt]
  \item[$(a1)$] $ I_{j_1, \dots , j_i} = I_{l_1, \dots , l_i}$ as intervals in $S^1$, for $i=1,\dots,r$ (if $r=0$ the vertex is $v_0$).
  \item[$(a2)$] For all $1\leq p < k-r$, the intervals 
  $ I_{j_1, \dots , j_r, j_{r+1}, \dots, j_{r+p}}$ and  
  $I_{l_1, \dots , l_r , l_{r+1},\dots,l_{r+p} }$ are adjacent in the cyclic ordering of $S^1$.
 
     \end{itemize}
  \item[$(b)$] At level $k$: the intervals $  I_{j_1, \dots , j_{k}}$ and $  I_{l_1, \dots , l_{k}} $ are adjacent along $S^1$ and:
  \begin{itemize}[noitemsep, leftmargin=10pt]
  \item[$(b1)$] 
  $\Phi^m (  I_{j_1, \dots , j_{k}} ) \bigcap \Phi^m (   I_{l_1, \dots , l_{k}} ) = \emptyset $, for all 
  $r < m < k$,
  \item[$(b2)$] $\Phi^k (  I_{j_1, \dots , j_{k}} ) \bigcap \Phi^k (   I_{l_1, \dots , l_{k}} ) = $ one point and\\
  $\Phi^k (  I_{j_1, \dots , j_{k}} ) \bigcup \Phi^k (   I_{l_1, \dots , l_{k}} ) = $ a non-degenerate interval.
  \end{itemize}
  \end{itemize}
$(\mathscr{E})$ Two edges, connecting vertices from a level $m$ to level $m+1$, with the same label and starting from an identified vertex at level $m$ are identified to an edge labeled with the common label. \textcolor{black}{This identification extends up to the terminal vertices of both edges to give a single vertex.}
   \end{definition}
   
   We have to verify this dynamical graph is well defined for our class of maps and to find an interpretation of the vertices and edges of the graph. 

  \begin{Lemma}\label{vertexInterval}
  If $\Phi$ is a piecewise homeomorphism of $S^1$ satisfying the conditions {\rm(SE)}, {\rm(EC)}, {\rm(E$\pm$), (CS)} then the dynamical graph $\Gamma_{\Phi}$ is well defined.\\
In addition every vertex 
$w\neq v_0$ in $\Gamma_{\Phi}$, in a level $k\geq 1$, is associated to an interval  of $S^1$ of the following types: 
  \vspace*{-5pt}
     \begin{enumerate}[noitemsep, leftmargin=20pt]
   \item[$(i)$] $ I_w =  I_{j_1, \dots, j_{k}}$, \textrm{ or }
    \item[$(ii)$] $ I_w =  I_{j^1_1, \dots , j^1_{k}} \bigcup  I_{j^2_1, \dots , j^2_{k}} \dots \bigcup  I_{j^n_1, \dots , j^n_{k}}$, for some integer $ n = n(k, \Phi) $. The 
    intervals  $I_{j^i_1, \dots , j^i_{k}} $  belong to the same level $k$ and are pairewise adjacent along $S^1$.
  \end{enumerate}
  \end{Lemma}
  
 \begin{proof}  
   $\bullet$ If a vertex $v \in T_{\Phi}$, at a level $k>0$, is so that no identification occur for $v$ then it is associated to an interval of level $k$, say
 $ I_{j_1, \dots, j_k}$. We define $I_v:= I_{j_1, \dots, j_k}$  and we say that $I_v$ is an {\it interval of type $(i)$.}
 
 $\bullet$ If two vertices $v_{j_1, \dots, j_k}$ and $v_{l_1, \dots, l_k}$ of $ T_{\Phi}$ satisfy the conditions {\it(a)} and {\it(b)} of $(\mathscr{V})$ let us study the relation $ \sim_{\Phi}$.
 
 - Assume first that $r=0$ in $(a1)$. By 
  condition $(a2)$, the two intervals $I_{j_1}$ and $I_{l_1}$ are adjacent, so they have a cutting point $ z $ in common and the $k-1$ first intervals in the sequence, up to  $ I_{j_1, \dots , j_{k-1}}$ and 
  $ I_{l_1, \dots , l_{k-1}}$ are adjacent. By conditions (E+), (E-), this property is satisfied for our map $\Phi$ on the intervals containing the cutting point $z$ on one side, for the integer $k = k(z)$ given by (EC) at $z$.\\
  By condition $(b)$, the intervals $ I_{j_1, \dots , j_{k}}$ and 
  $ I_{l_1, \dots , l_{k}}$ are adjacent. As above, this property is satisfied  for $\Phi$ by the intervals containing the cutting point $z$ as an extreme point. By conditions (E+), (E-), the $\Phi^m $ images of these intervals are disjoint for $m < k(z)-1$, so condition $(b1)$ is satisfied.\\
  By condition $(b2)$, the $\Phi^k$-images of $ I_{j_1, \dots , j_{k}}$ and 
  $ I_{l_1, \dots , l_{k}}$ have one point in common. 
  For the map $\Phi$ this point  is  the $\Phi^{k(z)} $ image of the cutting point $z$ given above,  by the eventual coincidence condition (EC) on $\Phi$ at $z$.
  The conditions $(b1)$ and $(b2)$ are satisfied for this iterate $k(z)$ and such a condition is satisfied for each cutting point and therefore for each pair of adjacent intervals. The second condition in $(b2)$ is satisfied since $\Phi$ is  a piecewise homeomorphism.
  
-  When $r>0$ in $(a)$, then the pair of adjacent intervals 
  $ I_{j_1, \dots , j_r, j_{r+1}}$ and  $I_{j_1, \dots , j_r , l_{r+1}}$ given by condition $(a2)$ are so that the 
  $\Phi^r$-image of these  intervals are adjacent at level $1$. 
  Thus they have a cutting point in common and the arguments above  apply:  there is an integer $k$ for which conditions $(b1)$ and $(b2)$ are satisfied.
   In addition, if we denote by $\tilde{v}$ the vertex obtained by the identification $ \sim_{\Phi}$ 
   from $v_{j_1, \dots, j_k}$ and $v_{l_1, \dots , l_k}$, then we define 
   $I_{\tilde{v}} := I_{j_1, \dots , j_{k}} \cup I_{l_1, \dots , l_{k}}$, which is an interval
   since  $ I_{j_1, \dots , j_{k}}$ and $ I_{l_1, \dots , l_{k}}$ are adjacent along $S^1$, we say that $I_{\tilde{v}}$  is an {\it interval of type $(ii)$}.\\
     The identification  in Definition \ref{graph} -$(\mathscr{V})$ is well defined and occurs at each level after some minimal level:
   $K _0= {\rm min}\{k(j)|j = 1,\dots, 2N\}$, where the $k(j)$'s are the integers of condition (EC).
   
$\bullet$  Let us consider the identification $(\mathscr{E})$ in Definition \ref{graph}. This is an inductive operation from a level $k+1$
   if some identification occurred at level $k$. If an identification $(\mathscr{V})$ occurred at level $k$ defining a vertex 
   $\tilde{v}$ then, by condition $(b)$,
   there is a point\\
    $\mathscr{z}= \Phi^k (  I_{j_1, \dots , j_{k}} ) \bigcap \Phi^k (   I_{l_1, \dots , l_{k}} )$ and
   two cases can arise:
   
       (1)  $\mathscr{z}= \Phi^k (  I_{j_1, \dots , j_{k}} ) \bigcap \Phi^k (   I_{l_1, \dots , l_{k}} ) $ is a cutting point $z$,
    
   (2)  $ \mathscr{z} = \Phi^k (  I_{j_1, \dots , j_{k}} ) \bigcap \Phi^k (   I_{l_1, \dots , l_{k}} ) $ belongs to the interior of an interval, say $I_{a_1}$.\\
    In case (1), all edges starting from the identified vertex $\tilde{v}$ have different labels and the identification $(\mathscr{E})$ does not happen.  
  
In case (2),  there is a sub-interval $ I_{j_1, \dots , j_{k}, a_1}$ of 
   $ I_{j_1, \dots , j_{k}}$ and an edge labeled $\Psi_{a_1}$ connecting $v_{j_1, \dots , j_{k}}$ to 
   $ v_{j_1, \dots , j_{k}, a_1}$ and an edge, labeled $\Psi_{a_1}$, connecting 
   $v_{l_1, \dots , l_{k}}$ to $ v_{l_1, \dots , l_{k}, a_1}$ in $T_{\Phi}$. 
   The identification of the two vertices: $v_{j_1, \dots , j_{k}}$ and $v_{l_1, \dots , l_{k}}$ by 
   $ \sim_{\Phi}$ at level $k$ implies that two edges labelled  $\Psi_{a_1}$ start from the new vertex 
   $\tilde{v}$. 
   The identification in Definition \ref{graph}-{($\mathscr{E}$)} identifies these two edges to a single edge, labelled $\Psi_{a_1}$, connecting 
   $\tilde{v}$ to $\tilde{v}^1$ at level $k+1$. 
   This identification is well defined at level $k+1$.\\
   In addition the two intervals $ I_{j_1, \dots , j_{k}, a_1}$ and $ I_{l_1, \dots , l_{k}, a_1}$ are adjacent on $S^1$ and we associate to the vertex $\tilde{v}^1$ the interval:
   $ I_{\tilde{v}^1} = I_{j_1, \dots , j_{k}, a_1} \cup I_{l_1, \dots , l_{k}, a_1}$, this is an interval of type $(ii)$.
    
    The identification of type $(\mathscr{E})$ is then applied inductively on each level following $k+1$. At level $k+2$, if the image $\Phi (\mathscr{z})$ is a cutting point then, as in case (1), the identification of type $(\mathscr{E})$ stops i.e., the edges starting from $\tilde{v}^1$
   have different label and there is no identification of the type $(\mathscr{E})$.
   If $\Phi (\mathscr{z})$ belongs to the interior of an interval $I_{a_2}$ then, as in case (2), two edges with label 
   $\Psi_{a_2}$ start at $\tilde{v}^1$ and a new identification of type $(\mathscr{E})$ occurs.
   The inductive identification of type $(\mathscr{E})$, starting at $\tilde{v}$, depends on the orbit $\Phi^m (\mathscr{z})$:\\
    -  If, for some $m\geq 0$, $\Phi^m (\mathscr{z})$ is a cutting point then the identification starting at level $k$ at 
    $\tilde{v}$ stops at level $k+m$, as in case (1).\\
  -  If $\Phi^m (\mathscr{z})$ is not a cutting point for all $m \geq 0$ then the identification of type $(\mathscr{E})$ starting at 
   $\tilde{v}$ does not stop and is well defined for each level $k+m$.
   At each level of this inductive identification, a new vertex is defined and is associated to an interval which is the union of the adjacent intervals associated to the identified vertices in $T_{\Phi}$. At each such level the new vertex is of type $(ii)$.\\
Finally, we have to check if an identification of type $(\mathscr{V})$ and one of type 
$(\mathscr{E})$ could possibly interact i.e., occur at the same level.\\ 
       Let us observe that the neighborhood: $V_j = V_j^{c_j} \cup V_j^{d_j}$ in the proof of 
  Lemma \ref{affin-diffeo} is in fact an interval of the form:
  $ I_w := I_{j_1, \dots , j_{k(j)}} \cup  I_{l_1, \dots , l_{k(j)}}$ i.e., of type $(ii)$ by 
  $(\mathscr{V})$, at level $k(j)$.\\
    It turns out that the identifications of type $(ii)$  by $(\mathscr{V})$ and the identifications  of type $(ii)$ by $(\mathscr{E})$,  can indeed interact. This happens in the following situations:\\
  An  identification of type $(ii)$  by $(\mathscr{E})$  occurs if $\mathscr{z} = \Phi^{k(j)} (z_j) \in \textrm{int} ( I_{a} )$
  for some $a$. Assume that $\mathscr{z} = \Phi^{k(j)} (z_j) \in V_{a} \cap I_{a}$, where $V_{a}$ is the neighborhood of the cutting point $z_{a}$ described above.
   An identification of type $(ii)$  by $(\mathscr{E})$ occurs at level $k(j) + 1$ and, by condition (E+): 
  $\Phi^m (\mathscr{z}) \in I_{\delta^m(a)}$ for $m \leq k(a) -1$. This implies that 
  $\Phi^m (\mathscr{z})$ is not a cutting point for all $m \leq k(a) -1$. By the condition (2) above, an identification of type $(ii)$  by $(\mathscr{E})$ occurs for each level from $k(j) + 1$ up to
  $k(j) + k(a) -1$.
  At level $k(j) + k(a) $ an identification of type $(ii)$  by $(\mathscr{E})$ and one of type $(ii)$  by 
  $(\mathscr{V})$ occur at the same level, and three vertices of the tree $ T_{\Phi}$  will be identified. These vertices are related with three intervals that are  pairwise adjacent along $S^1$ , say:
 $$ I_{j_1, \dots, j_{k(j)}, a',a'_2, \dots ,a'_{k(a)}},\; 
   I_{j_1, \dots, j_{k(j)}, a, a_2, \dots, a_{k(a)}},\; 
   I_{l_1, \dots, l_{k(j)}, a, a_2, \dots, a_{k(a)}}, \textrm{ where $a' = \zeta^{-1} (a).$}$$
 We associate the vertex obtained from the  identifications of the three vertices with the union of these three intervals, this is an interval of type $(ii)$.
  
For the next levels, the two cases (1) or (2) above can occur, depending on the orbits of each  cutting point i.e.,  $z_j$ and $z_{a}$.

The phenomenon described above, where two identifications of different type 
 arise for the same vertex, can possibly occur at any level large enough. 
 The intervals associated to vertices in $ T_{\Phi}$ that are involved are pairwise adjacent, as above, and the union of these intervals is an interval. 
  The number $n_m$ of these intervals depends on the map $\Phi$ via the orbits of the cutting points and the level $m$.\\
   Hence, the dynamical graph $\Gamma_{\Phi}$ is well defined from the map $\Phi$.
   \end{proof}

   \begin{Rm}\label{V-E-ii}
    Notice that  a vertex obtained by an identification of type $(\mathscr{V})$ has two incoming edges i.e., from level $k-1$ to level $k$. A vertex obtained by an identification of type $(\mathscr{E})$ has only one incoming edge, as the vertices of type {\it (i)}.
   If necessary, we will mark the difference by denoting the corresponding vertices or intervals of  {\it type} ({\it ii}-$\mathscr{V}$) or {\it type (ii-$\mathscr{E})$}. \\
    It is interesting to observe that the identification of type $(\mathscr{E})$  is essentially a Stallings folding \cite{Sta}.
 \end{Rm}
 
  \begin{Prop}\label{GammaHomeo}
  If $\Phi$ and $\Phi'$ are two piecewise homeomorphisms of $S^1$ with the same combinatorics i.e., the same permutations $\zeta$ and $\iota$, the same properties $\rm(SE),$ $\rm(EC),$ $\rm(E\pm),$  $\rm(CS\textrm{-}\lambda)$ with the same slope $\lambda$,
  then the graphs $\Gamma_{\Phi}$ and $\Gamma_{\Phi '}$ are homeomorphic.
  \end{Prop} 
\begin{proof}
Since the combinatorics are the same, all the combinatorial data used in the constructions: $k(j)$,
$\gamma$ and $\delta$ are the same for $\Phi$ and $\Phi'$.
The identification of type 
  ({\it ii}-$\mathscr{V}$) defines vertices with two incoming edges, and $2N-2$ outgoing edges, by condition (\ref{k(j)-image}) in the proof of Lemma \ref{affin-diffeo}. The vertices of types $(i)$ and   ({\it ii}-$\mathscr{E}$)  have one incoming edge and $2N-1$ outgoing edges by condition (SE). The sequence of identifications that occur for the two maps can be quite different (see (1) and (2) in the proof of Lemma \ref{vertexInterval}) but each resulting vertex has the same structure. Therefore the two graphs are homeomorphic.
  Notice that the two maps $\Phi$ and $\Phi'$ are in general quite different and in particular non-conjugate.
\end{proof} 
 \begin{Ex} To illustrate the possible types of identification, as mentioned in the proof of Lemma \ref{vertexInterval} and Proposition \ref{GammaHomeo}, let us consider the following example:\\
  for  $\zeta=(1\, 2\, \overline{1}\, \overline{2}\, 5\, 6\, \overline{5}\, \overline{6} )$  we have $\delta=(1\, \overline{2}\, \overline{1}\, 2 \, 5 \, \overline{6}\, \overline{5}\, 6)$ and  $k(j)=4$ for all $j$. If $j=1$ and 
  $\mathscr{z}^4$ is the point given by condition {\rm(EC)},  Figures \ref{i1} and \ref{i2} exhibit two possible identifications, here at some levels up to 8.
 \end{Ex}
 \begin{figure}[ht!]
\begin{subfigure}[h]{0.65\linewidth}
\resizebox{0.95\textwidth}{!}{%
\begin{tikzpicture}[grow'=up,scale=1]
\tikzstyle{level 1}=[sibling distance=5.0in]
\tikzstyle{level 2}=[sibling distance=0.3in]
\tikzstyle{level 3}=[sibling distance=0.4in]
\tikzstyle{level 4}=[sibling distance=0.4in]
\tikzstyle{level 5}=[sibling distance=0.75in]
\tikzstyle{level 6}=[sibling distance=0.35in]
\tikzstyle{level 7}=[sibling distance=0.3in]
\tikzstyle{level 8}=[sibling distance=0.3in]

\node [label=$\tiny{0}$]
 {} coordinate (t9)
child{ coordinate (to6) edge from parent[color=black, thin]
child{ coordinate (to62) edge from parent[color=black, thin]}
child{ coordinate (to6o2) edge from parent[color=black, thin]}
child{ coordinate (to65) edge from parent[color=black, thin]
child{ coordinate (to65o2) edge from parent[color=black, thin]}
child{ coordinate (to655) edge from parent[color=black, thin]}
child{ coordinate (to656) edge from parent[color=black, thin]
child{ coordinate (to6565) edge from parent[color=black, thin]}
child{ coordinate (to6566) edge from parent[color=black, thin]}
child{ coordinate (to656o5) edge from parent[color=black, thin]
child{ coordinate (to656o5-) edge from parent[color=black, thin]}
child{ coordinate (to656o5--) edge from parent[color=black, thin]}
child{ coordinate (to656o5o6) edge from parent[color=black, thin]
child{ coordinate (to656o5o6o2) edge from parent[color=black, thin]}
child{ coordinate (to656o5o65) edge from parent[color=black, thin]
child{ coordinate (to656o5o650) edge from parent[color=black, thin]}
child{ coordinate (to656o5o6500) edge from parent[color=black, thin]
child{ coordinate (to656o5o650-) edge from parent[color=black, thin]}
child{ coordinate (to656o5o650--) edge from parent[color=black, thin]}}
}}
child{ coordinate (to656o51) edge from parent[color=red,ultra thick]
child{ coordinate (to656o51o2) edge from parent[color=red,ultra thick]
child{ coordinate (to656o51o2-) edge from parent[color=red,ultra thick]
child{ coordinate (to656o51o2-i) edge from parent[color=red,ultra thick]}}}}
}}}}
child{ coordinate (t1) edge from parent[color=black, thin]
child{ coordinate (t1o2) edge from parent[color=black, thin]
child{ coordinate (t1o2o1) edge from parent[color=black, thin]
child{ coordinate (t1o2o12) edge from parent[color=black, thin]
child{ coordinate (t1o2o121) edge from parent[color=red,ultra thick]
child{ coordinate (t1o2o121o2) edge from parent[color=red,ultra thick]
child{ coordinate (t1o2o121o2-) edge from parent[color=red,ultra thick]
child{ coordinate (t1o2o121o2-i) edge from parent[color=red,ultra thick]}
child{ coordinate (t1o2o121o2-ii) edge from parent[color=black, thin]}}
child{ coordinate (t1o2o121o2--) edge from parent[color=black, thin]}}
child{ coordinate (t1o2o121-) edge from parent[color=black, thin]}}
child{ coordinate (t1o2o122) edge from parent[color=black, thin]}
child{ coordinate (t1o2o12o1) edge from parent[color=black, thin]}}
child{ coordinate (t1o2o1o1) edge from parent[color=black, thin]}
child{ coordinate (t1o2o1o2) edge from parent[color=black, thin]}}
child{ coordinate (t1o2o2) edge from parent[color=black, thin]}
child{ coordinate (t1o25) edge from parent[color=black, thin]}}
child{ coordinate (t15) edge from parent[color=black, thin]}
child{ coordinate (t1o5) edge from parent[color=black, thin]}};

\node[circle, fill=white,inner sep=1pt, minimum size=1pt] at (t1) {\small{$1$}};
\node[circle, fill=white,inner sep=1pt, minimum size=1pt] at (t1o2) {\small{${1\overline{2}}$}};
\node[circle, fill=white,inner sep=1pt, minimum size=1pt] at (t1o2o1) {\small{${1\overline{21}}$}};
\node[circle, fill=white,inner sep=1pt, minimum size=1pt] at (t1o2o12) {\color{blue}{\small{$\boldsymbol{1\overline{21}2}$}}};
\node[circle, fill=white,inner sep=1pt, minimum size=1pt] at (t1o2o121) {\color{blue}{\small{$\boldsymbol{1\overline{21}21}$}}};
\node[circle, fill=white,inner sep=1pt, minimum size=1pt] at (t1o2o121o2) {\color{blue}{\small{$\boldsymbol{1\overline{21}21\overline{2}}$}}};
\node[circle, fill=white,inner sep=1pt, minimum size=1pt] at (t1o2o121o2-) {\color{blue}{$\bullet$}};
\node[circle, fill=white,inner sep=1pt, minimum size=1pt] at (t1o2o121o2-i) {\color{blue}{$\bullet$}};
\node[circle, fill=white,inner sep=1pt, minimum size=1pt] at (t1o2o121o2-ii) {{$\dots$}};

\node[circle, fill=white,inner sep=1pt, minimum size=1pt] at (t1o2o121o2--) {{$\dots$}};

\node[circle, fill=white,inner sep=1pt, minimum size=1pt] at (t1o2o121-) {\dots};

\node[circle, fill=white,inner sep=1pt, minimum size=1pt] at (t1o2o122) {\small{${1\overline{21}22}$}};
\node[circle, fill=white,inner sep=1pt, minimum size=1pt] at (t1o2o12o1) {\small{${1\overline{21}2\overline{1}}$}};
\node[circle, fill=white,inner sep=1pt, minimum size=1pt] at (t1o2o1o1) {\small{${1\overline{211}}$}};
\node[circle, fill=white,inner sep=1pt, minimum size=1pt] at (t1o2o1o2) {\dots};

\node[circle, fill=white,inner sep=1pt, minimum size=1pt] at (t1o2o2) {\small{${1\overline{22}}$}};
\node[circle, fill=white,inner sep=1pt, minimum size=1pt] at (t1o25) {\dots};

\node[circle, fill=white,inner sep=1pt, minimum size=1pt] at (t15) {\small{${15}$}};
\node[circle, fill=white,inner sep=1pt, minimum size=1pt] at (t1o5) {\dots};

\node[circle, fill=white,inner sep=1pt, minimum size=1pt] at (to6) {\small{$\overline{6}$}};
\node[circle, fill=white,inner sep=1pt, minimum size=1pt] at (to62) {\dots};
\node[circle, fill=white,inner sep=1pt, minimum size=1pt] at (to6o2) {\small{$\overline{62}$}};
\node[circle, fill=white,inner sep=1pt, minimum size=1pt] at (to65) {\small{$\overline{6}5$}};
\node[circle, fill=white,inner sep=1pt, minimum size=1pt] at (to65o2) {\small{$\dots$}};

\node[circle, fill=white,inner sep=1pt, minimum size=1pt] at (to655) {\small{$\overline{6}55$}};
\node[circle, fill=white,inner sep=1pt, minimum size=1pt] at (to656) {\small{$\overline{6}56$}};
\node[circle, fill=white,inner sep=1pt, minimum size=1pt] at (to6565) {\small{$\overline{6}565$}};
\node[circle, fill=white,inner sep=1pt, minimum size=1pt] at (to6566) {\small{$\overline{6}566$}};
\node[circle, fill=white,inner sep=1pt, minimum size=1pt] at (to656o5) {\color{blue}{\small{$\boldsymbol{\overline{6}56\overline{5}}$}}};
\node[circle, fill=white,inner sep=1pt, minimum size=1pt] at (to656o5-) {\small{$\overline{6}56\overline{5}6$}};
\node[circle, fill=white,inner sep=1pt, minimum size=1pt] at (to656o5--) {\small{$\overline{6}56\overline{55}$}};

\node[circle, fill=white,inner sep=0pt, minimum size=0pt] at (to656o5o650) {$\dots$};
\node[circle, fill=white,inner sep=0pt, minimum size=0pt] at (to656o5o650-) {$\dots$};
\node[circle, fill=white,inner sep=0pt, minimum size=0pt] at (to656o5o650--) {\color{blue}{$\bullet$}};

\node[circle, fill=white,inner sep=0pt, minimum size=0pt] at (to656o5o6500) {\color{gray}$\bullet$};
\node[circle, fill=white,inner sep=0pt, minimum size=0pt] at (to656o51o2-) {\color{blue}{$\bullet$}};
\node[circle, fill=white,inner sep=0pt, minimum size=0pt] at (to656o51o2-i) {\color{blue}{$\bullet$}};

\node[circle, fill=white,inner sep=0pt, minimum size=0pt] at (to656o5o6) {\small{${\overline{6}56\overline{56}}$}};
\node[circle, fill=white,inner sep=0pt, minimum size=0pt] at (to656o5o6o2) {\color{gray}\dots};
\node[circle, fill=white,inner sep=0pt, minimum size=0pt] at (to656o5o65) {\small{${\overline{6}56\overline{56}5}$}};

\node[circle, fill=white,inner sep=0pt, minimum size=0pt] at (to656o51) {\color{blue}{\small{$\boldsymbol{\overline{6}56\overline{5}1}$}}};
\node[circle, fill=white,inner sep=0pt, minimum size=0pt] at (to656o51o2) {\color{blue}{\small{$\boldsymbol{\overline{6}56\overline{5}1\overline{2}}$}}};

 \draw [decorate,color=black] (0.5,3)
   node[left] {\begin{Large}{$T_{\Phi}$}\end{Large}};
\end{tikzpicture}
}
\hspace*{-3cm}
\captionsetup{font=footnotesize} 

\end{subfigure}
\begin{subfigure}[h]{0.65\linewidth}
\resizebox{0.6\textwidth}{!}{%
\begin{tikzpicture}[grow'=up,scale=1]
\tikzstyle{level 1}=[sibling distance=3.in]
\tikzstyle{level 2}=[sibling distance=0.5in]
\tikzstyle{level 3}=[sibling distance=0.5in]
\tikzstyle{level 4}=[sibling distance=0.5in]
\tikzstyle{level 5}=[sibling distance=0.75in]
\tikzstyle{level 6}=[sibling distance=0.2in]
\tikzstyle{level 7}=[sibling distance=0.2in]
\tikzstyle{level 8}=[sibling distance=0.35in]
%

\node [label=$\tiny{0}$]
 {} coordinate (t9)
child{ coordinate (to6) edge from parent[color=black, thin]
child{ coordinate (to62) edge from parent[color=black, thin]}
child{ coordinate (to6o2) edge from parent[color=black, thin]}
child{ coordinate (to65) edge from parent[color=black, thin]
child{ coordinate (to65o2) edge from parent[color=black, thin]}
child{ coordinate (to655) edge from parent[color=black, thin]}
child{ coordinate (to656) edge from parent[color=black, thin]
child{ coordinate (to6565) edge from parent[color=black, thin]}
child{ coordinate (to6566) edge from parent[color=black, thin]}
child{ coordinate (to656o5) edge from parent[color=black, thin]
child{ coordinate (to656o51) edge from parent[color=black, thin]}
child{ coordinate (to656o52) edge from parent[color=black, thin]}
child{ coordinate (to656o53) edge from parent[color=black, thin]
child{ coordinate (to656o532) edge from parent[color=black, thin]}
child{ coordinate (to656o533) edge from parent[color=black, thin]
child{ coordinate (to656o5331) edge from parent[color=black, thin]
}
child{ coordinate (to656o5332) edge from parent[color=black, thin]
child{ coordinate (to656o53321) edge from parent[color=black, thin]}
child{ coordinate (to656o53322) edge from parent[color=black, thin]}
}
}}
child{ coordinate (to656o54) edge from parent[color=red,ultra thick]
child{ coordinate (to656o541) edge from parent[color=red,ultra thick]
child{ coordinate (to656o5411) edge from parent[color=red,ultra thick]
child{ coordinate (to656o54111) edge from parent[color=red,ultra thick]}
child{ coordinate (to656o54112) edge from parent[color=black, thin]}}
child{ coordinate (to656o5412) edge from parent[color=black, thin]
}}
child{ coordinate (to656o542) edge from parent[color=black, thin]}}
child{ coordinate (to656o55) edge from parent[color=black, thin]}
child{ coordinate (to656o56) edge from parent[color=black, thin]}
}}}}
child{ coordinate (t1) edge from parent[color=black, thin]
child{ coordinate (t1o2) edge from parent[color=black, thin]
child{ coordinate (t1o2o1) edge from parent[color=black, thin]
child{ coordinate (t1o2o12) edge from parent[color=black, thin]}
child{ coordinate (t1o2o121) edge from parent[color=black, thin]}
child{ coordinate (t1o2o121o2) edge from parent[color=black, thin]}}
child{ coordinate (t1o2o2) edge from parent[color=black, thin]}
child{ coordinate (t1o25) edge from parent[color=black, thin]}}
child{ coordinate (t15) edge from parent[color=black, thin]}
child{ coordinate (t1o5) edge from parent[color=black, thin]}};

\node[circle, fill=white,inner sep=1pt, minimum size=1pt] at (t1) {$\boldsymbol{\cdot}$};
\node[circle, fill=white,inner sep=1pt, minimum size=1pt] at (t1o2) {$\boldsymbol{\cdot}$};
\node[circle, fill=white,inner sep=1pt, minimum size=1pt] at (t1o2o1) {$\boldsymbol{\cdot}$};
\node[circle, fill=white,inner sep=1pt, minimum size=1pt] at (t1o2o12) {\color{blue}{$\bullet$}};
\node[circle, fill=white,inner sep=1pt, minimum size=1pt] at (t1o2o121) {{$\boldsymbol{\cdot}$}};
\node[circle, fill=white,inner sep=1pt, minimum size=1pt] at (t1o2o121o2) {\dots};

%
%

\node[circle, fill=white,inner sep=1pt, minimum size=1pt] at (t1o2o2) {$\boldsymbol{\cdot}$};
\node[circle, fill=white,inner sep=1pt, minimum size=1pt] at (t1o25) {\dots};

\node[circle, fill=white,inner sep=1pt, minimum size=1pt] at (t15) {$\boldsymbol{\cdot}$};
\node[circle, fill=white,inner sep=1pt, minimum size=1pt] at (t1o5) {\dots};

\node[circle, fill=white,inner sep=1pt, minimum size=1pt] at (to6) {$\boldsymbol{\cdot}$};
\node[circle, fill=white,inner sep=1pt, minimum size=1pt] at (to62) {\dots};
\node[circle, fill=white,inner sep=1pt, minimum size=1pt] at (to6o2) {$\boldsymbol{\cdot}$};
\node[circle, fill=white,inner sep=1pt, minimum size=1pt] at (to65) {$\boldsymbol{\cdot}$};
\node[circle, fill=white,inner sep=1pt, minimum size=1pt] at (to65o2) {\small{$\dots$}};
\node[circle, fill=white,inner sep=1pt, minimum size=1pt] at (to655) {$\boldsymbol{\cdot}$};
\node[circle, fill=white,inner sep=1pt, minimum size=1pt] at (to656) {$\boldsymbol{\cdot}$};
\node[circle, fill=white,inner sep=1pt, minimum size=1pt] at (to6565) {\dots};
\node[circle, fill=white,inner sep=1pt, minimum size=1pt] at (to6566) {$\boldsymbol{\cdot}$};
\node[circle, fill=white,inner sep=1pt, minimum size=1pt] at (to656o5) {\color{blue}{$\bullet$}};
\node[circle, fill=white,inner sep=1pt, minimum size=1pt] at (to656o51) {$\boldsymbol{\cdot}$};
\node[circle, fill=white,inner sep=1pt, minimum size=1pt] at (to656o52) {$\boldsymbol{\cdot}$};
\node[circle, fill=white,inner sep=1pt, minimum size=1pt] at (to656o53) {\color{gray}$\bullet$};
\node[circle, fill=white,inner sep=1pt, minimum size=1pt] at (to656o532) {$\dots$};
\node[circle, fill=white,inner sep=1pt, minimum size=1pt] at (to656o533) {\color{gray}$\bullet$};
\node[circle, fill=white,inner sep=1pt, minimum size=1pt] at (to656o5331) {\dots};
\node[circle, fill=white,inner sep=1pt, minimum size=1pt] at (to656o5332) {\color{gray}$\bullet$};
\node[circle, fill=white,inner sep=1pt, minimum size=1pt] at (to656o53321) {$\dots$};
\node[circle, fill=white,inner sep=1pt, minimum size=1pt] at (to656o53322) {\color{blue}{$\bullet$}};

\node[circle, fill=white,inner sep=1pt, minimum size=1pt] at (to656o54) {\color{blue}{$\bullet$}};
\node[circle, fill=white,inner sep=1pt, minimum size=1pt] at (to656o541) {\color{blue}{$\bullet$}};
\node[circle, fill=white,inner sep=1pt, minimum size=1pt] at (to656o5411) {\color{blue}{$\bullet$}};
\node[circle, fill=white,inner sep=1pt, minimum size=1pt] at (to656o54112) {\color{gray}\dots};
\node[circle, fill=white,inner sep=1pt, minimum size=1pt] at (to656o5412) {{$\dots$}};
\node[circle, fill=white,inner sep=1pt, minimum size=1pt] at (to656o542) {\dots};
\node[circle, fill=white,inner sep=1pt, minimum size=1pt] at (to656o55) {$\boldsymbol{\cdot}$};
\node[circle, fill=white,inner sep=1pt, minimum size=1pt] at (to656o56) {$\boldsymbol{\cdot}$};





 \draw [decorate,color=black] (0.5,2.5)
   node[left] {\begin{Large}{$\Gamma_{\Phi}$}\end{Large}};

\end{tikzpicture}
}
\end{subfigure}
\caption{{\small{Identification of type $(\mathscr{V})$  at levels $4$ and $8,$ and of type 
 $(\mathscr{E})$ from  level $5$ when
 $\mathscr{\textcolor{black}{z^4}} \in I_1 \cap V_1$ }}}
\label{i1}
\end{figure}
\begin{figure}[ht!]
\begin{subfigure}[h]{0.65\linewidth}
\resizebox{1\textwidth}{!}{%
\begin{tikzpicture}[grow'=up,scale=1]
\tikzstyle{level 1}=[sibling distance=6.2in]
\tikzstyle{level 2}=[sibling distance=0.25in]
\tikzstyle{level 3}=[sibling distance=0.4in]
\tikzstyle{level 4}=[sibling distance=0.4in]
\tikzstyle{level 5}=[sibling distance=0.75in]
\tikzstyle{level 6}=[sibling distance=0.5in]
\tikzstyle{level 7}=[sibling distance=0.15in]
\tikzstyle{level 8}=[sibling distance=0.3in]



\node [label=$\tiny{0}$]
 {} coordinate (t9)
child{ coordinate (to6) edge from parent[color=black, thin]
child{ coordinate (to62) edge from parent[color=black, thin]}
child{ coordinate (to6o2) edge from parent[color=black, thin]}
child{ coordinate (to65) edge from parent[color=black, thin]
child{ coordinate (to65o2) edge from parent[color=black, thin]}
child{ coordinate (to655) edge from parent[color=black, thin]}
child{ coordinate (to656) edge from parent[color=black, thin]
child{ coordinate (to6565) edge from parent[color=black, thin]}
child{ coordinate (to6566) edge from parent[color=black, thin]}
child{ coordinate (to656o5) edge from parent[color=black, thin]
child{ coordinate (to656o5-) edge from parent[color=black, thin]}
child{ coordinate (to656o5--) edge from parent[color=black, thin]}
child{ coordinate (to656o5o6) edge from parent[color=black, thin]
}
child{ coordinate (to656o51) edge from parent[color=blue,ultra thick]
child{ coordinate (to656o51a) edge from parent[color=black, thin]}
child{ coordinate (to656o51b) edge from parent[color=black, thin]}
child{ coordinate (to656o51o2) edge from parent[color=blue,ultra thick]
child{ coordinate (to656o51o2a) edge from parent[color=black, thin]}
child{ coordinate (to656o51o2b) edge from parent[color=black, thin]}
child{ coordinate (to656o51o2c) edge from parent[color=black, thin]}
child{ coordinate (to656o51o2d) edge from parent[color=black, thin]}
child{ coordinate (to656o51o2e) edge from parent[color=black, thin]}
child{ coordinate (to656o51o2-) edge from parent[color=blue,ultra thick]
}}}
}}}}
child{ coordinate (t1) edge from parent[color=black, thin]
child{ coordinate (t1o2) edge from parent[color=black, thin]
child{ coordinate (t1o2o1) edge from parent[color=black, thin]
child{ coordinate (t1o2o12) edge from parent[color=black, thin]
child{ coordinate (t1o2o121) edge from parent[color=blue,ultra thick]
child{ coordinate (t1o2o12a) edge from parent[color=blue,ultra thick]
child{ coordinate (t1o2o12aa) edge from parent[color=blue,ultra thick]}
child{ coordinate (t1o2o12ab) edge from parent[color=black, thin]}
}
child{ coordinate (t1o2o12b) edge from parent[color=black, thin]}
child{ coordinate (t1o2o12c) edge from parent[color=black, thin]}
child{ coordinate (t1o2o121o2) edge from parent[color=black, thin]}
child{ coordinate (t1o2o121-) edge from parent[color=black, thin]}}
child{ coordinate (t1o2o122) edge from parent[color=black, thin]}
child{ coordinate (t1o2o12o1) edge from parent[color=black, thin]}}
child{ coordinate (t1o2o1o1) edge from parent[color=black, thin]}
child{ coordinate (t1o2o1o2) edge from parent[color=black, thin]}}
child{ coordinate (t1o2o2) edge from parent[color=black, thin]}
child{ coordinate (t1o25) edge from parent[color=black, thin]}}
child{ coordinate (t15) edge from parent[color=black, thin]}
child{ coordinate (t1o5) edge from parent[color=black, thin]}};

\node[circle, fill=white,inner sep=1pt, minimum size=1pt] at (t1) {\small{$1$}};
\node[circle, fill=white,inner sep=1pt, minimum size=1pt] at (t1o2) {\small{${1\overline{2}}$}};
\node[circle, fill=white,inner sep=1pt, minimum size=1pt] at (t1o2o1) {\small{${1\overline{21}}$}};
\node[circle, fill=white,inner sep=1pt, minimum size=1pt] at (t1o2o12) {\color{blue}{\color{gray}\small{$\boldsymbol{1\overline{21}2}$}}};
\node[circle, fill=white,inner sep=1pt, minimum size=1pt] at (t1o2o121) {\color{blue}{\small{$\boldsymbol{1\overline{21}21}$}}};
\node[circle, fill=white,inner sep=1pt, minimum size=1pt] at (t1o2o12a) {\color{blue}{\small{$\boldsymbol{1\overline{21}216}$}}};
\node[circle, fill=white,inner sep=1pt, minimum size=1pt] at (t1o2o12b) {{\small{$1\overline{21}21\overline{5}$}}};
\node[circle, fill=white,inner sep=1pt, minimum size=1pt] at (t1o2o12c) {{\small{$1\overline{21}21\overline{6}$}}};
\node[circle, fill=white,inner sep=1pt, minimum size=1pt] at (t1o2o121o2) {{\small{$1\overline{21}211$}}};
\node[circle, fill=white,inner sep=1pt, minimum size=1pt] at (t1o2o12aa) {\color{blue}{$\bullet$}};


\node[circle, fill=white,inner sep=1pt, minimum size=1pt] at (t1o2o121-) {{\small{$1\overline{21}212$}}};

\node[circle, fill=white,inner sep=1pt, minimum size=1pt] at (t1o2o122) {\small{${1\overline{21}22}$}};
\node[circle, fill=white,inner sep=1pt, minimum size=1pt] at (t1o2o12o1) {\small{${1\overline{21}2\overline{1}}$}};
\node[circle, fill=white,inner sep=1pt, minimum size=1pt] at (t1o2o1o1) {\small{${1\overline{211}}$}};
\node[circle, fill=white,inner sep=1pt, minimum size=1pt] at (t1o2o1o2) {\dots};

\node[circle, fill=white,inner sep=1pt, minimum size=1pt] at (t1o2o2) {\small{${1\overline{22}}$}};
\node[circle, fill=white,inner sep=1pt, minimum size=1pt] at (t1o25) {\dots};

\node[circle, fill=white,inner sep=1pt, minimum size=1pt] at (t15) {\small{${15}$}};
\node[circle, fill=white,inner sep=1pt, minimum size=1pt] at (t1o5) {\dots};

\node[circle, fill=white,inner sep=1pt, minimum size=1pt] at (to6) {\small{$\overline{6}$}};
\node[circle, fill=white,inner sep=1pt, minimum size=1pt] at (to62) {\dots};
\node[circle, fill=white,inner sep=1pt, minimum size=1pt] at (to6o2) {\small{$\overline{62}$}};
\node[circle, fill=white,inner sep=1pt, minimum size=1pt] at (to65) {\small{$\overline{6}5$}};
\node[circle, fill=white,inner sep=1pt, minimum size=1pt] at (to65o2) {\small{$\dots$}};

\node[circle, fill=white,inner sep=1pt, minimum size=1pt] at (to655) {\small{$\overline{6}55$}};
\node[circle, fill=white,inner sep=1pt, minimum size=1pt] at (to656) {\small{$\overline{6}56$}};
\node[circle, fill=white,inner sep=1pt, minimum size=1pt] at (to6565) {\small{$\overline{6}56\overline{56}$}};
\node[circle, fill=white,inner sep=1pt, minimum size=1pt] at (to6566) {\small{$\overline{6}566$}};
\node[circle, fill=white,inner sep=1pt, minimum size=1pt] at (to656o5) {\color{blue}{\small{$\boldsymbol{\overline{6}56\overline{5}}$}}};
\node[circle, fill=white,inner sep=1pt, minimum size=1pt] at (to656o5-) {\small{$\overline{6}56\overline{5}6$}};
\node[circle, fill=white,inner sep=1pt, minimum size=1pt] at (to656o5--) {\small{$\overline{6}56\overline{55}$}};


\node[circle, fill=white,inner sep=0pt, minimum size=0pt] at (to656o51o2-) {\color{blue}{$\bullet$}};

\node[circle, fill=white,inner sep=0pt, minimum size=0pt] at (to656o5o6) {\small{$\overline{6}56\overline{56}$}};

\node[circle, fill=white,inner sep=0pt, minimum size=0pt] at (to656o51) {\color{blue}{\small{$\boldsymbol{\overline{6}56\overline{5}1}$}}};
\node[circle, fill=white,inner sep=0pt, minimum size=0pt] at (to656o51a) {{\small{$\overline{6}56\overline{5}1\overline{2}$}}};
\node[circle, fill=white,inner sep=0pt, minimum size=0pt] at (to656o51b) {{\small{$\overline{6}56\overline{5}15$}}};
\node[circle, fill=white,inner sep=0pt, minimum size=0pt] at (to656o51o2) {{\small{${\overline{6}56\overline{5}16}$}}};

 \draw [decorate,color=black] (0.5,2.5)
   node[left] {\begin{Large}{$T_{\Phi}$}\end{Large}};

\end{tikzpicture}
}

\end{subfigure}
\begin{subfigure}[h]{0.65\linewidth}
\resizebox{0.55\textwidth}{!}{%
\begin{tikzpicture}[grow'=up,scale=1]
\tikzstyle{level 1}=[sibling distance=3.in]
\tikzstyle{level 2}=[sibling distance=0.5in]
\tikzstyle{level 3}=[sibling distance=0.5in]
\tikzstyle{level 4}=[sibling distance=0.5in]
\tikzstyle{level 5}=[sibling distance=0.75in]
\tikzstyle{level 6}=[sibling distance=0.2in]
\tikzstyle{level 7}=[sibling distance=0.2in]
\tikzstyle{level 8}=[sibling distance=0.35in]

\node [label=$\tiny{0}$]
 {} coordinate (t9)
child{ coordinate (to6) edge from parent[color=black, thin]
child{ coordinate (to62) edge from parent[color=black, thin]}
child{ coordinate (to6o2) edge from parent[color=black, thin]}
child{ coordinate (to65) edge from parent[color=black, thin]
child{ coordinate (to65o2) edge from parent[color=black, thin]}
child{ coordinate (to655) edge from parent[color=black, thin]}
child{ coordinate (to656) edge from parent[color=black, thin]
child{ coordinate (to6565) edge from parent[color=black, thin]}
child{ coordinate (to6566) edge from parent[color=black, thin]}
child{ coordinate (to656o5) edge from parent[color=black, thin]
child{ coordinate (to656o51) edge from parent[color=black, thin]}
child{ coordinate (to656o52) edge from parent[color=black, thin]}
child{ coordinate (to656o53) edge from parent[color=black, thin]
}
child{ coordinate (to656o54) edge from parent[color=blue,ultra thick]
child{ coordinate (to656o54A) edge from parent[color=black, thin]}
child{ coordinate (to656o54B) edge from parent[color=black, thin]}
child{ coordinate (to656o541) edge from parent[color=blue,ultra thick]
child{ coordinate (to656o5411m) edge from parent[color=black, thin]}
child{ coordinate (to656o5411n) edge from parent[color=black, thin]}
child{ coordinate (to656o5411o) edge from parent[color=black, thin]}
child{ coordinate (to656o5411p) edge from parent[color=black, thin]}
child{ coordinate (to656o5411q) edge from parent[color=black, thin]}
child{ coordinate (to656o5411) edge from parent[color=blue,ultra thick]
}
child{ coordinate (to656o5412) edge from parent[color=black, thin]
}}
child{ coordinate (to656o542) edge from parent[color=black, thin]}
child{ coordinate (to656o543) edge from parent[color=black, thin]}
child{ coordinate (to656o544) edge from parent[color=black, thin]}
child{ coordinate (to656o545) edge from parent[color=black, thin]}}
child{ coordinate (to656o55) edge from parent[color=black, thin]}
child{ coordinate (to656o56) edge from parent[color=black, thin]}
}}}}
child{ coordinate (t1) edge from parent[color=black, thin]
child{ coordinate (t1o2) edge from parent[color=black, thin]
child{ coordinate (t1o2o1) edge from parent[color=black, thin]
child{ coordinate (t1o2o12) edge from parent[color=black, thin]}
child{ coordinate (t1o2o121) edge from parent[color=black, thin]}
child{ coordinate (t1o2o121o2) edge from parent[color=black, thin]}}
child{ coordinate (t1o2o2) edge from parent[color=black, thin]}
child{ coordinate (t1o25) edge from parent[color=black, thin]}}
child{ coordinate (t15) edge from parent[color=black, thin]}
child{ coordinate (t1o5) edge from parent[color=black, thin]}};

\node[circle, fill=white,inner sep=1pt, minimum size=1pt] at (t1) {$\boldsymbol{\cdot}$};
\node[circle, fill=white,inner sep=1pt, minimum size=1pt] at (t1o2) {$\boldsymbol{\cdot}$};
\node[circle, fill=white,inner sep=1pt, minimum size=1pt] at (t1o2o1) {$\boldsymbol{\cdot}$};
\node[circle, fill=white,inner sep=1pt, minimum size=1pt] at (t1o2o12) {\color{blue}{$\bullet$}};
\node[circle, fill=white,inner sep=1pt, minimum size=1pt] at (t1o2o121) {{$\boldsymbol{\cdot}$}};
\node[circle, fill=white,inner sep=1pt, minimum size=1pt] at (t1o2o121o2) {\dots};

%
%

\node[circle, fill=white,inner sep=1pt, minimum size=1pt] at (t1o2o2) {$\boldsymbol{\cdot}$};
\node[circle, fill=white,inner sep=1pt, minimum size=1pt] at (t1o25) {\dots};

\node[circle, fill=white,inner sep=1pt, minimum size=1pt] at (t15) {$\boldsymbol{\cdot}$};
\node[circle, fill=white,inner sep=1pt, minimum size=1pt] at (t1o5) {\dots};

\node[circle, fill=white,inner sep=1pt, minimum size=1pt] at (to6) {$\boldsymbol{\cdot}$};
\node[circle, fill=white,inner sep=1pt, minimum size=1pt] at (to62) {\dots};
\node[circle, fill=white,inner sep=1pt, minimum size=1pt] at (to6o2) {$\boldsymbol{\cdot}$};
\node[circle, fill=white,inner sep=1pt, minimum size=1pt] at (to65) {$\boldsymbol{\cdot}$};
\node[circle, fill=white,inner sep=1pt, minimum size=1pt] at (to65o2) {\small{$\dots$}};
\node[circle, fill=white,inner sep=1pt, minimum size=1pt] at (to655) {$\boldsymbol{\cdot}$};
\node[circle, fill=white,inner sep=1pt, minimum size=1pt] at (to656) {$\boldsymbol{\cdot}$};
\node[circle, fill=white,inner sep=1pt, minimum size=1pt] at (to6565) {\dots};
\node[circle, fill=white,inner sep=1pt, minimum size=1pt] at (to6566) {$\boldsymbol{\cdot}$};
\node[circle, fill=white,inner sep=1pt, minimum size=1pt] at (to656o5) {\color{blue}{$\bullet$}};
\node[circle, fill=white,inner sep=1pt, minimum size=1pt] at (to656o51) {$\boldsymbol{\cdot}$};
\node[circle, fill=white,inner sep=1pt, minimum size=1pt] at (to656o52) {$\boldsymbol{\cdot}$};

\node[circle, fill=white,inner sep=1pt, minimum size=1pt] at (to656o54) {\color{blue}{$\bullet$}};
\node[circle, fill=white,inner sep=1pt, minimum size=1pt] at (to656o541) {\color{blue}{$\bullet$}};
\node[circle, fill=white,inner sep=1pt, minimum size=1pt] at (to656o5411) {\color{blue}{$\bullet$}};
\node[circle, fill=white,inner sep=1pt, minimum size=1pt] at (to656o55) {$\boldsymbol{\cdot}$};
\node[circle, fill=white,inner sep=1pt, minimum size=1pt] at (to656o56) {$\boldsymbol{\cdot}$};





 \draw [decorate,color=black] (0.5,2.5)
   node[left] {\begin{Large}{$\Gamma_{\Phi}$}\end{Large}};

\end{tikzpicture}
}

\end{subfigure}
\caption{{\small{Identification of type $(\mathscr{V})$  at level $4$, and type  $(\mathscr{E})$ from level $5$, 
when  $\mathscr{\textcolor{black}{z^4}} \in I_1 \setminus 
\{V_1\cup V_2\} $.}}}
\label{i2}
\end{figure}
 \begin{Lemma}\label{Phi-Phi'}
Let $\Phi : S^1 \rightarrow S^1$ satisfying the conditions $\rm(SE), (EC), (E\pm), (CS\textrm{-}\textcolor{black}{\lambda})$ then there exists 
$\Phi' : S^1 \rightarrow S^1$ with the same combinatorics as $\Phi$ so that the identifications of type (ii) in Lemma 
\ref{vertexInterval} are all 2 to 1 for $\Gamma_{\Phi'}$.
 \end{Lemma}
 \begin{proof}
 The idea is to change the map $\Phi$ by changing the cutting points while preserving the combinatorics.
 We replace $\Phi$ by the affine map $\widetilde{\Phi}$ via the conjugacy of condition (CS), the combinatorics are evidently the same. Consider the neighborhood 
 $\widetilde{V}_j$ of Lemma \ref{affin-diffeo} and the $\lambda$-affine extension $\widetilde{\Phi}_j^{\tilde{V}}$
 of Lemma \ref{diffeo-V},
 from which we obtain:
 $\widetilde{\Phi}_j^{\tilde{V}} (\widetilde{V}_j) \subset \tilde{I}_{\delta(j)} \setminus \widetilde{V}_{\delta(j)}$. Conditions 
 (E$\pm$), together with the definition of $\widetilde{V}_j$ give:
 \begin{equation}\label{extenE+-}
   \widetilde{\Phi}^{m} ( \widetilde{\Phi}_j^{\tilde{V}} (\widetilde{V}_j) ) \textcolor{black}{\subset} \tilde{I}_{\delta^{m+1}(j)}  \textrm{ , }  
   \widetilde{\Phi}^{m} ( \widetilde{\Phi}_{\zeta^{-1}(j)}^{\tilde{V}} (\widetilde{V}_j) ) \textcolor{black}{\subset} 
   \tilde{I}_{\gamma^{m+1}(\zeta^{-1} (j))} , 
   \textrm{ } 0\leq m\leq k(j) - 2.
\end{equation}    
\noindent  Condition (EC) gives 
 $\textcolor{black}{\mathscr{z}^{k(j)}_j} = \widetilde{\Phi}^{k(j)} (\tilde{z}_j) \in \tilde{I}_{\alpha}$, for some 
 $\alpha \in \{ 1, \dots, 2N \}$. Consider the fixed point $p^{\alpha} \in \tilde{I}_{\alpha}$ of 
 $\widetilde{\Phi}_{\alpha}$, given by condition (SE) and (CS).
 The definition of the involution $\iota$ and condition (SE) implies that the fixed point $p^{\alpha}$ belongs to the subinterval $I_{\alpha, \alpha}$ which is disjoint from
 $ I_{\alpha, \delta(\alpha)} \cup I_{\alpha, \gamma(\alpha)}$. By definition of the intervals $\widetilde{V}_j$ in Lemma \ref{affin-diffeo}, we obtain: $p^{\alpha} \notin \widetilde{V}_{\alpha} \cup \widetilde{V}_{\zeta(\alpha) }$.
 
 If $\textcolor{black}{\mathscr{z}^{k(j)}_j} = p^{\alpha}$ then we do not change the cutting point $z_j$.
 
  If $\textcolor{black}{\mathscr{z}^{k(j)}_j} \neq p^{\alpha}$ then either 
 $ p^{\alpha} \in \widetilde{\Phi}^{k(j)}(V_j^+)$ for $V_j^+:= \widetilde{V}_j \cap \widetilde{I}_j$ or 
 $ p^{\alpha} \in \widetilde{\Phi}^{k(j)}(V_j^-)$ for $V_j^-:= \widetilde{V}_j \cap \widetilde{I}_{\zeta^{-1}(j)}$.
 Assume, for instance, that $ p^{\alpha} \in \widetilde{\Phi}^{k(j)} ( V_j^+ )$ and let 
 $ p^{\alpha}_j = \widetilde{\Phi}^{- k(j)} ( p^{\alpha}) \in V_j^+ \subset \tilde{I}_{j}$, the other case is symmetric.
 The goal is to transform the map $\widetilde{\Phi}$ to $\Phi'$ so that $ p^{\alpha}_j$ is the new cutting point.
 Since $ p^{\alpha}_j \in V_j^+ \subset \tilde{I}_{j}$ we define 
 $\Phi'_j = \widetilde{\Phi}_j$ \textcolor{black}{restricted to} the interval 
 $[ p^{\alpha}_j  ,  \tilde{z}_{\zeta(j)} )$.
 We define the map $\Phi'_{\zeta^{-1}(j)}$ by the $\lambda$-affine extension of 
 $ \widetilde{\Phi}_{\zeta^{-1}(j)}$, i.e.,  from the map $ \widetilde{\Phi}^{V}_{\zeta^{-1}(j)}$, as defined in (\ref{affine-extension}), restricted to the interval   $[\tilde{z}_{\zeta^{-1}(j)} , p^{\alpha}_j  )$ rather than
 $[ \tilde{z}_{\zeta^{-1}(j)} , \tilde{z}_{j} )$ \textcolor{black}{for the map $ \widetilde{\Phi}_{\zeta^{-1}(j)}$}.
 
 \noindent We apply the same construction for each cutting point.
 The permutations $\zeta$ and $\iota$ as well as all the $k(j)$ are the same for $\Phi'$ as for $\Phi$. 
 The properties (\ref{extenE+-}), coming from  Lemma \ref{diffeo-V}-$(a)$ and (E$\pm$) for $\Phi$, imply that the conditions 
 (E$\pm$) are satisfied by $\Phi'$. The condition (EC) is satisfied by $\Phi'$ from the equality
 $(b)$ in Lemma \ref{diffeo-V}. Condition (CS), with slope $\lambda$ is satisfied by $\Phi'$ by construction since $\Phi'$ is affine with the same slope.
 The two maps $\Phi$ and $\Phi'$ have thus the same combinatorics.\\
 The choice of the new cutting point $p^{\alpha}_j$ of $\Phi'$ implies that 
 $\mathscr{z}' =  \Phi'^{k(j)} (p^{\alpha}_j) = p^{\alpha}$ which is fixed by $\Phi'$ since 
 $\Phi' = \widetilde{\Phi}$ outside the set $\bigcup V_j$ and we observed that
 $p^{\alpha} \notin V_{\alpha} \cup V_{\zeta(\alpha) }$. 
 Therefore we obtain: $ \Phi'^{m} (\mathscr{z}')  = p^{\alpha} $ for all $m \geq 1$ and thus this orbit is always outside the set 
 $\bigcup V_j$.
 By the proof of Lemma \ref{vertexInterval}, the identifications of type  ({\it ii}-$\mathscr{E}$) occur at all levels after level $k(j)$ and do not interact with identifications of type  ({\it ii}-$\mathscr{V}$). Therefore the identifications of type $(ii)$ in Lemma 
\ref{vertexInterval} are all with two intervals, i.e., are 2 to 1.
   \end{proof} 
   
\begin{Rm}\label{Markov}
All the cutting points of the new map $\Phi'$ are pre-periodic and thus the map satisfies a Markov property. \textcolor{black}{This observation also implies that the maps $\Phi$ and $\Phi'$ are non-conjugate.}
\end{Rm} 
\noindent This remark is immediate from a dynamical system point of view, it has an important consequence for our class of maps:

\begin{Lemma}\label{algebraic} If $\Phi$ satisfies the conditions 
$\rm(SE), (E\pm), (EC) ,(CS\textrm{-}\lambda)$ for some $\lambda >1$
then $\lambda$ is an algebraic integer.
\end{Lemma}

\begin{proof}
By Lemma \ref{Phi-Phi'} the two maps $\Phi$ and $\Phi'$ have the same combinatorics, in particular they have the same slope $\lambda >1$. It is classical in one dimensional dynamics that a piecewise affine map $f$ with constant slope 
$\lambda >1$ has positive topological entropy $h(f) = log (\lambda) $. From Remark \ref{Markov}, the map $\Phi'$ satisfies a Markov property and thus it is also classical that its topological entropy is the logarithm of an algebraic integer, as the largest eigenvalue of an integer matrix (see for instance \cite{ALM} for the classical facts).
\end{proof}


 \begin{Lemma}\label{QuasiI}
The two graphs $\Gamma_{\Phi'}$ and $\Gamma^0_{\Phi'}$, endowed with the combinatorial metric (every edge has length one), are quasi-isometric.
 \end{Lemma}
 \begin{proof}
 Let us denote  by $d_{\Gamma_{\Phi'}^0}$ and $d_{\Gamma_{\Phi'}}$ the combinatorial distances in 
 $\Gamma_{\Phi'}^0 $ and $\Gamma_{\Phi'}$.
 The two sets of vertices $V(\Gamma_{\Phi'}^0 )$ and $V(\Gamma_{\Phi'})$ are related by a map
   ${\cal V} : V(\Gamma_{\Phi'}^0 ) \rightarrow V(\Gamma_{\Phi'})$ which is induced by the relation 
   $\sim_{\Phi}$ of Definition \ref{graph} and is at most 2 to 1 by Lemma \ref{Phi-Phi'}.\\
   Each vertex $ v \in V(\Gamma_{\Phi'}^0 )\setminus \{v_0\}$ is associated with an interval 
   $ I_v := I_{j_1, \dots , j_{k}}$ and thus with a vertex of the tree $T_{\Phi'}$.\\
 Two vertices of $\Gamma_{\Phi'}^0 $ with the same $\cal{V}$-image correspond to adjacent intervals at the same level $k$, they are at distance one in 
 $\Gamma_{\Phi'}^0 $. 
 Two vertices connected by an edge on a sphere $S_p^0$ of radius $p$ centred at the base vertex $v_0$ in 
 $\Gamma_{\Phi'}^0 $ are mapped either to a single vertex in the sphere $S_p$ of radius $p$, centred at $v_0$ in $\Gamma_{\Phi'} $ or to two distinct vertices on the same sphere. These two vertices are connected in $\Gamma_{\Phi'} $ by a path of length at most $k(j)$, for some 
 $j \in \{1, \dots , 2N\}$, where $k(j)$ is the integer in the condition (EC). 
 We define $K_{\Phi'} := {\rm max } \{ k(j)| j= 1, \dots, 2N\}$ and we assert that:\\
    \centerline{ $ d_{\Gamma_{\Phi'}} ( {\cal V}(v^0_{\alpha}) , {\cal V}(v^0_{\beta} )  ) \leq K_{\Phi'}. 
  d_{\Gamma_{\Phi'}^0}  (v^0_{\alpha} , v^0_{\beta} ) + \textcolor{black}{1}$,}
   for any pair of vertices $ (v^0_{\alpha} , v^0_{\beta} )$ in 
   $V(\Gamma_{\Phi'}^0)\times V(\Gamma_{\Phi'}^0)$.\\ 
   Indeed a minimal length path between $ v^0_{\alpha}$ and  $v^0_{\beta} $ is a concatenation of some paths along the spheres centred at $v_0$ and some paths along rays starting at $v_0$. The length of the paths along the rays are preserved by the map 
  ${\cal V}$ and the length of the paths along the spheres are at most expanded by a factor bounded by $K_{\Phi'}$.
  On the other direction, the same observation and the 
  fact that ${\cal V}$ could identify at most two vertices implies:\\
  \centerline{ $ d_{\Gamma_{\Phi'}} ( {\cal V}(v^0_{\alpha}) , {\cal V}(v^0_{\beta} )  ) \geq  
  \frac{1}{K_{\Phi'}} d_{\Gamma_{\Phi'}^0}  (v^0_{\alpha} , v^0_{\beta} ) - \textcolor{black}{1}$.}   \end{proof}

  \begin{Cor}\label{Graph1Hyper}
 The dynamical graph $\Gamma_{\Phi}$, with the combinatorial distance, is hyperbolic with boundary homeomorphic to $S^1$. 
 \end{Cor}
 \begin{proof}
 A metric space quasi-isometric to a Gromov hyperbolic space is Gromov hyperbolic with the same boundary (see for instance \cite{GdlH}). By Lemmas \ref{Graph0Hyper} and \ref{QuasiI} the graph 
 $\Gamma_{\Phi'}$ is hyperbolic with boundary $S^1$. By Proposition \ref{GammaHomeo}
 the same property is satisfied by $\Gamma_{\Phi}$.
\end{proof} 
   
   \section{An action of $G_{X_{\Phi}}$ on $\Gamma_{\Phi}$}
  
 The group $G_{X_{\Phi}}$ of Definition \ref{groupGPhi}, as a subgroup of $\textrm{Homeo}^+ (S^1)$, is the main object of study for the rest of the paper. From Theorem \ref{CP-relations}, some relations are satisfied among the generators: the cutting point relations.  We do not know at this point if this is the whole set of relations. 
 The classical method to study such groups is via a geometric action on a metric space.
  The graph $\Gamma_{\Phi}$ of the previous section has been defined for that purpose. It is a hyperbolic metric space that reflects the dynamics of the map $\Phi$ 
but an action of $G_{X_{\Phi}}$ on $\Gamma_{\Phi}$ has to be defined, via the data we have i.e., the dynamics of the map $\Phi$.
  
   Recall that a {\em geometric action} of a group on a metric space is a morphism, acting by isometries that is co-compact and properly discontinuous.\\
   By Lemma \ref{vertexInterval}, each vertex $v \in V( \Gamma_{\Phi} )$ is identified with an interval 
   $I_v \subset S^1$ and each $g \in G_{X_{\Phi}}$ is, in particular, a homeomorphism of $S^1$.
   We need to understand, for each $g \in G_{X_{\Phi}}$ how the interval $g(I_v)$ is related to some 
   $I_w$ for $w \in V( \Gamma_{\Phi} )$.\\
   An ideal situation would be that for ``all $v$ and all $g$ there is $w$ so that $g(I_v) = I_w$", we will see immediately that this does not happen for all vertices (see Lemma \ref{action1}).
    The idea for defining an action is to weaken this ideal situation and find an interval $I_w$ so that   
 $g(I_v)$ and $I_w$ are  ``close enough" i.e., admit a controlled error. In the previous section, we simplified the notations by not making distinctions between $\Phi$ and $\widetilde{\Phi}$ or 
 $\widetilde{I}_j$ and ${I}_j$. We keep this simplification when no confusion is possible and we simplify also the notation for the generators, we write $\varphi_j$ for  $\varphi_j^{\infty}$.
  
  \subsection{A preliminary step}
  \noindent   Let us describe how the generators 
  $\varphi_j \in X_{\Phi}$ of the group do act on the partition intervals $I_m$ for all $m \in \{1,\dots, 2N\}$ i.e., on the intervals associated to vertices of level 1 in $\Gamma_{\Phi}$.
  
  \begin{Lemma}\label{action1}
  If $\Phi$ is a piecewise homeomorphism of $S^1$ satisfying the conditions {\rm(SE), (E$\pm$), (EC)} and {\rm(CS)}, let $\varphi_j \in X_{\Phi}$ be a generator of the group  $G_{X_{\Phi}}$ given by 
  Theorem \ref{CP-relations}.
  If $I_m$ is a partition interval, for $m \in \{1, \dots , 2N\}$, then $\varphi_j (I_m)$ satisfies one of the following conditions:
  \vspace*{-5pt}
     \begin{enumerate}[noitemsep, leftmargin=20pt]
  \item[$(a)$] If $m = j$ then: $\varphi_j (I_j) \cap I_k \neq \emptyset$ for all $k\neq \iota(j)$.
  \item[$(b)$] If $m \notin \{j, \zeta^{\pm 1} (j) \}$ then: $\varphi_j (I_m) = I_{\iota(j) , m}$.
  \item[$(c)$] $\varphi_j (I_{\zeta (j)}) = I_{\iota(j) , \zeta (j)} \cup L_{\iota(j)}$ and 
   $ \varphi_j ( I_{\zeta^{-1} (j)} ) = I_{\iota(j) , \zeta^{-1} (j)} \cup R_{\iota(j)}$,
   where $L_{\iota(j)}$ and $R_{\iota(j)}$ are the intervals defined in 
  (\ref{complement}) such that:
  $L_{\iota(j)} \subsetneq I _{\gamma(j), \gamma^2 (j) , \dots, \gamma^{k(\zeta(j)) - 1} (j)} $ and 
   $R_{\iota(j)} \subsetneq I _{\delta(j), \delta^2 (j) , \dots, \delta^{k( j ) - 1} (j)}$, 
  for $k(\zeta (j) )$ and $k(j)$  the integers of condition {\rm(EC)}.
  \end{enumerate}
   \end{Lemma}
  
\begin{proof}  The proof is a case by case study.
  
\noindent  $(a)$ This is simply condition (SE) on the map $\Phi$, since $\varphi_j (I_j) = \widetilde{\Phi}_{j}(I_j)$.
  
\noindent  $(b)$ By definition of the generators $\varphi_{ j}$ in 
Theorem \ref{CP-relations}, 
  they satisfy:\\  
  $(\varphi_{ j})_{|\widetilde{\Phi}_{\iota(j)} ( I_{\iota(j)} ) } = ( \widetilde{\Phi}_{\iota(j)} )^{-1}_{|\widetilde{\Phi}_{\iota(j)} ( I_{\iota(j)} ) }$. 
 Condition (SE) implies that:
    $ \widetilde{\Phi}_{\iota(j)}(I_{\iota(j)}) \cap I_m = I_m$, for all $m \neq j , \zeta^{\pm 1} (j)$.
   Therefore we obtain:
    $I_{\iota(j)} \cap \widetilde{\Phi}^{-1}_{|I_{\iota(j)}} (I_m) = 
   \widetilde{\Phi}^{-1}_{|I_{\iota (j) }} (I_m)$ which reads: $I_{\iota(j) , m } = \varphi_j (I_m)$ (see Figure \ref{fig:7 (a)}).
    
    \begin{figure}[htbp]
\centerline{\includegraphics[height=40mm]{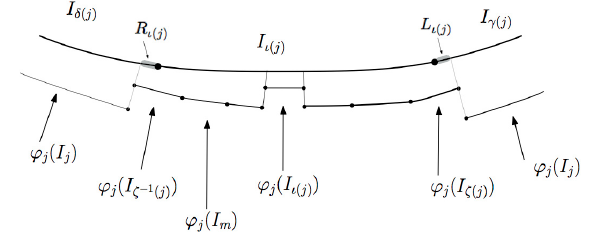} }
\caption{ Image of the intervals $I_m$ under $\varphi_j$}
\label{fig:7 (a)}
\end{figure}

\noindent $(c)$ The two situations are symmetric, we restrict to one of them, for instance to $\varphi_j (I_{\zeta (j) })$.
 By condition (SE), applied to $I_j$ and $I_{\iota (j) }$, we have:
 
\centerline{ (i) $\varphi_j (I_j) \cap I_{\gamma (j)} \subsetneq  I_{\gamma (j)}$, and 
 (ii) $\varphi_{\iota (j) }  (I_{\iota (j) })\cap I_{\zeta (j)} \subsetneq  I_{\zeta (j)}$.}

\noindent  By (i) and the continuity of $\varphi_j$ we have $L_{\iota (j)} = \varphi_j(I_{\zeta(j)})\cap I_{\gamma(j)} \neq \emptyset,$ and by (ii) we have
  $\varphi_j(I_{\zeta(j)}) \cap I_{\iota(j)}\neq \emptyset$.
  On the other hand,
  by
  definition of the generators $\varphi_j$ and Lemma \ref{prop-gen} we have:
  $\varphi_j(I_{\zeta(j)}) \cap I_{\iota(j)}= \widetilde{\Phi}^{-1}_{|I_{\iota (j) }}(I_{\zeta(j)})\cap I_{\iota(j)}=I_{\iota(j),\zeta(j)}.$\\
  Thus, we obtain $\varphi_j (I_{\zeta (j)}) = I_{\iota(j) , \zeta (j)} \cup L_{\iota(j)}$ (see Figure \ref{fig:7 (a)}).
  
  \noindent   To complete the proof we verify the properties of the interval $L_{\iota (j)  }$ 
   (resp. $R_{\iota (j)}$). With the notations of the cutting points, this interval is:\\
 \centerline{  $L_{\iota (j)  } = [  \varphi_j (z_{\zeta (j)} ) , z_{\iota (j) } ]$ (see Figure \ref{fig:7 (a)}).}
 By condition (E+) at $z_{\zeta (j) }$ we have:
 $\forall i,  0 \leq i \leq k ( \zeta ( j ) ) - 2: \tilde{\Phi}^i ( \tilde{\Phi}_j (z_{\zeta (j) } ) )\in I_{\gamma^{i+1} (j) }$.
 For $i = 0$: 
  $\tilde{\Phi}_j (z_{\zeta (j) }) = \varphi_j  (z_{\zeta (j) }) \in I_{\gamma (j) }$, and for 
  $i = 1$: $\tilde{\Phi} ( \tilde{\Phi}_j (z_{\zeta (j) }) ) \in I_{\gamma^2 (j)}$,
    this last condition means that the $\tilde{\Phi}$ image of the point 
  $\varphi_j (z_{\zeta (j) }) \in I_{\gamma (j) }$ belongs to the same partition interval as the 
  $\tilde{\Phi}$ image of the cutting point $z_{\iota (j)}$.
  Therefore the point $\varphi_j (z_{\zeta (j)} )$ belongs to the interior of the last sub-interval of level 2, with respect to the cyclic ordering of $S^1$, of the partition interval $I_{\gamma (j)}$ which is $I_{\gamma (j), \gamma^2 (j)}$.
  This implies that: $ L_{ \iota(j) } \subsetneq I_{\gamma (j), \gamma^2 (j)}$
   which is part of the statement. At this stage we only use the first iterate ($i =1$) in conditions (E+). 
   The proof of $(c)$ is completed by applying the same arguments for all iterates: 
   $ i \leq k ( \zeta (j) ) - 2$ in condition (E+), we obtain:\\
  \centerline{
  $L_{\iota(j)} \subsetneq I _{\gamma(j), \gamma^2 (j) , \dots, \gamma^{k(\zeta(j) ) - 1} (j)}.$}
  
  \vspace{3pt}
\noindent    This completes the proof of statement $(c)$ in this case.
The symmetric situation in case $(c)$, is obtained by replacing $\zeta$ by 
   $\zeta^{-1}$, $\gamma$ by $\delta$ and condition (E+) by (E-).
   \end{proof}
   
   \subsection{Additional properties of $\Gamma_{\Phi}$}
 From the proof of Lemma \ref{action1}, the intervals of level $m \leq k(j) -2$ in the tree $T_{\Phi}$ that are extreme in the interval $I_j$, i.e., that contain a cutting point, are of the form:
  \vspace{-5pt}
\begin{equation}\label{extremIj}
I_{j, \delta(j), \dots, \delta^m (j)}  \textrm{ and }  I_{j, \gamma(j), \dots, \gamma^m (j)} .
\end{equation}

 \vspace{-5pt}
\noindent  The intervals of type $(ii)$ and level $k(j)$ in the proof of Lemma \ref{vertexInterval} are thus of the form:
  \vspace{-5pt}
\begin{equation}\label{Ivtilde}
I_{\tilde{v}_j} := 
I_{j, \delta(j), \dots, \delta^{k(j) - 1 }(j)} \cup I_{\zeta^{-1}(j), \gamma(\zeta^{-1}(j)), \dots, \gamma^{k(j) - 1 } (\zeta^{-1}(j))}, 
\end{equation}

 \vspace{-3pt}
 
\noindent where the first interval is extreme of level $k(j)$ on the $(+)$ side of $\tilde{z}_j$ and the second is extreme on the $(-)$ side of the same cutting point.
 This interval is of type {\it(ii-$\mathscr{V}$)}, see also Remark \ref{V-E-ii}. It contains sub-intervals of level $k(j) + 1$ and possibly one with the cutting point  $z_{_j}$ in its interior, as in case 2) in the proof of Lemma \ref{vertexInterval}:\\
 \centerline{ $I_{\tilde{v}_j , \alpha} :=  I_{ \zeta^{- 1} (j), \gamma ( \zeta^{- 1} (j) ), \dots , 
\gamma^{k(j) -1} ( \zeta^{- 1} (j) ), \alpha}  \cup  I_{ j , \delta ( j ), \dots , 
\delta^{k(j) -1} ( j ), \alpha}$,}

 \vspace{3pt}
\noindent where $\alpha$ satisfies  (\ref{k(j)-image}), this interval is of type {\it(ii-$\mathscr{E}$)}.

\noindent More generally, from Definition \ref{graph} and  Remark \ref{V-E-ii}, an interval of type {\it(ii-$\mathscr{V}$)} is of the form:
\begin{equation}\label{id-ver}
 I_{\tilde{v}_{\ell , j}} := I_{\ell,\zeta^{-1}(j),\gamma(\zeta^{-1}(j)),\dots,\gamma^{k(j)-1}(\zeta^{-1}(j))} \cup 
  I_{\ell,j,\delta(j),\dots,\delta^{k(j)-1}(j)},
  \end{equation}
  
   \vspace{-3pt}
\noindent where $\ell$ is a finite sequence (possibly empty)  in $\{1,\dots,2N\}$.
The vertices, associated to these intervals by Lemma \ref{vertexInterval}, are denoted:  
$ \tilde{v}_j $, $\tilde{v}_{j , \alpha}$ and 
$ \tilde{v}_{\ell , j}$, respectively.\\
The next result induces an additional structure of the graph $ \Gamma_{\Phi}$ around each vertex.
  
 \begin{Prop}\label{cyclicorder}
If the map $\Phi$ satisfies the ruling conditions: {\rm(SE), (EC), (E$\pm$), (CS)} then the set of edges that are incident to a vertex $v \in  V( \Gamma_{\Phi})$ admits a natural cyclic ordering induced by the cyclic ordering of the partition intervals $I_j$ along $S^1$. In addition each vertex has valency $2N$.
\end{Prop}
\begin{proof}
 By definition of $\Gamma_{\Phi}$, the cyclic ordering of the intervals $I_j$ along $S^1$ defines a cyclic ordering of the vertices of level $1$ and thus a cyclic ordering on the edges incident at $v_0$.
 By Proposition \ref{GammaHomeo} the structure of $\Gamma_{\Phi}$ depends only on the combinatorics of the map $\Phi$. To simplify the arguments we assume that the identification of type $(ii)$ are all 2 to 1, as in Lemma \ref{Phi-Phi'}.
 
 If $v=v_{j_1,\dots,j_t} \in  V( \Gamma_{\Phi})$ is a vertex of type $(i)$ or {\it(ii-$\mathscr{E}$)} and level 
 $t \geq 1$: 
then it is connected to one vertex of level $t-1$, i.e., to $v=v_{j_1,\dots,j_{t -1}}$,
and to $2N-1$ vertices of level $t+1$, by condition (SE). 
At level $t+1$, these vertices $w_i$ are ordered by the ordering of the sub-intervals $I_{w_i}$ along the interval  $I_v$,  as sub-intervals of $S^1$. Recall that the ordering along $S^1$ is expressed by the permutation $\zeta$
(see $\S$ \ref{TheClassPhi}). By condition (SE) these vertices at level $t+1$ are:
$\quad   v_{j_1,\dots, j_t,{\zeta(\overline{j_t})}},\; v_{j_1,\dots, j_t,\zeta^2(\overline{j_t})},\;\dots \; , v_{j_1,\dots, j_t,\zeta^{2N-1}(\overline{j_t})} $.
 
\noindent   The edges arriving at these vertices, from $v$, are labelled respectively:\\
 \centerline{ $ \Psi_{\zeta(\overline{j_t})},\dots, \Psi_{\zeta^{2N-1}(\overline{j_t})}$.}
  The vertex at  level $t-1$ is $v_{j_1,\dots,j_{t-1}}$ and the reverse edge, i.e., from   $v$ to it,  is labelled 
  $ \Psi_{\overline{j_{t}}}$. Therefore, the vertices of type $(i)$ or {\it(ii-$\mathscr{E}$)} admit a cyclic ordering of the edges induced by the permutation $\zeta$.

If $v$ is a vertex of type {\it(ii-$\mathscr{V}$)}:
 then, there is $j\in\{1,\dots,2N\}$ and a finite sequence $\ell$ in 
$\{1,\dots,2N\}$ so that $v$ is identified with an interval $I_{\tilde{v}_{\ell , j}}$
as in (\ref{id-ver}).

\noindent From the equivalence relation $\sim_{\Phi}$, the vertex $v$ has 
  two incoming edges and they are adjacent by Lemma \ref{adj}. These two edges are labelled, reading from $v$, as:
  
  \centerline{ $\Psi_{\overline{\delta^{k(j)-1}(j)}} \textrm{  and  } \Psi_{\overline{\gamma^{k(j)-1}(\zeta^{-1}(j))}}$.}
  
\noindent And there are   
   $2N-2$ outgoing edges, ordered by the ordering along $S^1$. By condition (\ref{k(j)-image}) in the proof of Lemma \ref{affin-diffeo}, they are labelled as:
   
\centerline{ $\Psi_{\zeta(\overline{\gamma^{k(j)-1}(\zeta^{-1}(j))})}, \dots, \Psi_{\zeta^{2N-2}(\overline{\gamma^{k(j)-1}(\zeta^{-1}(j))})}$.}
\noindent In all cases, i.e., for the vertices of type $(i)$, {\it(ii-$\mathscr{E}$)} or {\it(ii-$\mathscr{V}$)}, $2N$ edges are incident at $v$ and they are cyclically ordered by the permutation $\zeta$ and thus by the ordering of the intervals along $S^1$. 
\end{proof}

\begin{Cor}\label{loops} To each pair of adjacent edges, for the natural cyclic ordering of Proposition \ref{cyclicorder}, at any vertex $v$, there is exactly one ``cutting point" relation of Theorem \ref{CP-relations} defined by this pair.
\end{Cor}
\begin{proof}
By the proof of Theorem \ref{CP-relations}, each ``cutting point" relation is associated to a cycle of the permutation $\delta$ or $\gamma$. From the proof of Lemma \ref{action1}, a cycle of the  permutation 
$\delta$ or $\gamma$ is also associated to the  orbit of a cutting point $z_j$ under $\Phi$ via conditions (E+), (E-).
In term of the edges in $\Gamma_{\Phi}$, the cycle defines the following loop, given by the sequence of labeled edges:
$\Psi_{\zeta^{-1}(j)}, \;\Psi_{\gamma(\zeta^{-1}(j))},   \dots,  \;\Psi_{\gamma^{k(j)-1}(\zeta^{-1}(j))},\; \Psi_{\overline{\delta^{k(j)-1}(j)}},\dots,\; \Psi_{\overline{\delta(j)}},\; \Psi_{\overline{j}},$\\
 for each $j=1,\dots,2N$, see Figure \ref{Cv0}. 
The two edges labeled $\Psi_{\zeta^{-1}(j)}$ and $\Psi_j$ are adjacent by definition of $\zeta$.
\noindent  Moreover, since the cycles of the permutations are disjoint, each pair of consecutive edges is associated to  exactly one ``cutting point" relation.
\end{proof} 
\begin{Rm}\label{compact} Any vertex $v \in V(\Gamma_{\Phi})$ is contained in a compact set  ${\cal{C}}_v$ defined by the  union of the loops associated to the  pairs of consecutive edges in Corollary \ref{loops} (see Figure \ref{Cv0} for 
${\cal{C}}_{v_0}$).
 The set ${\cal{C}}_{v_0}$ is based at $v_0$ and the vertices at local maximal distance from $v_0$ are vertices of type {\it(ii-$\mathscr{V}$)}, that we will call the extreme vertices of ${\cal{C}}_{v_0}$. The other vertices are of type $(i)$ according to Lemma \ref{vertexInterval}. Indeed, the extreme vertices above are the first one, starting from $v_0$, for which an identification of Definition \ref{graph} can occur.
\end{Rm}
\begin{figure}[ht!]
  \begin{center}
   \resizebox{0.4\textwidth}{!}{%
\begin{tikzpicture}[>=stealth,
  node at every point/.style={%
    decoration={%
      show path construction,
      lineto code={%
        \path [decoration={markings,
          mark=at position 0 with {\fill circle [radius=1pt];},
          mark=at position .5 with {\arrow{>};},
          mark=at position 1 with {\fill circle [radius=1pt];},
        }, decorate] (\tikzinputsegmentfirst) -- (\tikzinputsegmentlast);
      },
    },
    postaction=decorate
  }
  ]
   \draw [node at every point,thick,-](4,0)--(3,0.5)  ;
  \draw [node at every point,thick,-](4,0)--(4.2,-1.2)  ;
   \draw [node at every point,thick,-](4,0)--(3,-0.75)  ;
   \draw[dashed,gray,very thin,thick,-] (4.2,-1.2)--(5,-2)-- (5,-0.75)  ;
 \draw[dashed,gray,very thin,thick,-] (3,-0.75) -- (1.5,0.25)--(3,0.5);
 \draw [node at every point,thick,-](4,0)--(5,-0.75)  ;
 \draw[dashed,gray,very thin,thick,-] (5,-0.75) -- (6.75,0.2)--(5,0.5);
     \draw [decorate,color=black] (4.1,-0.65)
   node[left] {$\dots$};
      \draw [node at every point,thick,-](4,0)--(4,1.2) ;
   \draw [decorate,color=black] (4.15,-0.2)
   node[left] {\scalebox{0.45}{$v_0$}};
   \draw [node at every point,thick,-](4,0)--(5,0.5) ;
      \draw[dashed,gray,very thin,thick,-] (3,0.5) -- (3.25,2.2)--(4,1.2)--(4.8,2.2)-- (5,0.5);
       \draw [node at every point,thick,-](4,1.2)--(4.5,1.81) ;
         \draw [node at every point,thick,-](4,1.2)--(3.55,1.81) ;
    \draw [node at every point,thick,-](3,0.5)--(2.1,0.35)  ;
    \draw [node at every point,thick,-](3,0.5)--(3.1,1.25)  ;
  \draw [decorate,color=black] (3.5,1.75)
  node[right] {\scalebox{0.45}{$\Psi_{\delta(j)}$}};
   \draw [decorate,color=black] (3.58,0.65)
   node[right] {\scalebox{0.45}{$\Psi_j$}};
      \draw [decorate,color=black] (4.2,0.45)
    node[right] {\scalebox{0.45}{$\Psi_{\zeta(j)}$}};
      \draw [decorate,color=black] (2.8,0.15)
    node[right] {\scalebox{0.45}{$\Psi_{\zeta^{-1}(j)}$}};
     \draw [decorate,color=black] (3.15,1)
     node[left] {\scalebox{0.45}{{ $\Psi_{\gamma(\zeta^{-1}(j))}$}}};
   \end{tikzpicture}}
\caption{The compact set ${\cal{C}}_{v_0}$ in $\Gamma_{\Phi}$}
\label{Cv0}
\end{center}
\end{figure}
 \subsection{How the generators do act on the vertices of ${\cal{C}}_{v_0}$?}
    \noindent In this part we study the action of each generator $\varphi_j$ on the set of intervals corresponding to the vertices of the compact set ${\cal{C}}_{v_0}$. Lemma \ref{action1} is the first step and most of the arguments are exactly like in its proof. Observe that ${\cal{C}}_{v_0}$ is contained in the ball of $\Gamma_{\Phi} $: 
 $ \textrm{Ball}(v_0 , K_{\Phi})$ where the radius  $K_{\Phi} = \textrm{max} \{  k(j): j \in \{1, \dots, 2N\} \} $ was defined in the proof of Lemma \ref{QuasiI}.

\begin{Prop}\label{type-1}
With the above definitions and notations, the image under $\varphi_j$ of the intervals $I_v$ of type $(i)$ in Lemma \ref{vertexInterval}, associated to the vertices in $ {\cal{C}}_{v_0}$, are given by the following cases:
  \vspace*{-5pt}
     \begin{enumerate}[noitemsep, leftmargin=15pt]
\item If the cutting point $z_j $ is a  boundary point of  $I_v$ then,   
for $0 \leq m \leq k(j)-2$: 
\begin{enumerate}[noitemsep,leftmargin=7pt]
 \item if  $I_v = I_{j, \delta (j), \delta^2 (j) , \dots, \delta^m (j)}$  then
 $\varphi_j ( I_{v} ) \subset I_{ \delta(j), \dots , \delta^m (j)}$ 
 with 
$\varphi_j ( I_{v} ) \cap  I_{ \delta(j), \dots , \delta^m (j), \alpha}  \neq \emptyset$,
for all possible such $\alpha \in \{ 1, \dots, 2N \}$, i.e., all except one,
 \item if $I_v = I_{\zeta^{-1} (j), \dots, \gamma^m (\zeta^{-1} (j))}$ 
then $ \varphi_j ( I_{v} )= I_{\iota (j) , \zeta^{-1} ( j), \dots, 
\gamma^m (\zeta^{-1} ( j))} \cup  R_{\iota ( j)}$, where
 $R_{\iota ( j)} $ 
satisfies the properties $(c)$ in Lemma \ref{action1}.
\end{enumerate}
 \item  If $z_{\zeta (j)} $ is a  boundary point of  $I_v$ then,   
for $0 \leq m \leq k(\zeta(j))-2$: 
\begin{enumerate}[noitemsep,leftmargin=7pt]
 \item if  $I_v = I_{j, \gamma (j), \gamma^2 (j) , \dots, \gamma^m (j)}$  then
 $\varphi_j ( I_{v} ) \subset I_{ \gamma(j), \dots , \gamma^m (j)}$ 
 with 
$\varphi_j ( I_{v} ) \cap  I_{ \gamma(j), \dots , \gamma^m (j), \alpha}  \neq \emptyset$,
for all possible such $\alpha \in \{ 1, \dots, 2N \}$,  i.e., all except one,
 \item if $I_v = I_{\zeta (j), \delta(\zeta(j))\dots, \delta^m (\zeta(j))}$ 
then $ \varphi_j ( I_{v} )= I_{\iota (j) , \zeta ( j), \dots, 
\delta^m \zeta(j)} \cup  L_{\iota ( j)}$, where
 $L_{\iota ( j)} $ 
satisfies the properties $(c)$ in Lemma \ref{action1}.
\end{enumerate}
 \item  If $I_v$ is of type $(i)$ and does not contain $z_{j}$ or $z_{_{\zeta(j)}}$  as a boundary point then it 
 has the form: $I_v = I_{j_1, \dots, j_r}$ for $j_1 \neq j $ and
 $\varphi_j (I_v) = I_{\iota(j), j_1, \dots, j_r}$. 
 \end{enumerate}
 \end{Prop}

\begin{proof}
Let $v\in V(\Gamma_{\Phi})\cap {\cal{C}}_{v_0}$.

\vspace{3pt}
\noindent 1) If the cutting point $z_j$ belongs to the boundary of $I_v$ and $v$ is a vertex of type $(i)$ according to Lemma \ref{vertexInterval} then it is given by (\ref{extremIj}), 
 the corresponding set of intervals are:
 
\noindent $(a)$
If $I_v = I_{j, \delta (j), \delta^2 (j) , \dots, \delta^m (j)}$.\\
For $m=1$, the \textcolor{black}{definition of $I_{j, \delta(j)}$ and the} argument in the proof of Lemma \ref{action1}-$(a)$ imply:
  $\varphi_j (I_{j, \delta(j)}) \subset I_{\delta (j)}$ and 
  $\varphi_j (I_{j, \delta(j)}) \cap  I_{\delta (j), \alpha} \neq \emptyset$ for all such possible 
  $\alpha$.\\
  The same argument applies for all $1 \leq m \leq k (j) - 2$ and we obtain the statement $1$-$(a)$ in this case.

 \noindent $(b)$ If $I_v= I_{\zeta^{-1} ( j), \gamma(\zeta^{-1} ( j)), \dots , \gamma^m(\zeta^{-1} ( j))}$,  the arguments in the proof of Lemma \ref{action1}-$(c)$ apply and we obtain, for all 
 $1 \leq m \leq k (j) - 2$: 

 $\varphi_j ( I_{\zeta^{-1} (j), \gamma (\zeta^{-1} (j)), \gamma^2 (\zeta^{-1} (j)) , \dots, 
  \gamma^m  ( \zeta^{-1} ( j ) )  }) =
   I_{\iota(j), \zeta^{-1} (j) , \gamma (\zeta^{-1} (j)) , \dots, \gamma^m (\zeta^{-1}(j) )} 
   \cup R_{\iota (j)}$,\\
 where  $ R_{\iota (j)} \subsetneq 
  I_{ \delta (j) ,  \delta^{2} ( j), \dots ,  \delta^{k(j) - 1} ( j )}$ by Lemma \ref{action1}-(c).
 
  \vspace{3pt} 
 \noindent 2) If $z_j$ is replaced by $z_{_{\zeta(j)}}$ then $\delta(j)$ is replaced by $\gamma(j)$,  
  the condition (E+) is replaced by (E-) and the same arguments as in the previous cases apply, by symmetry.

\vspace{3pt} 
 \noindent 3) If the interval $I_v$ of type $(i)$, level $r > 1$ does not contain the cutting points $z_{\zeta(j)}$ or $z_j$ then it has the form $I_{j_1, \dots, j_r}$ with $j_1 \neq j$ and if 
 $j_1 = \zeta^{\pm 1}( j )$ then 
 $j_2 \neq \gamma( j ) \textrm{ or } \delta(\zeta^{-1} ( j ) )$. In these cases, the arguments in Lemma \ref{action1}-$(b)$ apply and $\varphi_j  ( I_{j_1, \dots, j_r} ) = I_{\iota (j) , j_1, \dots, j_r} $. 
 \end{proof}

For the next result we consider the intervals of type $(ii)$ of  Lemma \ref{vertexInterval}. They are the ``extreme vertices" of ${\cal{C}}_{v_0}$ in Remark \ref{compact} (see Figure \ref{Cv0}).
  At level $k(j)$, around $z_j$, they are given by the intervals $I_{\tilde{v}_j}$ in (\ref{Ivtilde}).

\begin{Prop}\label{type-2}
With the above definitions and notations the image, under $\varphi_j$, of the intervals 
$I_{\tilde{v}_m}$ of type {\it(ii)} associated to the vertices in $ {\cal{C}}_{v_0}$ are given by the following cases:
\vspace*{-20pt}
     \begin{enumerate}[noitemsep, leftmargin=15pt]
\item[$1)$]  $ \varphi_j ( I_{\tilde{v}_{_j}} ) \subset I_{\delta (j), \delta^2 (j), \dots, \delta^{k(j) - 1} ( j )}$ 
(resp. $ \varphi_j(I_{\tilde{v}_{_{\zeta(j)}}} )\subset  I_{ \gamma (j) ,  \gamma^2 (j), \dots , \gamma ^{k(\zeta(j)) - 1} ( j )}$)
and it intersects all sub intervals of level $k(j)$ (resp. $k(\zeta(j))$), except one.

\item[$2)$]  $\varphi_j ( I_{\tilde{v}_{_l}} )
 = \textcolor{black}{ I_{\tilde{v}_{\iota(j)},l}} ,$ with the notation (\ref{id-ver}) for $l \notin \{j, \zeta(j)\}.$
\end{enumerate}
\end{Prop}

\begin{proof}
1)
 From the definition of the neighborhood $V_j$ in Lemma \ref{affin-diffeo}, we observe that the interval $I_{\tilde{v}_j}$ of (\ref{Ivtilde}) satisfies: $V_j =  I_{\tilde{v}_j}$.

 \noindent  The generators $\varphi_j$ given by Theorem \ref{CP-relations}, together with
 Lemma \ref{diffeo-V} gives:

 \centerline{ $\varphi_j ( I_{\tilde{v}_j} ) \subset I_{\delta(j)} \setminus 
I_{\tilde{v}_{\delta(j)}}$.}
\vspace*{3pt} 
  \noindent From the construction of the neighborhood $V_j $ in Lemma \ref{affin-diffeo} we obtain:\\
  $\Phi^m ( \varphi_j ( I_{\tilde{v}_j} ) ) \subset I_{{\delta}^{m+1} (j)}$ for all $m = 0, \dots, k(j) -2$, and thus :
  $\varphi_j ( I_{\tilde{v}_j} ) \subset I_{\delta (j), \delta^2 (j), \dots, \delta^{k(j) - 1} ( j )}$.
  
\noindent   For the next iterate of $\Phi$, the condition (\ref{k(j)-image}) implies:\\
    \centerline{ $\Phi^{k(j)-1} ( \varphi_j ( I_{\tilde{v}_j} ) ) \cap I_m \neq \emptyset$ for all 
    $m \neq \overline{\gamma^{k(j)-1}}(\zeta^{-1}(j)), \overline{\delta^{k(j)-1}}(j)$.}
     
\noindent We observe that 
 $I_{\tilde{v}_{\delta(j)}} \cap I_{\delta(j)}  $ is a subinterval of 
 $I_{\delta (j), \delta^2 (j), \dots, \delta^{k(j) } ( j )}$ of level $k(j)$ and by 
 Lemma \ref{diffeo-V}-(a)
  we have:
  $\varphi_j ( I_{\tilde{v}_j} ) \cap I_{\tilde{v}_{\delta(j)}} = \emptyset$.
 Therefore $\varphi_j ( I_{\tilde{v}_j} ) $ intersects all subintervals of level $k(j)$ of $I_{\delta (j), \delta^2 (j), \dots, \delta^{k(j) - 1} ( j )}$, except one, i.e.,
 $I_{\tilde{v}_{\delta(j)}} \cap I_{\delta(j)}  $.

\vspace{5pt} 
 If $I_{\tilde{v}_j}$ is replaced by $I_{\tilde{v}_{_{\zeta(j)}}}$ then the same arguments apply by replacing $\delta(j)$ with $\gamma(j)$.

\noindent 2) For $I_{\tilde{v}_l}=I_{\zeta^{- 1} (l), \gamma ( \zeta^{- 1} (l) ), \dots , \gamma^{k(l) -1} ( \zeta^{- 1} (l) )} \cup I_{l , \delta ( l ), \dots , \delta^{k(l) -1} ( l )},$ we have
\vspace{3pt}
\\
\centerline{$\varphi_j(I_{\zeta^{- 1} (l), \gamma ( \zeta^{- 1} (l) ), \dots , \gamma^{k(l) -1} ( \zeta^{- 1} (l) )})=I_{\iota(j),\zeta^{- 1} (l), \gamma ( \zeta^{- 1} (l) ), \dots , \gamma^{k(l) -1} ( \zeta^{- 1} (l) )} $}

\vspace{3pt}
\noindent and $\varphi_j( I_{l , \delta ( l ), \dots , \delta^{k(l) -1} ( l )})= I_{\iota(j),l , \delta ( l ), \dots , \delta^{k(l) -1} ( l )}$
because $I_{\tilde{v}_l} \subset \Phi_{\iota(j)}(I_{\iota(j)})$ and 
$l \notin \{j, \zeta(j)\}$, 

\vspace{3pt} \noindent this is the same argument as in case (3) of Proposition \ref{type-1}.
 Then:\\
$\varphi_j(I_{\tilde{v}_l})=I_{\iota(j),\zeta^{- 1} (l), \gamma ( \zeta^{- 1} (l) ), \dots , \gamma^{k(l) -1} ( \zeta^{- 1} (l) )} \cup I_{\iota(j),l , \delta ( l ), \dots , \delta^{k(l) -1} ( l )}
=\textcolor{black}{ I_{\tilde{v}_{\iota(j)},l}},$ with the notation  (\ref{id-ver}).
\end{proof}
  
   

\subsection{The action}  
We define here a map 
  ${\cal A}_g : \Gamma_{\Phi} \rightarrow \Gamma_{\Phi}$ for all $g \in G_{X_{\Phi}}$,
       Lemma \ref{action1}, the Propositions \ref{type-1} and \ref{type-2} are guide lines to this aim.  
  From Lemma \ref{vertexInterval}, each vertex $v\neq v_0$ of $\Gamma_{\Phi}$ is identified with an interval $I_v$ of $S^1$,
  and each $g \in G_{X_{\Phi}}$ maps $I_v$ to $g( I_v)$, another interval of $S^1$. 
  We have to understand how each interval $g( I_v)$ is related to some interval $I_w$, for a vertex $w$ of $\Gamma_{\Phi}$.  Lemma \ref{action1} implies, in particular, that we cannot expect: $``g( I_v) = I_w"$ for all intervals $I_v$. But it shows that if we allow a ``small" error then we can associate to $g( I_v)$ an interval $I_w$. This is one way to interpret Lemma \ref{action1},  its consequences in Proposition \ref{type-1}  and \ref{type-2} and the following definition. 
   
  \begin{definition}\label{action}
   Let $G_{X_{\Phi}}$ be the group of Definition \ref{groupGPhi},
and let $\Gamma_{\Phi}$ be the dynamical graph of Definition \ref{graph} with vertex set
   $V (\Gamma_{\Phi} )$. For each $v\in V (\Gamma_{\Phi} )$, let $I_v$ be the interval associated to $v$ by Lemma \ref{vertexInterval}.
For each generator $\varphi_j \in X_{\Phi}$, $ j=1,\dots,2N$, let
\centerline{${\cal A}_{\varphi_j} :V( \Gamma_{\Phi}) \rightarrow V(\Gamma_{\Phi})$}

\noindent be a map defined
   as follows:
   \vspace{-5pt}
   \begin{enumerate}[noitemsep, leftmargin=17pt]
\item
  If $ v \neq v_0 $ and  $ \varphi_j (I_v)$ intersects all partition intervals $I_k$ of level one except one, then:  
  ${\cal A}_{\varphi_j} (v) := v_0$
\item   If $  v\neq v_0 $ and there exists $ w \in V (\Gamma_{\Phi} ) $  such that 
  $\varphi_j (I_v) \subseteq I_w$ and $\varphi_j (I_v)$ intersects all subintervals
   $ I_{w'} \subset I_w$ of level one more than $w$, except possibly one, then: \\
   ${\cal A}_{\varphi_j} (v) := w$
\item 
\begin{enumerate}[noitemsep, leftmargin=15pt]
\item[$(i)$] If $  v\neq v_0 $ and there exists $ w \in V (\Gamma_{\Phi} ) $ a vertex of type $(i)$ or {\it(ii-$\mathscr{E}$)} in Lemma \ref{vertexInterval} such that 
$ I_w \subset \varphi_j  (I_v) $ and no other $I_{w'}$, for $w '$ of the same level as $w$ is contained in  $\varphi_j  (I_v)$ 
  then: ${\cal A}_{\varphi_j } (v) := w$
  \item[$(ii)$] If $  v\neq v_0 $ and there exists $ w \in V (\Gamma_{\Phi} ) $ a vertex of type {\it(ii-$\mathscr{V}$)} in Lemma \ref{vertexInterval} such
  that $ I_w \subset \varphi_j  (I_v) $ and 
  $\varphi_j ( I_v) $ does not contain $I_{w'}$ for $w '$ of level one less that $w$ 
  then:  ${\cal A}_{\varphi_j } (v) := w$
  \end{enumerate}
\item ${\cal A}_{\varphi_j} (v_0) := v_{\iota(j)}$
  \end{enumerate}
  If $g=\varphi_{n_{1}}\circ \dots \circ \varphi_{n_{k}}$ we define 
 $ {\cal A}_{_{g}}:={\cal A}_{\varphi_{n_1}}\circ \dots \circ {\cal A}_{\varphi_{n_k}}$.  
  \end{definition}
    
  The goal of this subsection is to show that the map $ {\cal A}_{_{g}}$ is well defined and can be extended to a map on the graph $\Gamma_{\Phi}$. We will have to check, in particular, that the map $ {\cal A}_{_{g}}$ does not depend on the expression, in the generators, of the element $g$. 
  
   The next subsection will be about proving that this map defines a geometric action. These are the main technical parts of the proof.
  
The definition of the map ${\cal A}_g$ is new and not standard. As a warm up, let us check it is well defined
   for each generator $\varphi_j$ on the vertices of level 1. 
 For this, we compute $\varphi_j( I_m)$ for $j \textrm{ and } m  \in \{1, \dots, 2N\} $ and Lemma \ref{action1} gives all the possibilities:
 
\noindent $\bullet$ If $m= j$ then, case $(a)$ in Lemma \ref{action1} and case $1$ of Definition \ref{action} gives:
  ${\cal A}_{\varphi_j} (v_j) = v_0.$\\
  \noindent $\bullet$ If $ m \neq j, \zeta^{\pm 1} (j)$ then, case $(b)$ of Lemma \ref{action1} and case $2$ of Definition \ref{action} gives:\\
 ${\cal A}_{\varphi_j} (v_m) = v_{\iota(j) , m}$.\\
 \noindent $\bullet$ If $ m =  \zeta^{\pm 1} (j)$ then, case $(c)$ of Lemma \ref{action1} and case $3$-$(i)$ of Definition \ref{action} gives:\\
 ${\cal A}_{\varphi_j} (v_{m} ) = v_{\iota(j) , m}$.
 
\noindent  With case $4$ in Definition \ref{action}, we obtain, for each generator $\varphi_j$, that  
 ${\cal A}_{\varphi_j}$ maps the ball of radius one centred at $v_0$ in $\Gamma_{\Phi}$, to the ball of radius one centred at $v_{\iota (j)}$.

    \begin{Prop}\label{WellDef-Cv0}
 The map $ {\cal A}_{\varphi_j}$ of Definition \ref{action} is well defined for all the vertices  in the compact set ${\cal{C}}_{v_0}$ of Remark \ref{compact}, for all $j \in \{ 1, \dots, 2N\}$. 
    \end{Prop}
     \begin{proof} We already checked that  $ {\cal A}_{\varphi_j}$  is well defined for the vertices of level 
  $\leq 1$.  Let us verify this property for all the vertices in ${\cal{C}}_{v_0}$.

\vspace{5pt}
\noindent 1)  If $v$ is a vertex of type $(i)$ in ${\cal{C}}_{v_0}$, the image of the corresponding interval by $\varphi_j$ is given by Proposition \ref{type-1}. For these cases  either 
$v = v_{j_1,\gamma(j_1), \dots, \gamma^n(j_1)}$ or  
  $v = v_{j_1,\delta(j_1), \dots, \delta^n(j_1)}$, for some $ j_1$ and  $n\leq  k(j_1)-2$.
 \vspace{-5pt}
   \begin{enumerate}[noitemsep, leftmargin=19pt] 
\item[(a)]  If $j_1\neq j, \zeta^{\pm 1} (j)$, Proposition \ref{type-1} case (3) gives:
$\varphi_j ( I_{j_1,j_2, \dots, j_n}) =  I_{ \iota(j), j_1,j_2, \dots, j_n}$
 and by Definition \ref{action} case $2$:\\
\centerline{ \textrm{either }${\cal{A}}_{\varphi_j}(v) = v_{\iota(j), j_1,\gamma(j_1), \dots, \gamma^n(j_1)} \textrm{ or } 
    {\cal{A}}_{\varphi_j}(v)=v_{\iota(j),j_1,\delta(j_1), \dots, \delta^n(j_1)}$.}
    \item[(b)] If $j_1 = j$, then Proposition \ref{type-1} case $2$-$(a)$ gives: $\varphi_j ( I_{j,\gamma(j), \dots, \gamma^n(j)} ) \subset  I_{ \gamma(j), \dots, \gamma^n(j)}$,
(resp.: $\varphi_j ( I_{j,\delta(j), \dots, \delta^n(j)} ) \subset  I_{ \delta(j), \dots, \delta^n(j)}$). In addition $\varphi_j ( I_{v})$ intersects all subintervals
of level $n+1$. By Definition \ref{action} case $2$ we obtain:\\
\centerline{\textrm{either } ${\cal{A}}_{\varphi_j}(v) = v_{\gamma(j), \dots, \gamma^n(j)} \textrm{ or } 
    {\cal{A}}_{\varphi_j}(v) = v_{\delta(j), \dots, \delta^n(j)}$.}
    \item[(c)]  If $j_1 = \zeta^{\pm 1} (j)$, for instance $j_1 = \zeta^{- 1} (j)$, Proposition \ref{type-1} case $1$-$(b)$ gives:\\ 
 \centerline{ $ \varphi_j ( I_{v} )= I_{\iota (j) , \zeta^{-1} ( j), \gamma (\zeta^{-1} ( j)), \dots, 
\gamma^n (\zeta^{-1} ( j))} \cup  R_{\iota ( j)}$, with 
$R_{\iota ( j)} \subset I_{ \delta(j), \dots, \delta^{k(j) -1}(j)}$,}
and there are two different situations: 
  \vspace{-5pt}
   \begin{enumerate}[noitemsep, leftmargin=8pt] 
 \item[(i)] If $n < k(j) - 2$, then 
 $I_{\iota (j) , \zeta^{-1} ( j), \gamma (\zeta^{-1} ( j)), \dots, 
\gamma^n (\zeta^{-1} ( j))}$ is an interval of type $(i)$ and level $n +2 \leq k(j) - 1$
and $R_{\iota ( j)}$ is contained in an interval of level $k(j) -1$.
 Definition \ref{action} case $3$-$(i)$ gives:\\
 \centerline{ ${\cal{A}}_{\varphi_j}(v) = v_{\iota(j), \zeta^{- 1} (j),\gamma(\zeta^{- 1} (j)), \dots, \gamma^n(\zeta^{- 1} (j)) } $.}
\item[(ii)] If $n = k(j) - 2$, then 
 $I_{\iota (j) , \zeta^{-1} ( j), \gamma (\zeta^{-1} ( j)), \dots, 
\gamma^{k(j) -2} (\zeta^{-1} ( j))}$ is an interval of level $ k(j) $
and $R_{\iota ( j)}$ is contained in an interval of level $k(j) -1$ and thus does not contain an interval of level $k(j) -1$.
Recall that the interval of type $(ii)$ containing the cutting point $z_{\delta(j)}$ is given by 
 (\ref{Ivtilde}): 
$ I_{\tilde{v}_{\delta(j)}} = I_{\iota (j) , \zeta^{-1} ( j), \gamma (\zeta^{-1} ( j)), \dots, 
\gamma^{k(j) - 2} (\zeta^{-1} ( j))} 
 \cup I_{ \delta ( j ), \dots , \delta^{k(j)} ( j ) }$.\\ 
 By Lemma \ref{diffeo-V} it satisfies:
 $\varphi_j ( I_{\tilde{v}_{j}} ) \cap I_{\tilde{v}_{\delta(j)}} = \emptyset $, which implies:\\
 $ R_{\iota ( j)} \cap I_{\tilde{v}_{\delta(j)}} = I_{ \delta ( j ), \dots , \delta^{k(j)} ( j ) }$,
  these equalities together give:
 
   \centerline{ $ \varphi_j ( I_{v} ) = I_{\tilde{v}_{\delta(j)}} \cup 
   [ R_{\iota ( j)} \setminus I_{ \delta ( j ), \dots , \delta^{k(j)} ( j ) }] $.} 
      
\noindent   Therefore $ \varphi_j ( I_{v} )$ contains the interval of type $(ii)$
$I_{\tilde{v}_{\delta(j)}}$
of level $k(j)$ and does not contain any interval of level $k(j) -1$. Thus, by Definition \ref{action} case $3$-$(ii)$ we obtain:\\
\centerline{ ${\cal{A}}_{\varphi_j}(v) = v_{\tilde{v}_{\delta(j)}} $.}
\end{enumerate}
\end{enumerate}
\noindent   2) If $v$ is a vertex of type $(ii)$, the image of the corresponding interval under $\varphi_j$ is given by Proposition \ref{type-2}, which gives:
 \vspace{-5pt}
   \begin{enumerate}[noitemsep, leftmargin=19pt] 
 \item[(a)] \textcolor{black}{if $v=\tilde{v}_{_j}$},  $ {\cal{A}}_{\varphi_j }( \tilde{v}_{_j})= v_{\delta (j), \delta^2 (j), \dots, \delta^{k(j) - 1} ( j )}$ and $ {\cal{A}}_{ \varphi_j}(  \tilde{v}_{_{\zeta(j)}})=  v_{ \gamma (j) ,  \gamma^2 (j), \dots , \gamma ^{k(\zeta(j)) - 1} ( j )}$,
 by Proposition \ref{type-2} case $1)$ and  Definition \ref{action} case $2$. 
\item[(b)]\textcolor{black}{ if $v=\tilde{v}_{_n}$} for 
$n \notin \{j, \zeta(j)\},$ $\textcolor{black}{{\cal{A}}_{\varphi_j }}( \tilde{v}_{_n})=\tilde{v}_{\iota(j),n}$,  
by Proposition \textcolor{black}{\ref{type-1}} case $3$ and Definition \ref{action} case $2$.
 \end{enumerate}
 This completes the case by case proof for all the vertices in ${\cal C}_{v_0}$.
     \end{proof} 
     
\begin{Rm}\label{others}
The vertices studied in Proposition \ref{WellDef-Cv0} are associated to intervals containing a cutting point, either in its boundary or in its interior. There are many other intervals, they are of the form $I_{j_1,j_2,\dots,j_r}$ where
       $j_2 \notin \{\gamma(j_1),\delta(j_1)\}$ or $j_2 \in \{\gamma(j_1),\delta(j_1)\}$  and\\ 
        $j_3 \notin \{\gamma^2(j_1),\delta^2(j_1)\}$ and so on. 
       Suppose that $v_{j_1,j_2,\dots,j_r}$ is a vertex associated to such an interval then:
\vspace{-7pt}
   \begin{enumerate}[noitemsep, leftmargin=17pt]            
\item[$(1)$] If $j_1=j$: then $  {\cal{A}}_{\varphi_j }(v_{j ,j_2,\dots,j_r})= v_{j_2, \dots , j_r},$  by the definition of $I_{j_1,j_2,\dots,j_r}$ and Definition \ref{action} case $2$.
\item[$(2)$] If $j_1\neq j$: then
        ${\cal{A}}_{\varphi_j }(v_{j_1,j_2,\dots,j_r}) = v_{\iota(j), j_1, j_2,\dots, j_r},$  by  the definition of $I_{j_1,j_2,\dots,j_r}$,
         and Definition \ref{action} case $(2)$.     
 \end{enumerate}
\end{Rm}
 The following result is a co-compactness property for the map $ {\cal A}$. 
     \begin{Prop}\label{PhiAction}
    For any vertex $v \in V (\Gamma_{\Phi} ) \setminus\{v_0\} $ of level $n$, there exists a group element 
    $g \in G_{X_{\Phi}}$ of length $l \leq n$ so that: $ {\cal A}_{g} (v) \in {\cal{C}}_{v_0} $.
   \end{Prop}
   
\begin{proof}
Assume that $v$ is of type $(i)$ and let $I_v = I_{j_1, j_2, \dots , j_n}$. 
 If $I_v$ does not contain a cutting point $z_{j_1}$  or $z_{\zeta(j_1 )}$ on its boundary, then by Remark \ref{others} case $(1)$, we have:
$I_v \subset \textrm{int} (I_{j_1})$ and 
$\varphi_{j_1} (I_v) = I_{ j_2, \dots , j_n}$ is an interval of type $(i)$ and level $n-1$ and thus:
$ {\cal A}_{\varphi_{j_1}} (v) = v_{ j_2, \dots , j_n}$.

If $I_v = I_{j_1, j_2, \dots , j_n}$ is of type $(i)$ and contains $z_{j_1}$  or $z_{\zeta(j_1 )}$ on its boundary then $\varphi_{j_1} (I_v) \subset I_{ j_2, \dots , j_n}$ and intersects all subintervals of level $n$, as in Proposition \ref{type-1} case $1$-$(a)$ and thus: 
$ {\cal A}_{\varphi_{j_1}} (v) = v_{ j_2, \dots , j_n}$ is a vertex of level $n-1$.

If $I_v = I_{j_1, j_2, \dots , j_n}$ is of type $(ii)$ and does not contain $z_{j_1}$ or $z_{\zeta(j_1 )}$  then, as above we obtain: $ {\cal A}_{\varphi_{j_1}} (v) = v_{ j_2, \dots , j_n}$ is a vertex of level $n-1$.

If $I_v = I_{j_1, j_2, \dots , j_n}$ is of type $(ii)$ and contains $z_{j_1}$  or $z_{\zeta(j_1 )}$ on its
interior then, as in Proposition \ref{type-2} case (1),
$\varphi_{j_1} (I_v) \subset I_{ j_2, \dots , j_n}$ and intersects all subintervals of level $n$ maybe except one and thus $ {\cal A}_{\varphi_{j_1}} (v) = v_{ j_2, \dots , j_n}$ is a vertex of level $n-1$.

In all cases, there is a generator $\varphi_{j_1}$ so that 
$ {\cal A}_{\varphi_{j_1}} (v)$ is a vertex of level $n-1$. By iterating this argument, we obtain a finite sequence of generators: $\varphi_{j_1}, \varphi_{j_2}, \dots, \varphi_{j_m}$ with $m\leq n -1$ so that:
$ {\cal A}_{\varphi_{j_m} \circ \dots \circ \varphi_{j_1}} (v) \in {\cal{C}}_{v_0}$.
\end{proof}

 Let us extend the map ${\cal{A}}_{\varphi_{j}}$, defined on the vertices of $\Gamma_{\Phi}$, to a map on the graph.
We denote by \textcolor{black}{$(v,w)$} the edge connecting the  vertices $v$ and $w$ in $\Gamma_{\Phi}$. 

\begin{Prop}\label{LocIsom} The map ${\cal{A}}_{\varphi_{j}}$  is well defined on the vertex set 
$V (\Gamma_{\Phi} )$. It extends to a well defined map on the set of edges as:
${\cal{A}}_{\varphi_{j}} (v, w) := ( {\cal{A}}_{\varphi_{j}} (v) , {\cal{A}}_{\varphi_{j}} (w) )$
 and is a bijective isometry, for  $j =1,\dots, 2N$, for the combinatorial metric on 
 $\Gamma_{\Phi}$.
\end{Prop} 
\begin{proof} 
\textcolor{black}{ By Remark \ref{others}, Propositions \ref{WellDef-Cv0} and \ref{PhiAction} each map ${\cal{A}}_{\varphi_{j}}$ is well defined on $V(\Gamma_{\Phi})$.}  
  It is enough to prove the result for the compact set ${\cal C}_{v_0}$.

\noindent  Let $(v,w)$ be an edge in ${\cal C}_{v_0}$, we can assume $v$ is 
  $v_{j_1,\gamma(j_1), \dots, \gamma^n(j_1)}$,\\
   (resp. $ v = v_{j_1,\delta(j_1), \dots, \delta^n(j_1)}$),
  and $w = v_{j_1,\gamma(j_1), \dots, \gamma^{n+1}(j_1)}$, (resp. $ w = v_{j_1,\delta(j_1), \dots, \delta^{n+1}(j_1)}$), for some $n\leq  k(j_1)-2$. \textcolor{black}{We compute the image of each vertex following the proof of  Proposition \ref{WellDef-Cv0}.}
  
\noindent 1) If $n < k(j_1)-2$ then the two vertices are of type $(i)$, this gives the following cases:
\vspace{-7pt}
   \begin{enumerate}[noitemsep, leftmargin=23pt]     
\item[(a)] If $j_1\neq j$, then by  case 1)-(a) and 1)-(c)-(i), the image
of each vertex gives that:\\
\centerline{ $({\cal{A}}_{\varphi_{j}}(v),{\cal{A}}_{\varphi_{j}}(w))$ is an edge in 
${\cal C}_{v_{\iota(j)}}$.}
\item[(b)] If $j_1= j$, then by  case 1)-(b), the image of each vertex 
gives that:\\  
\centerline{  $({\cal{A}}_{\varphi_{j}}(v),{\cal{A}}_{\varphi_{j}}(w))$ is an edge of 
${\cal C}_{v_{\iota(j)}} \cap {\cal C}_{v_0}.$}
\end{enumerate}
\noindent 2) If $n = k(j_1)-2$ then $v$ is of type $(i)$ and $w$ of type $(ii)$.
\vspace{-7pt}
   \begin{enumerate}[noitemsep, leftmargin=23pt]   
\item[(a)] If $j_1\notin \{\zeta^{\pm 1}(j),j\}$ then, by case  1)-(a) for $v$
  and case 2)-(b) for $w$ we obtain that:\\
\centerline{   $({\cal{A}}_{\varphi_{j}}(v),{\cal{A}}_{\varphi_{j}}(w))$ is an edge of 
${\cal C}_{v_{\iota(j)}}.$}

\item[(b)] If $j_1=j$, then by  case 1)-(b) for $v$ 
and case 2)-(a)
for $w$ we obtain that:\\
\centerline{ $({\cal{A}}_{\varphi_{j}}(v),{\cal{A}}_{\varphi_{j}}(w))$ is an edge of 
${\cal C}_{v_{\iota(j)}} \cap {\cal C}_{v_0}.$}

\item[(c)] If $j_1 = \zeta^{\pm 1}(j)$, then by  case 1)-(c)-(ii) the image of $v$ is of type $(ii)$
and by the case 2)-(a) the image of $w$ is of type $(i)$.
Hence, we obtain that:\\
 \centerline{   $({\cal{A}}_{\varphi_{j}}(v),{\cal{A}}_{\varphi_{j}}(w))$ is an edge of 
${\cal C}_{v_{\iota(j)}} \cap {\cal C}_{v_0}.$}
\end{enumerate}
\noindent For all the edges in ${\cal C}_{v_0}$, the map ${\cal{A}}_{\varphi_{j}}$ is well defined by
  $ {\cal{A}}_{\varphi_{j}}(v, w) :=({\cal{A}}_{\varphi_{j}}(v),{\cal{A}}_{\varphi_{j}}(w))$.

\noindent In particular no two edges are mapped to the same one. Therefore each 
 ${\cal{A}}_{\varphi_{j}}$ is a bijective isometry, when restricted to ${\cal C}_{v_0}$, for the combinatorial metric on $\Gamma_{\Phi}$.
 In addition, the map 
${\cal{A}}_{\varphi_{j}}$ increases or decreases by one the level of both vertices. 
The proof for the other compact sets  ${\cal C}_{v}$ is the same and thus the map is well defined on 
$\Gamma_{\Phi}$.
\end{proof}

\begin{Prop}\label{ActionOrdering} For every vertex $v$ of $\Gamma_{\Phi}$, 
 ${\cal{A}}_{\varphi_j}({\cal C}_{v}) = {\cal C}_{{\cal{A}}_{\varphi_j}(v)}$  
 and ${\cal A}_{\varphi_j}$ preserves the natural cyclic ordering of the edges given by Proposition \ref{cyclicorder}  around $v$. 
  \end{Prop}

\begin{proof} If $v = v_0$,
from the proof of Proposition \ref{LocIsom}: 
${\cal{A}}_{\varphi_{j}}({\cal C}_{v_0})={\cal C}_{v_{\iota(j)}}$ and, by Definition \ref{action} case {\it4}:
${\cal{A}}_{\varphi_{j}}(v_0) = v_{\iota(j)}$. For the other vertices the proof is the same.

  
 The cyclic ordering of Proposition \ref{cyclicorder} for the edges in 
$\Gamma_{\Phi}$ reflects the cyclic ordering of the intervals along the circle, it is given by the cyclic permutation $\zeta$. \\
Let us consider  $(v_0,v_k)$ and  $(v_0,v_{\zeta(k)})$  two consecutive edges around $v_0$. 
By Proposition \ref{LocIsom}, the image under ${\cal{A}}_{\varphi_{j}}$ depends on the value of $k$.\\
 If $k \neq j $ then 
 ${\cal{A}}_{\varphi_{j}}(v_0,v_k)=( {\cal{A}}_{\varphi_{j}}(v_0), {\cal{A}}_{\varphi_{j}}(v_k)) = (v_{\iota(j)},v_{\iota(j), k})$, and\\
${\cal{A}}_{\varphi_{j}}(v_0,v_{\zeta(k)}) = (v_{\iota(j)}, v_{{\iota(j),\zeta(k)}})$,
 these two edges are consecutive at the vertex $v_{\iota(j)}$.\\
  If $k = j $ then the image of the two edges are
 ${\cal{A}}_{\varphi_{j}}(v_0,v_j)=(v_{\iota (j)} , v_0 ) $ and\\
 ${\cal{A}}_{\varphi_{j}}(v_0,v_{\zeta(j)}) = (v_{\iota(j)}, v_{\iota(j),\zeta(j)})$,
 these two edges are consecutive around $v_{\iota(j)}$.\\
  Hence ${\cal{A}}_{\varphi_{j}}$ preserves the cyclic ordering of the edges around $v_0$. 

\noindent The proof for the other vertices is the same. 
 From  Proposition \ref{WellDef-Cv0}, Remark \ref{others}, Proposition \ref{PhiAction} and since each generator $\varphi_j$ is orientation preserving, the natural cyclic ordering at each vertex is preserved by the action. By composition, the same is true for each element in $G_{X_{\Phi}}$.
\end{proof}

\subsection{$G_{X_{\Phi}}$ is abstractly a surface group}

The length of an element $g \in G_{X_{\Phi}}$ is, as usual, the length of the shortest word expressing it in the generating set $X_{\Phi}$.

\begin{Prop}\label{groupEltDilatation}
Each element $g \in G_{X_{\Phi}}$ of length $n$  admits a non-trivial interval $J_g$ so that $g|_{J_g}$ is affine with slope $\lambda^n$.
 In addition, if $g$ has more than one expression of length $n$, then two expressions differ by some cutting point relations 
${\it(CPj)}$ of Theorem \ref{CP-relations}, for some $j \in \{1, \dots, 2N\}$.
\end{Prop}

\begin{proof} 
Let us consider the collection of integers given by (EC):
$ \{ k(j) : j \in \{ 1, \dots, 2N\} \}$, with
$K_0$ and $K_{\Phi}$ the minimal and maximal values of this set.

 We start the proof for the elements $g \in G_{X_{\Phi}}$ of length $n \leq K_{0}$,
 i.e., with an expression: 
$g=\varphi_{j_n}\circ \dots \circ \varphi_{j_1}$, satisfying, at least:
$\varphi_{j_{i+1}} \neq \varphi_{\overline{j_i}}$,  for $i=1, \dots, n-1$. 

\vspace{3pt}  
 (I) If $n=2 < K_0$: by condition (SE), the map $\Phi_{j_1}$ can be followed by  
any $\Phi_{k}$,  with $k \neq  \bar{j_1}$, for an iterate of length 2. This implies, from the definition of the generators in Definition \ref{groupGPhi}, that for each $j_2 \neq \bar{j_1}$, the element 
$g = \varphi_{j_2} \circ \varphi_{j_1}$ admits 
$J_g: = I_{j_1, j_2}$ as an interval where $g|_{ J_g}$ is affine with slope $\lambda^2$. 
This is the maximal possible slope for an element of length 2 in the group $G_{X_{\Phi}}$.
Since we are in a group, there cannot be more elements of length 2, starting with $\varphi_{j_1}$.

\vspace{5pt}  
(II) For $2<n < K_0$: we replace, in the above arguments, condition (SE) by the conditions (E-) and (E+) and we obtain that for all $n < K_0$ the element 
$g = \varphi_{j_n}\circ \dots \circ \varphi_{j_1}$ is of length $n$ with the only restriction that 
$\varphi_{j_{i+1}} \neq \varphi_{\overline{j_i}}$,  for all $i=1, \dots, n-1$.\\
On the graph $\Gamma_{\Phi}$, all the vertices $v$ in the interior of the ball $\textrm{ Ball} (v_0 , K_0) $ are of type $(i)$ and, on the corresponding interval $I_v = I_{j_1, \dots, j_n}$, the map 
 $g|_{I_{j_1, \dots, j_n}}$ is affine with slope $\lambda^n$, this is the maximal possible slope for an element of length $n$ and we choose $J_g := I_{j_1, \dots, j_n}$. 
 
 \vspace{5pt}  
(III) If $n = K_0$: let us consider an integer $j \in \{1, \dots, 2N \}$ so that $ k(j) = K_0$. 
 The element:
 $ g = \varphi_{\delta^{k(j)-1}(j) } \circ \dots\varphi_{\delta (j) } \circ \varphi_j$ has, at least, two expressions by Theorem \ref{CP-relations}.\\
  This element admits an interval $V_j$, given by Lemma \ref{affin-diffeo}, on which $g|_{V_j}$ is affine with slope $\lambda^n$. By definition of the generators in 
  Theorem \ref{CP-relations}, the interval $V_j$ might not be maximal with the property that the element is affine of slope $\lambda^n$.
  This interval is also denoted by $I_{{\tilde{v}}_j}$ in (\ref{Ivtilde}), and it is of type $(ii)$ by Lemma \ref{vertexInterval}. We choose in this case $J_g := I_{{\tilde{v}}_j}$.\\
By condition  (\ref{k(j)-image}) in the proof of Lemma \ref{affin-diffeo}, the two expressions of $g$ above can be followed by any $\varphi_{\alpha}$ for 
   $\alpha \notin \{\overline{\gamma^{k(j)-1}(\zeta^{-1}(j))}, \overline{\delta^{k(j)-1}(j)}\}$.\\
The two expressions of $g$, given by the cutting point relation ${\it(CPj)}$, have length $n$ and have $2N-2$ possible successors, i.e., elements of length $n+1$ with the same beginning, by condition (\ref{k(j)-image}). 
    The element $g$ cannot have more than two expressions, by a counting argument as in (I), this also proves that in cases (I) and (II) (when $n < K_0$) the expression is unique.
    The elements $g$ of length less than $ K_0$ are covered by one of the above cases (II) or (III). In all the cases the interval $J_g$ is chosen either as $I_{j_1,\dots, j_n}$, of type $(i)$ or 
   $I_{{\tilde{v}}_j}$,  of type $(ii)$. In addition, the element $g$ has either exactly one expression of length $n$ (case type $(i)$) or exactly two  (case type $(ii)$). 
  
   
   \vspace{5pt}  
 (IV) If $n>K_0$: For an element $ g = \varphi_{j_n} \circ \dots \circ \varphi_{j_1}$ of length $n > K_{0}$, the initial part of this expression of length $K_{0}$, i.e.,  
$g_1 = \varphi_{j_{K_{0}}} \circ \dots \circ \varphi_{j_1}$ is covered by the previous arguments. Thus there is an interval $J_{g_1}$ so that ${g_1}_{|J_{g_1}}$ is affine with slope 
 $\lambda ^{K_{0}}$ and two cases can occur: either $J_{g_1}$ is of type $(i)$ (resp. {\it(ii-$\mathscr{E}$))} or of type {\it(ii-$\mathscr{V}$)}.\\
 If $J_{g_1}$ is of type $(i)$ then, by the arguments in (II) above, $g_1$ has exactly $2N-1$ possible continuations of length $K_{0}+1$ and 
 $g_2 = \varphi_{j_{K_{0} +1}}\circ \varphi_{j_{K_{0}}} \circ \dots \circ \varphi_{j_1}$ is one of these continuations. The same argument applies to $g_2$ and we obtain an interval $J_{g_2} \subset J_{g_1} $.\\
 If $J_{g_1}$ is of type {\it(ii-$\mathscr{V}$)} then, by the argument in case (III), $g_1$ has exactly $2N-2$ possible continuations of length $K_{0}+1$ and 
 $g_2 = \varphi_{j_{K_{0} +1}}\circ \varphi_{j_{K_{0}}} \circ \dots \circ \varphi_{j_1}$ is one of these continuations. Again the same argument applies to $g_2$. In all these cases we obtain an interval $J_{g_2}$ so that 
 ${g_2}_{|_{J_{g_2}}}$ is affine with slope $\lambda ^{K_{0}+1}$. We complete this argument by induction.
 
 \vspace{5pt}
 At this point it remains to check the following:

      \noindent {\it Claim.}
 If two expressions of length $n$ define the same element in $G_{X_{\Phi}}$ then they differ by some cutting point relations \\
 {\it Proof of the claim:}    
     Notice that one property of the cutting point relations {\it(CPj)} has not been used yet:
     
 \noindent     - A cutting point relation, as a cyclic word in the generators, is such that for each consecutive letters 
     $\dots \varphi_j\circ \varphi_k \dots$, the indices $(j,\iota(k))$ 
      are adjacent, according to the permutation 
     $\zeta$, i.e., the intervals 
     $I_j$ and $I_{ \iota(k)}$ are adjacent along $S^1$.
     
\noindent     - In addition, all the adjacent pairs $(j,\iota(k))$ appear in the set of all cutting point relations.

\vspace{3pt}
    Suppose that  two expressions of length $n$: 
     $ A = \varphi_{j_n} \circ \dots \circ \varphi_{j_1}$ and $ B = \varphi_{k_n} \circ \dots \circ \varphi_{k_1}$ are equal in the group. If this equality represents a relation and is not a cutting point relation then, in the cyclic word $AB^{-1}$, some consecutive letters 
     $\varphi_a \circ \varphi_b$ are so that $(a, \iota(b))$ are not adjacent. Let us assume, for instance, 
     that the first indices $j_1$ and $k_1$ in $A$ and $B$ are not adjacent.
      This implies that the two intervals 
     $I_{j_1}$ and $I_{k_1}$ are not adjacent along $S^1$. Let $I_{j_1, \dots, j_n} \subset I_{j_1}$ (resp. 
     $I_{k_1, \dots, k_n} \subset I_{k_1}$) be the interval as above on which the element $A$ (resp. $B$) is affine of maximal slope $\lambda^n$. From the definition of the generators in Theorem \ref{CP-relations}, the maximal interval on which $A$ is affine of maximal slope is larger than $I_{j_1, \dots, j_n} \subset I_{j_1}$ but intersects, at most, a small subinterval of an adjacent interval $I_{\zeta^{\pm 1} (j_1)}$. Since $I_{j_1}$ and $I_{k_1}$ are not adjacent then the maximal intervals on which $A$ and $B$ are affine of maximal slope $\lambda^n$ are disjoint. Thus $A$ and $B$ cannot be equal, as homeomorphisms and thus in the group $G_{X_{\Phi}}$.

\textcolor{black}{     If $AB^{-1}$ is not a single relation then the equality is obtained as a concatenation of several relations. If all these relations are cutting point relations then we are done. If not then the above argument completes the proof. 
}
      \end{proof}

By combining the various results above we obtain:
\begin{Lemma}\label{GeometricAction} For all $g \in G_{X_{\Phi}}$, ${\cal A}_g: \Gamma_{\Phi} \rightarrow \Gamma_{\Phi}$ is a well defined morphism and
the map  ${\cal A}: G_{X_{\Phi}} \rightarrow {\rm Aut} (\Gamma_{\Phi}) $ defined by 
${\cal A}(g):={\cal A}_g$ is a 
geometric action of 
$G_{X_{\Phi}}$ on $\Gamma_{\Phi}$.
\end{Lemma}
\begin{proof}
Each map ${\cal{A}}_{\varphi_{j}}$ is a bijective isometry on the compact sets ${\cal C}_v$ by 
Proposition \ref{LocIsom} and  ${\cal{A}}_{\varphi_j}({\cal C}_{v})={\cal C}_{{\cal{A}}_{\varphi_j}(v)}$ 
for any ${\cal C}_v$ by Proposition \ref{ActionOrdering}. Therefore any composition: ${\cal{A}}_{\varphi_{j_n}}\circ \dots \circ {\cal{A}}_{\varphi_{j_1}}$ is an isometry. By definition, 
${\cal{A}}_{\varphi_{j_n}}\circ \dots \circ {\cal{A}}_{\varphi_{j_1}}={\cal{A}}_{\varphi_{j_n}\circ \dots \circ \varphi_{j_1}}$. We have to check that this map does not depend on the expression of the group element $g=\varphi_{j_n}\circ \dots \circ \varphi_{j_1}$, i.e., the map is a well defined morphism.

By Proposition \ref{groupEltDilatation}, the set of relations in $G_{X_{\Phi}}$ for the generating set $X_{\Phi}$ are:
\vspace{-8pt}
   \begin{enumerate}[noitemsep, leftmargin=17pt]   
 \item[1)] The trivial relations:  $ \varphi_{j} \circ\varphi_{\iota(j)} = {\rm id}_{G_{X_{\Phi}}}$, or
\item[2)]  the cutting point relations ${\it(CPj)}$ of Theorem \ref{CP-relations}, for $j=1,\dots,2N$.
\end{enumerate}

 \vspace{-5pt}
\noindent We will show that the map ${\cal{A}}$ respects these relations and, by Proposition \ref{PhiAction}, it is sufficient to check it on the compact set ${\cal C}_{v_0}$.

\vspace{3pt}
\noindent 1) For the trivial relations:
 by Definition \ref{action} we have ${\cal{A}}_{\varphi_{j}}(v_0)=v_{\iota(j)}$ and ${\cal{A}}_{\varphi_{\iota(j)}}(v_{\iota(j)})=v_0.$ For the other vertices $v\neq v_0$ in ${\cal C}_{v_0}$ we have either $v=v_{j_1,\gamma(j_1), \dots, \gamma^n(j_1)}$ or  
  $v=v_{j_1,\delta(j_1), \dots, \delta^n(j_1)}$, for  $n\leq  k(j_1)-1$. The proof follows from the case by case study in the proofs of  Proposition \ref{WellDef-Cv0} and Proposition \ref{LocIsom}, we obtain:
 \textcolor{black}{${\cal{A}}_{\varphi_{\iota(j)}} \circ {\cal{A}}_{\varphi_{j}}$} 
is the identity on ${\cal C}_{v_0}$ and thus on $\Gamma_{\Phi}$.

\vspace{3pt}
\noindent  2) For the cutting point relations ${\it CPj}$:
They  are related to several properties of the map 
$\Phi$ and the space $\Gamma_{\Phi}$. Each ${\it CPj}$ is given by a cutting point of the map and to the equivalence relation of Definition \ref{graph} via the notion of vertices and intervals of type {\it(ii-$\mathscr{V}$)} according to Lemma \ref{vertexInterval}. 
The cutting point relations are also associated with the ``loops", based at any vertex $v$ by Corollary \ref{loops}. Recall that the compact sets ${\cal{C}}_v$ are defined in Remark \ref{compact} as the union of all the loops, based at $v$.
By Proposition \ref{ActionOrdering}, 
${\cal{A}}_{\varphi_j}({\cal C}_{v})={\cal C}_{{\cal{A}}_{\varphi_j}(v)}$  and
 ${\cal A}_{\varphi_j}$ is a bijective isometry by Proposition \ref{LocIsom}. This implies, in particular, that each loop, based at $v$ is mapped to a loop, based at ${\cal A}_{\varphi_j} (v)$, for all $j$ and all $v$. 
 Thus the map ${\cal A}$ respects all the cutting point relations.


 By the Propositions \ref{PhiAction} and \ref{ActionOrdering}, the map ${\cal{A}}$
  is co-compact and thus ${\cal{A}}$ is a well defined, co-compact isometric morphism.
 
 It remains to check that ${\cal{A}}$ is properly discontinuous. The graph $\Gamma_{\Phi}$ is locally compact so a compact set in $\Gamma_{\Phi}$ is contained in a ball of finite radius. If $C_1$ and $C_2$ are two compact sets in $\Gamma_{\Phi}$ we can assume that $C_1$ is contained in a ball of radius $R$ centred at $v_0$. By Proposition \ref{PhiAction} there are elements $g \in G_{X_{\Phi}}$ so that 
 ${\cal{A}}_{g} (C_2) \cap C_1 \neq \emptyset$. These elements have a length, with respect to the generating set $X_{\Phi}$. This length is bounded, by at most twice the distance in $\Gamma_{\Phi}$, between $C_1$ and $C_2$, as in the proof of Proposition \ref{PhiAction}.
 Thus the set $\{  g \in G_{X_{\Phi}}: {\cal{A}}_{g} (C_2) \cap C_1 \neq \emptyset \}$ is finite and the action is properly discontinuous.  Therefore the map ${\cal A}$ is a geometric action.
\end{proof}
As a consequence of the above properties we obtain the following result:
\begin{theo}\label{*}
Let $\Phi$ be a piecewise orientation preserving homeomorphism on the circle satisfying the conditions: {\rm(E$\pm$), (EC), (CS-$\lambda$)} for some $\lambda > 1$.\\
Let $G_{\Phi}:=G_{X_\Phi}$  be the sub-group of $\rm{Homeo}^+ (S^1)$ given in Definition \ref{groupGPhi} then:
\vspace{-9pt}
   \begin{enumerate}[noitemsep, leftmargin=22pt]   
\item[$(1)$] The group  $G_{\Phi}$ is discrete.
\item[$(2)$] The group  $G_{\Phi}$ is Gromov-hyperbolic with boundary $S^1$.
\item[$(3)$] $G_{\Phi}$ \textcolor{black}{ is conjugate in \textcolor{black}{${\rm Homeo} (S^1)$} to the restriction of a torsion free Fuchsian group action on $S^1$.}
\item[$(4)$] \textcolor{black}{ The number $\lambda$ is an algebraic integer.}
\end{enumerate}
\end{theo}
   \begin{proof}

(1) The group acts geometrically on a discrete metric space by Lemma \ref{GeometricAction} so it is a discrete group. Recall that the graph $\Gamma_{\Phi}$ and the action of Definition \ref{action} depends only on the map $\Phi$.

\noindent (2) By Lemma \ref{GeometricAction} and Corollary \ref{Graph1Hyper} the group acts geometrically on a Gromov hyperbolic space with boundary $S^1$. Therefore the group is Gromov hyperbolic with boundary $S^1$ by the Milnor-Swartz Lemma 
 (see for instance \S 3 in \cite{GdlH}).
 
\noindent (3) The group is a convergence group by a result of E. Freden \cite{F}. Therefore the conditions of \cite{G},  \cite{Tukia} and  \cite{CJ} are satisfied and the group $G_{\Phi}$ 
\textcolor{black}{ is conjugate in \textcolor{black}{$\textrm{Homeo} (S^1)$} to the restriction of a Fuchsian group action on $S^1$.}

In order to complete the proof of (3) it suffices to check:

\noindent {\it Claim.} The group $G_{\Phi}$ is torsion free.

\noindent {\it Proof of the Claim.}
We already observed that each $g \in G_{\Phi}$ has bounded expansion and contraction factors by Proposition \ref{groupEltDilatation}.
This implies, in particular, that 
 each element $g \in G_{\Phi}$ admits an interval $I_v$ on which $g$ is affine of slope $\lambda^n$,  where $n$ is the length of the element.
This property implies that $ g^m \neq id$ for all $g \in G_{\Phi}-\{id\}$ and all $m$.

 By \cite {Zi} a Fuchsian group that is torsion free is a surface group. So the group 
$G_{\Phi}$ is conjugate to the restriction of a Fuchsian surface group action on $S^1$.

\noindent \textcolor{black}{ (4) The number $\lambda$ is an algebraic integer by Lemma \ref{algebraic}.}
\end{proof}
Theorem \ref{*} admits several interesting consequences. 
 The following one is direct  and  surprising.

\begin{Cor}\label{growth}
Let $S$ be a compact, closed, orientable hyperbolic surface.
\textcolor{black}{There is a discrete faithfull representation $\rho : \pi_1 (S) \rightarrow {\rm Homeo}^+ (S^1)$, a metric 
$\mu$ on $S^1 = \partial( \pi_1 (S))$ and a set of generators $X$ of $\rho(\pi_1 (S))$ so that:}\\
$\bullet$ Each $g \in \rho(\pi_1 (S))$ is piecewise affine, with respect to $\mu$, with slopes in 
$\{\lambda^k , k \in \mathbb{Z}\}$ and $\lambda$ is an algebraic integer.\\
$\bullet$ If $g \in \rho(\pi_1 (S))$ has length $n$ with respect to $X$, it admits an interval $U_g \subset S^1$ so that $g_{| U_{g}}$ is affine of slope $\lambda^n$.
\end{Cor}

\textcolor{black}{ This result is direct from Lemma \ref{condSatisfied}, Theorem \ref{*}, Definition \ref{groupGPhi}, Proposition \ref{groupEltDilatation} and Lemma \ref{algebraic}. It is surprising at several levels. For instance, as a hyperbolic group, the Gromov boundary admits many classes of metrics. We obtain here a very particular and rigid metric that reflects the growth property of the group presentation
(see Corollary \ref{VolumeEntropy} in \S 7) via the factor $\lambda$. The presentation we obtain, with piecewise affine elements is also surprising. For instance the fact that the length of an element is directly computable.}

We will see other consequences in the Appendix in \S 7.

\section{Orbit equivalence}

In this section we complete the proof of the main theorem: the group and the map are orbit equivalent. Let us recall the definition of orbit equivalence, as given in \cite{BS}.
\begin{definition}
A map $\Phi : S^1 \rightarrow S^1$ and a group $G$ acting on $S^1$ are orbit equivalent if, except for a finite number of pairs of points $(x, y) \in S^1 \times S^1$:\\
$\exists g \in  G$ so that $y = g(x)$ if and only if 
$\exists (m, n) \in \mathbb{N}\times \mathbb{N}$ so that 
$ \Phi^n (x) = \Phi^m (y)$.
\end{definition}

The following result is the first statement of the main Theorem.
\begin{theo}\label{OE}
If $\Phi : S^1 \rightarrow S^1$ is an orientation preserving piecewise homeomorphism  satisfying the conditions {\rm (EC), (E$\pm $)} and  {\rm(CS-$\lambda$) for some $\lambda > 1$}, then the group $G_{X_{\Phi}}$ of Theorem \ref{groupGPhi} and the map
$\widetilde{\Phi}$, conjugated to $\Phi$ by {\rm(CS)}, are orbit equivalent.
\end{theo}
 \begin{proof} The proof uses the piecewise affine map 
 $\widetilde{\Phi}$ conjugate to $\Phi$ by condition  (CS-$\lambda$) via some $g \in \textrm{Homeo}^+(S^1)$. The orbit equivalence is preserved by conjugacy and the above statement is valid for the map $\Phi$ and the group obtained from \textcolor{black}{$G_{X_{\Phi}}$} by conjugacy via the \textcolor{black}{ element $g\in \textrm{Homeo}^+(S^1)$ given by condition (CS)}.\\
 One direction of the orbit equivalence is direct from the definition of the map and the group.\\
$\bullet$ If $ \widetilde{\Phi}^n (x) = \widetilde{\Phi}^m (y)$ then there are two sequences of integers 
 $ \{ j_1, \dots , j_n  \}$ and $\{  l_1, \dots, l_m \}$ such that:
$ \varphi_{j_n} \circ \dots \circ \varphi_{j_1} ( x) =  
 \varphi_{l_m} \circ \dots \circ \varphi_{l_1} ( y).$\\
  This implies that $ y = g(x)$ for 
  $g = (  \varphi_{l_m} \circ \dots \circ \varphi_{l_1}) ^{-1} \circ
   \varphi_{j_n} \circ \dots \circ \varphi_{j_1}  \in \textcolor{black}{G_{X_{\Phi}} }$.
   
\noindent $\bullet$   For the other direction we assume $y = h(x) $ and, since 
   $ X_{\Phi} = \{ \varphi_1, \dots , \varphi_{2N} \}$ is generating $G_{X_{\Phi}}$, it is sufficient to restrict to 
   $ h = \varphi_j \in X_{\Phi}$.

  Recall that each generator $\varphi_j \in X_{\Phi}$ of Definition \ref{LimitHomeo} and Definition
  \ref{groupGPhi} is piecewise affine with two special points, the breaking points $\{l_j^0, r_j^0 \}$, that are periodic under $\widetilde{\Phi} $.
    By construction, each interval of the partition satisfies:\\
\centerline{\textcolor{black}{  $ \widetilde{I}_j = [\widetilde{z}_j , \widetilde{z}_{\zeta (j)} ) \subset (  l_j^0, r_j^0 ) $
  with 
 $\varphi_{\iota (j) } ( l_{\iota (j)}^0  ) = r_j^0 \textrm{ and }  
\varphi_{\iota (j) } ( r_{\iota (j)}^0 ) = l_j^0$.}}

\noindent Let us assume \textcolor{black}{$x \notin \{  l_j^0, r_j^0 \} $  then}:

$ \varphi_j $ is either expanding or contracting at $x$, i.e., with slope $\lambda$ or $\lambda^{-1}$.\\
In the second case $x = \varphi_j ^{-1} (y) $ and 
   $ \varphi_j ^{-1}$ is expanding at $y$.
      By this symmetry, we assume that $ \varphi_j $ is expanding at $x$ and thus $ x \in ( l_j^0 , r_j^0) $.
   Two cases can arise:
   
   \centerline{$  \textrm{(a) }  x \in \widetilde{I}_j \hspace{0.5cm} \textrm{      or }   \hspace{0,5cm}  
   \textrm{ (b) }  x \in ( l_j^0 , r_j^0) \setminus \widetilde{I}_j.$}
   
\noindent  In case (a): $ \varphi_j (x) = \widetilde{\Phi} (x)$, thus $ y = \widetilde{\Phi} (x)$ and 
$ (x, y)$ are in the same $\widetilde{\Phi}$-orbit.
       
\noindent   In case (b): there is another symmetry:\\
\centerline{ $ x \in (l_j^0 , \widetilde{z}_j ) $ or $ x \in ( \widetilde{z}_{\zeta (j)} , r_j^0 ) $,
   we assume that $ x \in (l_j^0 , \widetilde{z}_j ) $.}
   
\noindent   By Corollary \ref{PeriodicPoints} and  Definition \ref{LimitHomeo}, the breaking point satisfies: 
$l_j^0 \in ( \varphi_{\iota (j) } ( \widetilde{z}_{\delta (j) } ) , \widetilde{z}_j ) =  L_{j}$.
 By Lemma \ref{action1}, condition $(c)$, we obtain :

   
   
\centerline{   $l_j^0 \in L_{j} \subset 
    I_{ \zeta^{-1} (j), \gamma (  \zeta^{-1} (j) ) , \dots , \gamma^{k(j) - 2} (  \zeta^{-1} (j) )} \textrm{ and, by symmetry, }   r_j^0 \in R_{j}.$}
   
\noindent  The definition of these intervals implies:\\ 
    \centerline{$ \forall u \in I_{ \zeta^{-1} (j), \gamma (  \zeta^{-1} (j) ) , \dots , 
    \gamma^{k(j) - 2} (  \zeta^{-1} (j))}:
     \widetilde{\Phi}^i (u) \in \widetilde{I}_{\gamma^{i} (  \zeta^{-1} (j) )}, \forall i \in
      \{ 0, \dots, k(j) - 2 \},$}
    and thus condition (b) implies:
    
\centerline{   $\widetilde{\Phi}^i (x) \in \widetilde{I}_{\gamma^{i} (  \zeta^{-1} (j) )}, \forall i \in \{ 0, \dots, k(j) - 2 \}. $}  

\noindent     With the same argument we obtain: \\
     $y = \varphi_j (x) \in ( r_{\iota (j)}^0 , \varphi_{j} (\widetilde{z}_j) ) \subset R_{\iota (j)}$
     with $ R_{\iota (j)} \subset I_{\delta (j), \dots , \delta^{k(j) - 1}(j) }$, and thus:\\     
   $ \widetilde{\Phi}^i ( y ) \in \widetilde{I}_{\delta^{i+1} (  j )}, \forall i \in \{ 0, \dots, k(j) - 2 \}.$
    Hence, the $\widetilde{\Phi}$ orbits of $x$ and $y$ satisfy:
     \vspace{-5pt}
    \begin{equation}\label{(6)}
    \begin{array}{l}
 \widetilde{\Phi}^{k(j) -1} (y) = \varphi_{ \delta^{k(j) -1} (j)} \circ \dots \circ 
    \varphi_{ \delta ( j) } (y) \textrm{, and}\\
\widetilde{\Phi}^{k(j) -1} (x) = \varphi_{ \gamma^{k(j) - 2} (  \zeta^{-1} (j) ) } \circ  \dots \circ \varphi_{ \gamma (  \zeta^{-1} (j) ) } \circ \varphi_{ \zeta^{-1} (j)} (x) .
\end{array}
\end{equation}

 \vspace{-6pt}
\noindent    Recall that each cutting point $z_j$ defines the following  relation ${\it(CPj)}$ in the group $G_{X_{\Phi}}$:
    
 \centerline{ $ \varphi_{ \delta^{k(j) -1} (j)} \circ  \dots \circ \varphi_{ \delta ( j) }\circ \varphi_{ j} =
     \varphi_{ \gamma^{k(j) - 1} (  \zeta^{-1} (j) ) } \circ  \dots\circ \varphi_{ \gamma (  \zeta^{-1} (j) ) } \circ \varphi_{ \zeta^{-1} (j)}$ .}
     
 \noindent    If the relation ${\it(CPj)}$ is applied to the point $x$ we obtain:
 \vspace{-7pt}
     \begin{equation}\label{(9)}
  \widetilde{\Phi}^{k(j) - 1} (\varphi_j(x)) = \varphi_{ \gamma^{k(j) - 1} (  \zeta^{-1} (j) ) } [\widetilde{\Phi}^{k(j) - 1} (x) ].
  \end{equation}
  
   \vspace{-5pt}
     \noindent     Indeed, by replacing $y = \varphi_j (x)$ in the left hand side of the relation we obtain the first  equality in (\ref{(6)}) which is the left hand side of (\ref{(9)}). The right hand side of (\ref{(9)}) is obtained by replacing, in the right hand side of the  relation, the second equality in  (\ref{(6)}).
        Let us denote:\\
   \centerline{ $j_1 :=  \gamma^{k(j) - 1} (  \zeta^{-1} (j) )  \in \{ 1, \dots , 2N \}$, \; $x_1:=\widetilde{\Phi}^{k(j) - 1} (x)$ and 
   $\;y_1:=\varphi_{j_1}(x_1)$.}
   Observe that the index $j_1 = \alpha (j)$, as defined in Lemma \ref{alpha-beta}.\\

   \noindent The equality (\ref{(9)}) implies that an alternative, similar to (a) or (b) above,  applies again, more precisely:
  
  \vspace{-7pt}  
\centerline{   $ (a_1) \quad   x_1 \in \widetilde{I}_{ j_1} \hspace{0.5cm}
   \textrm{ or } \hspace{0.5cm} (b_1) \quad  x_1 \notin \widetilde{I}_{ j_1}.$}
 
 \vspace{5pt}
 \noindent   In case $(a_1):$ the equality (\ref{(9)}) gives:

 \vspace{5pt} 
\centerline{ $\widetilde{\Phi}^{k(j) - 1} (y) =  \widetilde{\Phi} [ \widetilde{\Phi}^{k(j) - 1} (x) ] =  \widetilde{\Phi}^{k(j) } (x) ,$}
 \noindent and the orbit equivalence is proved in this case.
 
  \vspace{3pt}
 \noindent In case $(b_1):$  We obtain that $x_1 \in (l _{j_1}^0, \widetilde{z}_{ j_1} )$, 
 \textcolor{black}{if $ x \in (l_j^0 , \widetilde{z}_j ) $.\\
   Indeed, since $x_1=\widetilde{\Phi}^{k(j) - 1} (x)$
 then $x_1 \in \widetilde{\Phi}^{k(j) - 1} (l_j^0 , \widetilde{z}_j )$. In addition
 $\widetilde{\Phi}^{k(j) - 1} (\widetilde{z}_j) \in \widetilde{I}_{ j_1}$ by condition (E-) and 
 $\widetilde{\Phi}^{k(j) - 1} (l _{j}^0) = l _{j_1}^0$ by Corollary \ref{PeriodicPoints}. Thus if $x_1 \notin \widetilde{I}_{ j_1}$ then $x_1 \in (l _{j_1}^0, \widetilde{z}_{ j_1} )$.}
 
 

\noindent This  alternative $(b_1)$ is thus exactly the same at the point $x_1$ than (b) was at the point $x$. This implies, in particular that: $ x_1 \in \widetilde{I}_{\zeta^{-1} (j_1)}$ and, more precisely:\\
 \centerline{ $ x_1 \in I_{\zeta^{-1} (j_1), \gamma (\zeta^{-1} (j_1) ), \dots ,  
 \gamma^{k(j_1) - 2} (\zeta^{-1} (j_1) )} \textrm{ and } 
   y_1 = \varphi_{j_1} (x_1) \in 
 I_{\delta(j_1), \dots,  \delta^{k(j_1) -1}(j_1) } ,$}
 
\noindent by the same arguments as for the points $x$ and $y$.\\
 Therefore we obtain:\\ 
 $\bullet$  a sequence of integers: $\{ j=j_0, j_1, \dots, j_n , \dots \}$ where each 
 $ j_m \in \{ 1, \dots , 2N \} $,\\
 $\bullet$ a sequence of points: $ x_n := \widetilde{\Phi}^{k(j_{n - 1}) - 1} (x_{n-1} ) $ and 
 $y_n: = \varphi_{j_n} (x_n )$, with the following alternative: 
 
  \vspace{-5pt}
  \centerline{ $(a_n) \quad  x_n \in \widetilde{I}_{j_n }  \hspace{0.5cm} \textrm{ or }  \hspace{0.5cm} (b_n) \quad   x_n \notin \widetilde{I}_{j_n }$ \textcolor{black}{and thus $x_n \in (l _{j_n}^0, \widetilde{z}_{ j_n} )$. }}
  
 \begin{Lemma}\label{expandingBn0} 
 With the above notations, if $ x \in (l_j^0 , \widetilde{z}_j ) $ and $y = \varphi_j (x)$, then there   \textcolor{black}{exist integers 
 $n_0 \geq 1$
and} $K(n_0) \geq k(j) -1$ so that $x_{n_0} \in \widetilde{I}_{j_{n_0} }$ and 
  $ \widetilde{\Phi}^{K(n_0)} (y ) = \widetilde{\Phi}^{K(n_0) +1} (x )$.\\
\textcolor{black}{ If $ x \in ( \widetilde{z}_{\zeta (j)} , r_j^0 )$  and $y = \varphi_j (x)$, then there exist integers $n'_0 \geq 1$ and $K(n'_0) \geq k(j) -1$ so that $x_{n'_0} \in \widetilde{I}_{j_{n'_0} }$ and 
  $ \widetilde{\Phi}^{K(n'_0)} (y ) = \widetilde{\Phi}^{K(n'_0) +1} (x )$.}
    \end{Lemma}

  \begin{proof}
  Recall that $S^1$ has a well defined metric $|\quad|$ for which $\widetilde{\Phi}$ is affine.\\
   Since 
  $x \in (l_j^0, \widetilde{z}_j )$, we define $\rho_{x,j}:= | (l_j^0, x ) | > 0$. By Corollary \ref{PeriodicPoints}, the breaking point $l_j^0$ is an expanding periodic point of period $K(r_{\alpha} (j)) \geq k(j) -1$.\\
  - If the alternative $(a_1)$ is satisfied then $n_0 = 1$ and $K(n_0) = k(j) -1$, i.e.,\\
  $ \widetilde{\Phi}^{k(j) -1} (y ) = \widetilde{\Phi}^{k(j)} (x )$ and the points $x$, $y$ are in the same 
  $ \widetilde{\Phi}$-orbit.\\
  - If $(b_1)$ is safisfied then $x_1 = \widetilde{\Phi}^{k(j)-1} (x ) \in (l_{j_1}^0, \widetilde{z}_{j_1} )$.
  The breaking point $l_{j_1}^0$ belongs to the same $\widetilde{\Phi}$-orbit as $l_{j}^0$ and we obtain:
  $\rho_{x_1,j_1} = | (l_{j_1}^0, x_1 ) | = \lambda^{k(j) -1} \cdot \rho_{x,j}$, since $\widetilde{\Phi}$ is affine.
  
  Let $M = {\textrm{max}}_{i = 1, \dots, 2N} | (l_i^0, \widetilde{z}_i ) |$, since the periodic point  $l_{j_1}^0$ is expanding and $x_n = \widetilde{\Phi}^{k(j_{n-1}) -1} (x_{n-1} )$ then there is $n_1 \geq 1$ so that:
$\rho_{x_{n_1},j_{n_1}} = \lambda^{K(n_1)} . \rho_{x,j} > M$, where
$K(n) = \sum_{i=0}^{n} ( k(j_i) -1 )$. 
Therefore there is $n_0 \leq n_1$  so that $x_{n_0} \in \widetilde{I}_{j_{n_0} }$. This is the first statement of the Lemma and thus the alternative $(a_{n_0})$ is satisfied which implies:\\
 $ \widetilde{\Phi}^{K(n_0)} (y ) = \widetilde{\Phi}^{K(n_0) +1} (x )$, thus $x$ and $y$ belong to the same
 $\widetilde{\Phi}$-orbit. The second case is obtained by symmetry.
  \end{proof}

 \noindent  This completes the proof of Theorem \ref{OE} and of the main Theorem.
\end{proof}


\section{Appendix}
In this Appendix we give a direct proof of:
\begin{THM}
The group $G_{\Phi}$ of Definition \ref{groupGPhi} \textcolor{black}{ is conjugate in $\textcolor{black}{{\rm Homeo} (S^1)}$ to the restriction of a torsion free Fuchsian group action on $S^1$. It is abstractly a surface group.}
\end{THM}
\noindent This result has been obtained in Theorem \ref{*} of section 5, by using the very strong geometrisation theorem of Tukia \cite{Tukia}, Gabai \cite{G} and Casson-Jungreis \cite{CJ}. 
The proofs of this geometrisation theorem, in one way or another, rely on extending the group action on the circle to an action on a disc.
 Our approach is not an exception to this general strategy.  One way to interpret this stategy is to prove that the group is abstractly a \textcolor{black}
 {Fuchsian} group.
  In our case we already have an important ingredient: a geometric action given by Definition \ref{action} on the hyperbolic metric graph $\Gamma_{\Phi}$ of Definition \ref{graph}.
This graph $\Gamma_{\Phi}$ satisfies particular properties, for instance it admits a cyclic ordering at each vertex that is preserved by the action by Propositions \ref{cyclicorder} and \ref{ActionOrdering}.

We need to prove that $\Gamma_{\Phi}$ can be embedded in a plane and the action can be extended to a planar action.

We define a 2-complex $\Gamma_{\Phi}^{(2)}$, in analogy with the Cayley 2-complex:

\noindent $\bullet$ For each closed path in $\Gamma_{\Phi}$, associated to a cutting point relation 
 ${\it(CPj)}$ by Corollary \ref{loops} (see Figure \ref{Cv0}), we define a two disc $\Delta_{z_j}$ whose boundary is a polygon with $2 \cdot k(j)$ sides, where $k(j)$ is given by condition (EC) at $z_j$.

\noindent $\bullet$  We glue ``isometrically" a disc $\Delta_{z_j}$ along a closed path in $\Gamma_{\Phi}$, as above, associated to ${\it(CPj)}$. Isometrically means that each side of $\Delta_{z_j}$ has length one and is glued along the corresponding edge in 
$\Gamma_{\Phi}$, also of length one.
\noindent We denote $\Gamma_{\Phi}^{(2)}$ the 2-complex obtained by gluing all possible such discs. The graph $\Gamma_{\Phi}$ is naturally the 1-skeleton of $\Gamma_{\Phi}^{(2)}$.

\begin{Lemma}\label{Plane}

 The 2-complex $\Gamma_{\Phi}^{(2)}$ is homeomorphic to $\mathbb{R}^{2}$.\\
 The action ${\cal {A}}_{g}$, $ g\in G_{\Phi}$ extends to a free, co-compact, properly discontinuous action $\widetilde{\cal {A}}_{g}$ of $G_{\Phi}$ on $\Gamma_{\Phi}^{(2)}$.

\end{Lemma}

\begin{proof} By the Propositions \ref{LocIsom} and \ref{ActionOrdering}, the action  ${\cal {A}}_{g}$ maps the link at a vertex 
$v \in V (\Gamma_{\Phi})$ to the link at $ w = {\cal {A}}_{g} (v) $  and 
this action preserves the cyclic ordering of Proposition \ref{cyclicorder}.  This implies, in particular, that adjacent edges at $v$ are mapped to adjacent edges at $w$.
Recall that adjacent edges define a relation ${\it(CPj)}$ by Corollary \ref{loops} for some 
$j \in \{ 1,\dots, 2N\}$. Therefore a closed path 
$\Pi^0$, based at $v$ in 
$ \Gamma_{\Phi}$ associated to a relation ${\it(CPj)}$ is mapped to a closed path 
$\widetilde{\Pi^0} $, based at $w$, associated to ${\it(CPj)}$. We extend the action 
${\cal {A}}_{g}$ on $ \Gamma_{\Phi}$ to an action $\widetilde{\cal {A}}_{g}$ on 
$\Gamma_{\Phi}^{(2)}$ by declaring that if 
${\cal {A}}_{g} (\Pi^0) = \widetilde{\Pi^0} $ then the disc $\Delta_{z_j}$ based at $v$ is mapped by $\widetilde{\cal {A}}_{g}$ to the disc $\widetilde{\Delta_{z_j}} $ based at $w$.

The set of 2-cells $\Delta_{z_j}$ for all $j \in \{1, \dots, 2N \}$, glued along each pair of adjacent edges at $v$ in $\Gamma_{\Phi}^{(2)}$, defines a neighborhood of $v$ in $\Gamma_{\Phi}^{(2)}$. This neighborhood is a 2-disc. Indeed, by the natural cyclic ordering of the edges at $v$, exactly two 2-cells are glued along an edge. Observe that the boundary of this neighborhood is a subset of the graph $\Gamma_{\Phi}$ which is precisely the boundary of the compact set
${\cal{C}}_v$ of Remark \ref{compact}.
This 2-disc is embedded in $\mathbb{R}^{2}$ and this property is true for each vertex. Thus 
$\Gamma_{\Phi}^{(2)}$ is homeomorphic to $\mathbb{R}^{2}$, since each point has a neighborhood homeomorphic to a 2-disc and 
$\Gamma_{\Phi}^{(2)}$ is contractible since any basic loop in $\Gamma_{\Phi}$ bounds a unique disc in 
$\Gamma_{\Phi}^{(2)}$.\\
The extended action $\widetilde{\cal {A}}_{g}$ defined above is co-compact, free, and properly discontinuous, exactly as the action ${\cal {A}}_{g}$ is on  $\Gamma_{\Phi}$.
\end{proof}
\noindent {\em Proof of the  Theorem}. The quotient of $\Gamma_{\Phi}^{(2)}$ by the action 
$\widetilde{\cal {A}}_{g}$ is a compact surface since $\Gamma_{\Phi}^{(2)}$ is homeomorphic to 
$\mathbb{R}^{2}$ and the action is co-compact and free. \textcolor{black}{The group $G_{X_{\Phi}}$ is abstractly a surface group and is thus conjugate in \textcolor{black}{$\textrm{Homeo} (S^1)$} to a torsion free Fuchsian group action on $S^1$.}
$\square$\\

As another consequence we obtain:
\begin{Cor}\label{Pgeom}
The group $G_{X_{\Phi}}$ of Definition \ref{groupGPhi} admits a presentation $P$ where the generating set is $X_{\Phi}$ and the set of relations are the cutting point relations ${\it(CPj)}$. Then the Cayley 2-complex of this presentation is homeomorphic to 
$\mathbb{R}^2$.
\end{Cor}

\begin{proof}
The group is generated by $X_{\Phi}$ by Definition \ref{groupGPhi} and, from Proposition \ref{groupEltDilatation}, the set of relations is the set ${\it(CPj)}$. From the proof of Lemma \ref{Plane}, the 2-complex 
$\Gamma_{\Phi}^{(2)}$ is identified with the Cayley 2-complex of the presentation $P$. Indeed, each vertex in 
$\Gamma_{\Phi}$ is associated to a group element, each edge is associated to a generator and each 2-disc 
in $\Gamma_{\Phi}^{(2)}$ is bounded by a loop associated to a relation.
\end{proof}

A presentation of a surface group $G = \pi_1 (S)$ is called {\em geometric} in \cite{Los} (see also  \cite{AJLM2}) if the Cayley 2-complex is planar. From the main Theorem  \ref{*} and Corollary \ref{Pgeom}, this is the case for the group $G_{\Phi}$ and the presentation $P$. 
 Recall that the volume entropy  $h_{vol} (G,P)$ of a group $G$ with presentation $P$ is the logarithm of the exponential growth rate of the number of elements of length $n$ with respect to the presentation $P$ (see \cite{GH}).
 As another consequence of this work we obtain:

\begin{Cor}\label{VolumeEntropy}
If $\Phi : S^1 \rightarrow S^1$ is an orientation preserving piecewise homeomorphism satisfying {\rm(EC), (E$\pm$)} and  {\rm(CS-$\lambda$}), for some $\lambda > 1$, then $\lambda$ is an algebraic integer and 
$log (\lambda ) = h_{vol} (G_{\Phi},P)$, for the group $G_{\Phi}$ of Definition \ref{groupGPhi}, with the presentation $P$ given by  Corollary
\ref{Pgeom}. The quantity $log (\lambda )$ is also the topological entropy of the map $\Phi$.
\end{Cor}

This result is an immediate consequence of Theorem \ref{*} together with the main result in \cite{AJLM2} since the presentation $P$ is geometric by Corollary \ref{Pgeom}. The fact that $\lambda > 1$ is an algebraic integer has already been obtained directly before, in Lemma \ref{algebraic}, in this Lemma it was noticed that 
$log (\lambda )$ is the topological entropy of the map $\Phi$.\\

The two Corollaries \ref{VolumeEntropy} and \ref{growth} together give the second Theorem of the introduction.


J\'er\^ome Los \\
I2M (Institut  Math\'ematique de Marseille),  Aix-Marseille Universit\'e,\\
CNRS UMR 7373,
 3 Place Victor Hugo, 13003 Marseille, France\\
{\rm jerome.los@univ-amu.fr}\\


Natalia A. Viana Bedoya\\
Departamento de Matem\'atica, 
Universidade Federal  de S\~ao Carlos\\
Rod. Washington Luis, Km. 235. C.P 676 - 13565-905 S\~ao Carlos, SP - Brasil\\
{\rm nvbedoya@ufscar.br}
 \end{document}